\documentclass{amsart}

\usepackage{amsmath,amssymb,amsthm,amsfonts,mathrsfs}
\usepackage[frame,cmtip,arrow,matrix,line,graph,curve]{xy}
\usepackage{graphpap,color,paralist,pstricks}
\usepackage[mathscr]{eucal}
\usepackage[pdftex,colorlinks,backref=page,citecolor=blue,linkcolor=black]{hyperref}
\usepackage{tikz}
\usepackage{epic,eepic,url}
\usepackage{yfonts}
\usepackage{enumerate}
\usepackage{mathtools}
\usepackage{mathabx}
\def\multiset#1#2{\ensuremath{\left(\kern-.3em\left(\genfrac{}{}{0pt}{}{#1}{#2}\right)\kern-.3em\right)}}
\usepackage{verbatim}

\theoremstyle{definition}
\newtheorem{theorem}{Theorem}[section]
\newtheorem{definition}[theorem]{Definition}
\newtheorem{theorem-definition}[theorem]{Theorem-Definition}
\newtheorem{conjecture}[theorem]{Conjecture}
\newtheorem{lemma}[theorem]{Lemma}
\newtheorem{proposition}[theorem]{Proposition}
\newtheorem{corollary}[theorem]{Corollary}

\newtheorem{question}[theorem]{Question}
\newtheorem*{theorem*}{Theorem}
\newtheorem*{definition*}{Definition}

\newtheorem{innercustomgeneric}{\customgenericname}
\providecommand{\customgenericname}{}
\newcommand{\newcustomtheorem}[2]{%
  \newenvironment{#1}[1]
  {%
   \renewcommand\customgenericname{#2}%
   \renewcommand\theinnercustomgeneric{##1}%
   \innercustomgeneric
  }
  {\endinnercustomgeneric}
}

\newcustomtheorem{customthm}{Theorem}
\newcustomtheorem{customlemma}{Lemma}

\theoremstyle{remark}
\newtheorem{remark}[theorem]{Remark}
\newtheorem*{remark*}{Remark}
\newtheorem{example}[theorem]{Example}

\newcommand{\JJ}{\mathrm{J}}
\newcommand{\MM}{\mathrm{M}}
\newcommand{\BB}{\mathrm{B}}

\newcommand{\supp}{\mathrm{supp}}

\newcommand{\HE}{\mathscr{H}}

\newcommand{\RR}{\mathbb{R}}

\newcommand{\NN}{\mathbb{N}}

\def\newop#1{\expandafter\def\csname #1\endcsname{\mathop{\rm
#1}\nolimits}}

\newop{diag}

\title{Lorentzian polynomials}

\author{Petter Br\"and\'en and June Huh}

\address{Department of Mathematics, KTH, Royal Institute of Technology, Stockholm, Sweden.}
\email{pbranden@kth.se}

\address{Institute for Advanced Study and Princeton University, Princeton, NJ, USA.}
\address{Korea Institute for Advanced Study, Seoul, Korea.}
\email{junehuh@ias.edu}

\begin{document}

\begin{abstract}
We study the class of Lorentzian polynomials. 
The class contains homogeneous stable polynomials as well as volume polynomials of convex bodies and projective varieties. 
We prove that the Hessian of a nonzero Lorentzian polynomial has exactly one positive eigenvalue at any point on the positive orthant. 
This property can be seen as an analog of the Hodge--Riemann relations for Lorentzian polynomials. 

Lorentzian polynomials are intimately connected to matroid theory and negative dependence properties. 
We show that matroids, and more generally $\MM$-convex sets, are characterized by the Lorentzian property,
and develop a theory around Lorentzian polynomials. 
In particular, we provide a large class of linear operators that preserve the Lorentzian property and prove that Lorentzian measures enjoy several negative dependence properties. 
We also prove that the class of tropicalized Lorentzian polynomials coincides with the class of $\MM$-convex functions in the sense of discrete convex analysis. 
The tropical connection is used to produce Lorentzian polynomials from $\MM$-convex functions.

We give two applications of the general theory.
First, we prove that the homogenized multivariate Tutte polynomial of a matroid is Lorentzian whenever  the parameter $q$ satisfies $0< q \leq 1$. 
Consequences are proofs of the strongest Mason's conjecture from 1972 and negative dependence properties of  the random cluster model in statistical physics. 
Second, we prove that the multivariate characteristic polynomial of an $\MM$-matrix is Lorentzian. 
This refines a result of Holtz who proved that the coefficients of the characteristic polynomial of an $\MM$-matrix form an ultra log-concave sequence.
\end{abstract}

\maketitle

\thispagestyle{empty}

\tableofcontents

%\begin{comment}
\section{Introduction}

Let $\mathrm{H}^d_n$ be the space of degree $d$ homogeneous polynomials in $n$ variables with real coefficients.
%and write $\mathrm{P}^d_n\subseteq \mathrm{H}^d_n$ for the open subset of polynomials all of whose coefficients are positive.
%Matroid theory was introduced by Whitney \cite{Whitney} as a discrete model for notions of independence in various areas of mathematics, while negative dependence tries to model ``repelling particles'' in statistical physics, see \cite{BBL,Pem}. Examples of negative dependence properties are negative correlation, negative association, log-concavity and ultra log-concavity.  Several conjectures have been made regarding negative dependence properties in matroid theory. %Only recently have some of these been solved using combinatorial Hodge theory \cite{AHK,HK,HW,HSW,Huh1,Huh}. 
Inspired by Hodge's index theorem for projective varieties, we introduce a  class of  polynomials with remarkable properties. 
Let $\mathring{\mathrm{L}}^2_n \subseteq \mathrm{H}^2_n$ be the open subset of quadratic forms with positive coefficients that have the \emph{Lorentzian signature} $(+,-,\ldots,-)$.
For $d$ larger than $2$, we define an open subset $\mathring{\mathrm{L}}^d_n \subseteq \mathrm{H}^d_n$ by setting
\[
\mathring{\mathrm{L}}^d_n=\Big\{f \in \mathrm{H}^d_n\mid \text{$\partial_i f \in \mathring{\mathrm{L}}^{d-1}_n$ for all $i$} \Big\},
\]
where $\partial_i$ is the partial derivative with respect to the $i$-th variable.
Thus $f$ belongs to $\mathring{\mathrm{L}}^d_n$  if and only if all polynomials of the form $\partial_{i_1}\partial_{i_2}\cdots \partial_{i_{d-2}} f$ belongs to $\mathring{\mathrm{L}}^2_n$.
The polynomials in $\mathring{\mathrm{L}}^d_n$  are called \emph{strictly Lorentzian}, and the limits of strictly Lorentzian polynomials are called \emph{Lorentzian}.
We show that the class of Lorentzian polynomials contains the class of homogeneous stable polynomials (Section \ref{Lorent}) as well as volume polynomials of convex bodies  and projective varieties (Sections \ref{SectionConvex} and \ref{SectionProjective}).

Lorentzian polynomials link discrete and continuous notions of convexity.
Let  $\mathrm{L}^2_n \subseteq \mathrm{H}^2_n$ be the closed subset of quadratic forms with nonnegative coefficients that have at most one positive eigenvalue,
which is the closure of  $\mathring{\mathrm{L}}^2_n$ in $\mathrm{H}^2_n$.
We write $\text{supp}(f) \subseteq \NN^n$ for the support of $f \in \mathrm{H}^d_n$, 
the set of monomials appearing in $f$ with nonzero coefficients.
For $d$ larger than $2$, we define $\mathrm{L}^d_n \subseteq \mathrm{H}^d_n$ by setting
\[
\mathrm{L}^d_n=\Big\{f \in \mathrm{M}^d_n\mid \text{$\partial_i f \in \mathrm{L}^{d-1}_n$ for all $i$} \Big\},
\]
where $\mathrm{M}^d_n \subseteq \mathrm{H}^d_n$ is the set of polynomials with nonnegative coefficients whose supports are \emph{$\MM$-convex}
 in the sense of discrete convex analysis \cite{Murota}:
For any index $i$ and any $\alpha,\beta \in \text{supp}(f)$ whose $i$-th coordinates satisfy $\alpha_i > \beta_i$,
there is an index $j$ satisfying
\[
\alpha_j<\beta_j \ \ \text{and} \ \ \alpha-e_i+e_j \in \text{supp}(f) \ \ \text{and} \ \ \beta-e_j+e_i \in \text{supp}(f),
\]
where $e_i$ is the $i$-th standard unit vector in $\NN^n$.
Since $f \in \mathrm{M}^d_n$ implies $\partial_i f \in \mathrm{M}^{d-1}_n$, we have
% any $f$ belongs to $\mathrm{L}^d_n$  if and only if all polynomials of the form $\partial_{i_1}\partial_{i_2}\cdots \partial_{i_{d-2}} f$ belongs to $\mathrm{L}^2_n$.
\[
\mathrm{L}^d_n=\Big\{f \in \mathrm{M}^d_n\mid \text{$\partial_{i_1}\partial_{i_2} \cdots \partial_{i_{d-2}} f \in \mathrm{L}^{2}_n$ for all $i_1,i_2,\ldots,i_{d-2}$} \Big\}.
\]

%In Section \ref{secChar}, we show 
%The central result of
%In Section \ref{BasicTheory}, 
Our central result states that
%that $\mathrm{L}^d_n$ is the closure of $\mathring{\mathrm{L}}^d_n$ in $\mathrm{H}^d_n$.
%In other words,   
$\mathrm{L}^d_n$  is the set of Lorentzian polynomials in $\mathrm{H}^d_n$ (Theorem \ref{chars}).
%For the inclusion $\mathrm{L}^d_n \subseteq \text{closure} (\mathring{\mathrm{L}}^d_n)$,
To show that $\mathrm{L}^d_n$ is contained in the closure of $\mathring{\mathrm{L}}^d_n$,
we construct a Nuij-type homotopy for $\mathrm{L}^d_n$  in Section \ref{Lorent}.
The construction is used in Section \ref{secHR} to prove that all polynomials in  $\mathrm{L}^d_n$ satisfy 
a formal version of the Hodge--Riemann relations:
The Hessian of any nonzero polynomial in $\mathrm{L}^d_n$ has exactly one positive eigenvalue at any point on the positive orthant.
To show that  $\mathrm{L}^d_n$ contains the closure of $\mathring{\mathrm{L}}^d_n$,
we develop the theory of \emph{$c$-Rayleigh polynomials} in Section \ref{Negative}.
%We prove that the support of any homogeneous $c$-Rayleigh polynomial must be $\mathrm{M}$-convex
%The  inclusion $ \text{closure} (\mathring{\mathrm{L}}^d_n) \subseteq \mathrm{L}^d_n$, which relies on the theory of $c$-Rayleigh polynomials developed in Section,  
Since homogeneous stable polynomials are Lorentzian, the latter inclusion generalizes a result of Choe \emph{et al.} that the support of any homogenous multi-affine  stable polynomial is the set of bases of a matroid \cite{COSW}.
In Section \ref{secChar}, we use the above results to show that
the classes of strongly log-concave  \cite{GurvitsL}, completely log-concave \cite{AGV}, and Lorentzian polynomials are identical for homogeneous polynomials (Theorem \ref{allequal}).
 This enables us to affirmatively answer two questions of Gurvits on strongly log-concave polynomials (Corollaries \ref{CorollarySupport} and \ref{CorollaryProduct}).

Lorentzian polynomials are intimately connected to matroid theory and discrete convex analysis. 
We show that matroids, and more generally $\MM$-convex sets, are characterized by the Lorentzian property.
Let $\mathbb{P}\mathrm{H}^d_n$ be the projectivization of the vector space $\mathrm{H}^d_n$,
and let $\mathrm{L}_\mathrm{J}$ be the set of polynomials in $\mathrm{L}^d_n$ with nonempty support $\mathrm{J}$.
We denote  the images of $\mathrm{L}^d_n$, $\mathring{\mathrm{L}}^d_n$, and  $\mathrm{L}_\mathrm{J}$ in $\mathbb{P}\mathrm{H}^d_n$
by $\mathbb{P}\mathrm{L}^d_n$, $\mathbb{P}\mathring{\mathrm{L}}^d_n$,  and $\mathbb{P}\mathrm{L}_\mathrm{J}$  respectively,
and write
\[
\mathbb{P}\mathrm{L}^d_n = \coprod_{\mathrm{J}} \mathbb{P}\mathrm{L}_\mathrm{J}, %\quad \text{where $\mathrm{L}_\mathrm{J} =  \big\{f \in \mathrm{L}^d_n \mid \text{supp}(f)=\mathrm{J}\big\}$},
\]
where the union is over all nonempty  $\MM$-convex subsets  of the $d$-th discrete simplex in $\NN^n$.
The space $\mathbb{P}\mathrm{L}^d_n$ is homeomorphic to the intersection of $\mathrm{L}^d_n$ with the unit sphere in  $\mathrm{H}^d_n$ for the Euclidean norm on the coefficients.
%and similarly for $\mathbb{P}\mathring{\mathrm{L}}^d_n$ and $\mathbb{P}\mathrm{L}_\mathrm{J}$.
We prove that
$\mathbb{P}\mathrm{L}^d_n$ is a compact contractible set with contractible interior $\mathbb{P}\mathring{\mathrm{L}}^d_n$ (Theorem \ref{BallLike}).\footnote{We conjecture that $\mathbb{P}\mathrm{L}^d_n$ is homeomorphic to the closed Euclidean ball of the same dimension (Conjecture \ref{Ball}).} 
In addition, we show that
 $\mathbb{P}\mathrm{L}_\mathrm{J}$ is nonempty and contractible  for every nonempty $\MM$-convex set $\mathrm{J}$  (Theorem \ref{charjump} and Proposition \ref{tot}).
%Specifically,  $\mathrm{J}\subseteq \NN^n$ is $\MM$-convex if and only if its generating function $\sum_{\alpha \in \mathrm{J}} \frac{1}{\alpha!} w^\alpha$ is a Lorentzian polynomial (Theorem \ref{charjump}).
 Similarly, writing $\underline{\mathrm{H}}^{d}_{n}$ for the space of multi-affine degree $d$ homogeneous polynomials in $n$ variables
and $\underline{\mathrm{L}}^{d}_{n}$ for the corresponding set of multi-affine Lorentzian polynomials,
we have
\[
\mathbb{P}\underline{\mathrm{L}}^d_n = \coprod_{\mathrm{B}} \mathbb{P}\underline{\mathrm{L}}_\mathrm{B}, 
\]
where the union is over all rank $d$ matroids on the $n$-element set $[n]$.
%By Theorem \ref{flow} and Corollary \ref{Corollary-Multi-Affine}, 
The space $\mathbb{P}\underline{\mathrm{L}}^d_n$ is  compact and contractible, %with contractible interior $\mathbb{P}\mathrm{L}_{{n \brack d}}$.
and $\mathbb{P}\underline{\mathrm{L}}_\mathrm{B}$ is nonempty and contractible for every  matroid $\mathrm{B}$  (Remark \ref{BallLikeMultiaffine}). 
The latter fact contrasts the case of stable polynomials.
For example, 
 there is no stable polynomial  whose support is the set of bases of the Fano plane \cite{Branden}.
%nonempty $\MM$-convex subset $\mathrm{J}$ of $\Delta^d_n$.
%In addition, by Proposition \ref{tot}, $\mathbb{P}\underline{\mathrm{L}}_\mathrm{B}$ is contractible for every $\mathrm{B} $. %It would be interesting to study the boundary structure of $\mathbb{P}\underline{\mathrm{L}}^d_n$ in greater detail.

%\begin{customthm}{8.1}\label{eight}
%Every theorem must be numbered by hand.
%\end{customthm}

%Here is a reference to theorem~\ref{eight} and
%one to the important lemma~\ref{life-universe-everything}

%\begin{customlemma}{42}\label{life-universe-everything}
%This lemma explains everything.
%\end{customlemma}

In Section \ref{secOp}, we describe a large class of linear operators preserving the class of Lorentzian polynomials, thus providing a toolbox for working with Lorentzian polynomials. 
We give a Lorentzian analog of a theorem of Borcea and Br\"and\'en for stable polynomials \cite{BBI},
who characterized linear operators preserving stable polynomials (Theorem \ref{hrcpr}).
It follows from our result that any homogeneous linear operator that preserves stable polynomials and polynomials with nonnegative coefficients
also  preserves  Lorentzian polynomials (Theorem \ref{stab-lor}). 

In Section \ref{tropsec}, we strengthen the connection between Lorentzian polynomials and discrete convex analysis.
%The tropical connection is used to produce Lorentzian polynomials from discrete convex functions.
%Let $\nu$ be a function from $\NN^n$ to $\mathbb{R} \cup \{\infty\}$.
For a  function  $\nu:\NN^n \to \mathbb{R} \cup \{\infty\}$, 
we write  $\text{dom}(\nu) \subseteq \NN^n$ for the effective domain of $\nu$, 
the  subset of $\NN^n$ where $\nu$ is finite.
%The effective domain of a function $\nu:\NN^n \to \mathbb{R} \cup \{\infty\}$, denoted $\text{dom}(\nu)$, is the subset of $\NN^n$ where $\nu$ is finite.
% by definition,
%\[
%\text{dom}(\nu)=\Big\{\alpha \in \NN^n \mid \nu(\alpha) <\infty\Big\}.
%\]
For a positive real parameter $q$, we consider the generating function
\[
f^{\nu}_{q}(w)=\sum_{\alpha \in \text{dom}(\nu)} \frac{q^{\nu(\alpha)}}{\alpha!} w^\alpha, \quad w=(w_1,\ldots,w_n). %\ \ \text{and} \ \  g^{\nu}_{q}(w)=\sum_{\alpha \in \text{dom}(\nu)} {\delta \choose \alpha} q^{\nu(\alpha)} w^\alpha,
\]
The main result here is Theorem \ref{classical}, which states that $f^{\nu}_{q}$ is a Lorentzian polynomial for all $0<q \le 1$ if and only if the function $\nu$ is \emph{$\MM$-convex} in the sense of discrete convex analysis \cite{Murota}:
For any index $i$ and any $\alpha,\beta \in \text{dom}(\nu)$ whose $i$-th coordinates satisfy $\alpha_i >\beta_i$, there is an index $j$ satisfying
\[
\alpha_j <\beta_j \ \ \text{and} \ \  \nu(\alpha)+\nu(\beta) \ge \nu(\alpha-e_i+e_j)+\nu(\beta-e_j+e_i).
\]
In particular,  $\mathrm{J}\subseteq \NN^n$ is $\MM$-convex if and only if its exponential generating function $\sum_{\alpha \in \mathrm{J}} \frac{1}{\alpha!} w^\alpha$ is a Lorentzian polynomial (Theorem \ref{charjump}).
Another special case of Theorem \ref{classical}  is the statement that a homogeneous polynomial with nonnegative coefficients is Lorentzian if the natural logarithms of its normalized coefficients form an $\MM$-concave function (Corollary \ref{normalizedcoefficients}).
Working over the field of formal Puiseux series $\mathbb{K}$,  %Theorem \ref{classical} can be used to 
we show 
 that the tropicalization of any Lorentzian polynomial  over $\mathbb{K}$ is an $\MM$-convex function, and that \emph{all} $\MM$-convex functions are limits of tropicalizations of Lorentzian polynomials over $\mathbb{K}$ (Corollary \ref{CorollaryTropical}).  %\footnote{If the field of formal Laurent series with real exponents is used, then all $\MM$-convex functions are tropicalizations of Lorentzian polynomials.} 
This generalizes a result of Br\"and\'en \cite{SIAM}, who showed that the tropicalization of any homogeneous stable polynomial  over $\mathbb{K}$ is $\MM$-convex.\footnote{In \cite{SIAM}, the field of formal Puiseux series with real exponents  was used. The tropicalization used in \cite{SIAM} differs from ours by a sign.}   
In particular, for any matroid $\MM$ with the set of bases $\mathrm{B}$,
the \emph{Dressian} of all valuated matroids on $\MM$  can be identified with the tropicalization of the space of Lorentzian polynomials over $\mathbb{K}$ with support $\mathrm{B}$.

In Sections  \ref{SectionConvex} and \ref{SectionProjective}, we show that the volume polynomials of convex bodies  and projective varieties are Lorentzian.
It follows that, for any convex bodies $\mathrm{K}_1,\ldots,\mathrm{K}_n$ in $\mathbb{R}^d$,
%In particular, for any $d$-dimensional projective variety $Y$ and any nef divisors $\mathrm{H}_1,\ldots,\mathrm{H}_n$ on $Y$,
the set of all $\alpha \in \NN^n$ satisfying the conditions%the following  condition is $\MM$-convex:
\[
%\sum_{i=1}^n \alpha_i=d
\alpha_1+\cdots+\alpha_n=d \ \ \text{and} \ \ 
 V(\underbrace{\mathrm{K}_1, \ldots, \mathrm{K}_1}_{\alpha_1}, \ldots,  \underbrace{\mathrm{K}_n,  \ldots, \mathrm{K}_n}_{\alpha_n}) \neq 0
 \]
 is $\MM$-convex, 
 where the symbol $V$ stands for the mixed volume of convex bodies in $\RR^d$.
Similarly,  for any  $d$-dimensional projective variety $Y$ and any nef divisors $\mathrm{H}_1,\ldots,\mathrm{H}_n$ on $Y$, 
the set of all $\alpha \in \NN^n$ satisfying the conditions%the following  condition is $\MM$-convex:
\[
%\sum_{i=1}^n \alpha_i=d
\alpha_1+\cdots+\alpha_n=d \ \ \text{and} \ \ 
 (\underbrace{\mathrm{H}_1 \cdot \ldots \cdot \mathrm{H}_1}_{\alpha_1} \cdot \ldots \cdot  \underbrace{\mathrm{H}_n \cdot \ldots \cdot \mathrm{H}_n}_{\alpha_n}) \neq 0
 \]
 is $\MM$-convex, where  the symbol $\cdot$ stands for the intersection product of Cartier divisors on $Y$.
The problem of finding a Lorentzian polynomial that is not a volume polynomial remains open.
For a precise formulation, see Question \ref{RealizationQuestion}.

%Let $\mathrm{M}$ be a rank $d$ matroid on an $n$ element set $[n]$. 
In Section \ref{SecqP}, we use the basic theory developed in Section \ref{BasicTheory} to show that the homogenized multivariate Tutte polynomial of any matroid is Lorentzian. 
We use the Lorentzian property to prove a conjecture of Mason from 1972 on the enumeration of independent sets  \cite{Mason}: %Another consequence is that the $q$-state Potts model of any matroid  satisfies strong negative dependence properties whenever 
For any matroid $\mathrm{M}$ on $[n]$ and any positive integer $k$,
 \[
\frac{I_k(\mathrm{M})^2}{{n \choose k}^2} \ge \frac{I_{k+1}(\mathrm{M})}{{n \choose k+1}}\frac{I_{k-1}(\mathrm{M})}{{n \choose k-1}},
\]
where $I_k(\mathrm{M})$ is the number of $k$-element independent sets of $\mathrm{M}$.  
More generally, the Lorentzian property  reveals several inequalities satisfied by the coefficients of the classical  \emph{Tutte polynomial} 
\[
\mathrm{T}_\mathrm{M}(x,y)=\sum_{A  \subseteq [n]} (x-1)^{\text{rk}_\mathrm{M}([n])-\text{rk}_\mathrm{M}(A)} (y-1)^{|A|-\text{rk}_\mathrm{M}(A)},
\]
where $\text{rk}_\mathrm{M}:\{0,1\}^n \to \mathbb{N}$ is the rank function of $\mathrm{M}$.
For example, if we write
\[
w^{\text{rk}_\mathrm{M}([n])}\mathrm{T}_\mathrm{M}\Big(1+\frac{q}{w},1+w\Big)%=\sum_{k=0}^n \Big( \sum_{A \in {n \brack k}} q^{\text{rk}_\mathrm{M}([n])-\text{rk}_\mathrm{M}(A)}\Big) w^k
=\sum_{k=0}^n c_q^k(\mathrm{M}) w^k, % q^d \sum_{A  \subseteq E} q^{-r(A)} w^{|A|}.
\]
then the sequence $c_q^k(\mathrm{M})$ is ultra log-concave for every $0 \le q \le 1$.\footnote{
Nima Anari, Kuikui Liu, Shayan Oveis Gharan and Cynthia Vinzant have  independently developed methods that partially overlap with our work in a series of papers  \cite{AGV,ALGVII,ALGVIII}. They study the class of \emph{completely log-concave polynomials}. For homogenous polynomials this class agrees with the class of Lorentzian polynomials, see Theorem \ref{allequal} in this paper.
 The main overlap is an independent proof of Mason's conjecture in  \cite{ALGVIII}. The manuscript \cite{BH}, which is not intended for publication, contains a short self-contained proof of Mason's conjecture which was published on arXiv simultaneously as \cite{ALGVIII}.
 In addition, the authors of \cite{AGV}  prove that the basis generating polynomial of any matroid is completely log-concave, using results of Adiprasito, Huh, and Katz \cite{AHK}.
An equivalent statement on the Hessian of the basis generating polynomial can be found in  \cite[Remark 15]{HW}. 
 A self-contained proof of the complete log-concavity of the basis generating polynomial, based on an implication similar to $(3) \Rightarrow (1)$ of Theorem \ref{allequal} in this paper, appears in \cite[Section 5.1]{ALGVII}.
The authors of \cite{ALGVII} apply these results to design an FPRAS to count the number of bases of any matroid given by an independent set oracle, and to prove the conjecture of Mihail and Vazirani that the bases exchange graph of any matroid has expansion at least $1$.}
%Nima Anari, Kuikui Liu, Shayan Oveis Gharan and Cynthia Vinzant \cite{ALGVIII} have independently developed methods that partially overlap with our work. The main overlap is that they give an independent proof of Mason's conjecture  which is similar to our proof. The manuscript \cite{BH}, which is not intended for publication, contains a short self-contained proof of Mason's conjecture which was published on arXiv simultaneously as \cite{ALGVIII}.}

% one can find $50$ different equivalent characterizations of nonsingular $\MM$-matrices
 In Section \ref{secMm},
we show that the multivariate  characteristic polynomial of any $\MM$-matrix is Lorentzian.\footnote{An $n \times n$ matrix is an \emph{$\MM$-matrix} if all the off-diagonal entries are nonpositive and all the principal minors are positive.
The class of $\MM$-matrices shares many properties of hermitian positive definite matrices and appears in mathematical economics and computational biology \cite{BP}.}  %(Theorem \ref{M-matrix}). 
This strengthens a theorem of Holtz \cite{Holtz}, who proved that the coefficients of the characteristic polynomial of any $\MM$-matrix form an ultra log-concave sequence.

In Section \ref{secLM}, we define a class of discrete probability measures, called \emph{Lorentzian measures},
properly containing the class of strongly Rayleigh measures studied in \cite{BBL}. 
We show that Lorentzian measures enjoy several negative dependence properties and prove that the class of Lorentzian measures is closed under the symmetric exclusion process. 
As an example, %for any matroid $\mathrm{M}$ on $n$ elements, 
we show that  the uniform  measure $\mu_\mathrm{M}$ on $\{0,1\}^n$ concentrated on the independent sets of a matroid $\mathrm{M}$ on $[n]$ is Lorentzian (Proposition \ref{PropositionMatroidMeasures}).
%We show that the measure $\mu_\mathrm{M}$ is Lorentzian (Proposition \ref{PropositionMatroidMeasures}).
A conjecture of Kahn \cite{Kahn} and Grimmett--Winkler \cite{GW} states that, for any graphic matroid $\mathrm{M}$ and distinct elements $i$ and $j$,
\[
\text{Pr}(\text{$F$ contains $i$ and $j$}) \le \text{Pr}(\text{$F$ contains $i$})\hspace{0.5mm} \text{Pr}(\text{$F$ contains  $j$}),
\]
where $F$ is an independent set of $\mathrm{M}$ chosen uniformly at random.
The Lorentzian property of the measure $\mu_\mathrm{M}$ shows that, for any matroid $\mathrm{M}$ and distinct elements $i$ and $j$,
\[
\text{Pr}(\text{$F$ contains $i$ and $j$}) \le 2\text{Pr}(\text{$F$ contains $i$})\hspace{0.5mm} \text{Pr}(\text{$F$ contains  $j$}),
\]
where $F$ is an independent set of $\mathrm{M}$ chosen uniformly at random.

\noindent
{\bf Acknowledgments.}
Petter Br\"and\'en is a Wallenberg Academy Fellow supported by the Knut and Alice Wallenberg Foundation and Vetenskapsr\aa det. 
June Huh was supported by NSF Grant DMS-1638352 and the Ellentuck Fund.
Special thanks go to  anonymous referees, Claus Scheiderer's reading group, Jonathan Leake, and Yanxin Liu, whose valuable comments significantly improved the quality of the paper.

%\end{comment}

\section{Basic theory}\label{BasicTheory}

\subsection{The space of Lorentzian polynomials}\label{Lorent}

%\subsubsection{}

Let $n$ and $d$ be nonnegative integers, and set $[n]=\{1,\ldots,n\}$. 
We write $\mathrm{H}^d_n$ for the set of degree $d$ homogeneous polynomials in $\mathbb{R}[w_1,\ldots,w_n]$.
We define a topology on $\mathrm{H}^d_n$ using the Euclidean norm for the coefficients, 
and write $\mathrm{P}^d_n \subseteq \mathrm{H}^d_n$ for the open subset of polynomials  all of whose %$\multiset{n}{k}$ 
coefficients are positive. 
The \emph{Hessian} of $f \in \RR[w_1,\ldots, w_n]$ is the symmetric matrix  
\[
\mathscr{H}_f(w)=\Big( \partial_i\partial_j f \Big)_{i,j=1}^n,
\] 
where $\partial_i$ stands for the partial derivative $\frac{\partial}{\partial w_i}$. 
For $\alpha \in \mathbb{N}^n$,
we write
\[
\alpha=\sum_{i=1}^n \alpha_ie_i \ \ \text{and} \ \ |\alpha|_1=\sum_{i=1}^n \alpha_i,
\]
where $\alpha_i$ is a nonnegative integer and $e_i$ is the standard unit vector in $\NN^n$, and set
\[
w^\alpha=w_1^{\alpha_1} \cdots w_n^{\alpha_n}  \ \ \text{and} \ \ 
 \partial^\alpha=\partial_1^{\alpha_1} \cdots \partial_n^{\alpha_n}.
 \]
We define the \emph{$d$-th discrete simplex}  $\Delta^d_n \subseteq \NN^n$ by
\[
\Delta^d_n =\Big\{ \alpha \in \NN^n \mid |\alpha|_1=d\Big\},
\]
and define the \emph{Boolean cube}  $\{0,1\}^n \subseteq \NN^n$ by
\[
\{0,1\}^n=\Bigg\{\sum_{i \in S} e_i  \in \NN^n \mid S \subseteq [n]\Bigg\}. 
\]
The intersection of the $d$-th discrete simplex and the Boolean cube  will be denoted
\[
 { n\brack d}=\{0,1\}^n \cap \Delta^d_n.
\]
The cardinality of ${n \brack d}$ is the binomial coefficient ${n \choose d}$.
%and the sets ${n \brack 1}$ and $[n]$ can be  identified with each other.
We often identify a subset $S$ of $[n]$ with the zero-one vector $\sum_{i \in S} e_i$ in $\NN^n$.
For example, we write $w^S$ for the square-free monomial $\prod_{i \in S} w_i$.

\begin{definition}[Lorentzian polynomials]\label{FirstDefinition}
We set $\mathring{\mathrm{L}}^0_n=\mathrm{P}^0_n$, $\mathring{\mathrm{L}}^1_n=\mathrm{P}^1_n$, and 
\[
\mathring{\mathrm{L}}^2_n=\Big\{f \in \mathrm{P}^2_n\mid \text{$\HE_f$ is nonsingular and has exactly one positive eigenvalue} \Big\}.
\]
For $d$ larger than $2$, we define $\mathring{\mathrm{L}}^d_n$ recursively by setting
\[
\mathring{\mathrm{L}}^d_n=\Big\{f \in \mathrm{P}^d_n\mid \text{$\partial_i f \in \mathring{\mathrm{L}}^{d-1}_n$ for all $i \in [n]$} \Big\}.
\]
The polynomials in $\mathring{\mathrm{L}}^d_n$  are called \emph{strictly Lorentzian}, and the limits of strictly Lorentzian polynomials are called \emph{Lorentzian}.
%polynomials in the closure of $\mathring{\mathrm{L}}^d_n$ are called \emph{Lorentzian}. 
\end{definition}

%The topology on the  space of homogeneous polynomials $\mathrm{H}^d_n$ is defined by the Euclidean norm for the coefficients.
Clearly, $\mathring{\mathrm{L}}^d_n$ is an open subset of $\mathrm{H}^d_n$, %of the space of degree $d$ homogeneous polynomials in $n$ variables.
and the space $\mathring{\mathrm{L}}^2_n$  may be identified with the set of $n \times n$ symmetric matrices with positive entries that have the \emph{Lorentzian signature} $(+,-,\ldots, -)$. 
Unwinding the recursive definition, we have
\[
\mathring{\mathrm{L}}^d_n= \Big\{f \in \mathrm{P}^d_n\mid \text{$\partial^\alpha f \in \mathring{\mathrm{L}}^{2}_n$ for every $\alpha \in \Delta^{d-2}_n$} \Big\}.
\]
Proposition \ref{PropositionStableLorentzian} below on stable polynomials shows that $\mathring{\mathrm{L}}^d_n$ is nonempty for every $n$ and $d$.

%The class of Lorentzian polynomials extends the the class of homogeneous stable polynomials.

An important subclass of Lorentzian polynomials is homogeneous stable polynomials, %(Proposition \ref{PropositionStableLorentzian}),
which play a guiding role in many of our proofs.
Recall that a polynomial $f$ in $\RR[w_1,\ldots, w_n]$ is \emph{stable}  if $f$ is non-vanishing on $\mathcal{H}^n$ or identically zero,
where $\mathcal{H}$ is the open upper half plane in $\mathbb{C}$.
Let $\mathrm{S}^d_n$ be the set of degree $d$ homogeneous stable polynomials in $n$ variables with nonnegative coefficients. %\footnote{All the nonzero coefficients of a homogeneous stable polynomial must have the same sign \cite[Theorem 6.1]{COSW}.}
Hurwitz's theorem shows that $\mathrm{S}^d_n$ is a closed subset of $\mathrm{H}^d_n$ \cite[Section 2]{Wagner11}.
When $f$ is homogeneous and has nonnegative coefficients,
 the stability of $f$ is equivalent to any one of the following statements on univariate polynomials in the variable $x$ \cite[Theorem 4.5]{BBL}:
\begin{enumerate}[--]\itemsep 5pt
\item
 For any $u \in \mathbb{R}^n_{>0}$,  $f(xu-v)$ has only real zeros for all $v \in \mathbb{R}^n$.
\item
 For some $u \in \mathbb{R}^n_{>0}$,  $f(xu-v)$ has only real zeros for all $v \in \mathbb{R}^n$.
 \item
 For any $u \in \mathbb{R}^n_{\ge 0}$ with $f(u)>0$, $f(xu-v)$ has only real zeros for all $v \in \mathbb{R}^n$.
  \item
 For some $u \in \mathbb{R}^n_{\ge 0}$ with $f(u)>0$, $f(xu-v)$ has only real zeros for all $v \in \mathbb{R}^n$.
\end{enumerate}
We refer to \cite{Wagner11} and \cite{Pemantle}   for background on the class of stable polynomials.
We will use the fact that any  polynomial $f \in \mathrm{S}^d_n$ is the limit of polynomials in the interior of $\mathrm{S}^d_n$, that is, of \emph{strictly stable polynomials} \cite{Nuij}. 
%Since all the nonzero coefficients of a homogeneous stable polynomial have the same sign \cite[Theorem 6.1]{COSW},
%the interior of $\mathrm{S}^d_n$ is a subset of $\mathrm{P}^d_n$.
%all the coefficients of a homogeneous strictly stable polynomial must be positive.

\begin{proposition}\label{PropositionStableLorentzian}
Any polynomial in $\mathrm{S}^d_n$ is Lorentzian.
\end{proposition}

\begin{proof}
We show that the interior of  $\mathrm{S}^d_n$ is a subset of  $\mathring{\mathrm{L}}^d_n$  by induction on $d$.
When $d=2$, the statement follows from Lemma \ref{StrictlyStable} below. %A quadratic form is strictly Lorentzian if and only if it is strictly stable.
The general case follows from the fact that $\partial_i$ is an open map sending $\mathrm{S}^d_n$ to $\mathrm{S}^{d-1}_n$ \cite[Lemma 2.4]{Wagner11}.
%it follows that homogeneous strictly stable  polynomials are strictly Lorentzian.
%Thus, homogeneous stable  polynomials with nonnegative coefficients are Lorentzian. 
\end{proof}

All the nonzero coefficients of a homogeneous stable polynomial have the same sign \cite[Theorem 6.1]{COSW}.
Thus, any homogeneous stable polynomial is a constant multiple of a Lorentzian polynomial.
For example, determinantal polynomials of the form
\[
f(w_1,\ldots, w_n) = \det(w_1A_1+\cdots+w_nA_n),
\]
where $A_1,\ldots, A_n$ are positive semidefinite matrices, are stable \cite[Proposition 2.4]{BB08}, and hence Lorentzian. 

\begin{example}\label{n2}
Consider the homogeneous bivariate polynomial with positive coefficients 
\[
f= \sum_{k=0}^d a_k w_1^k w_2^{d-k}.
\]
Computing the partial derivatives  of $f$ reveals that $f$ is strictly Lorentzian if and only if 
\[
 \frac {a_{k}^2} {{\binom d {k}}^2} > \frac {a_{k-1}} {\binom d {k-1}} \frac {a_{k+1}} {\binom d {k+1}} \ \ \text{for all $0<k<d$}.
\]
On the other hand,   $f$ is stable if and only if the univariate polynomial $f|_{w_2=1}$ has only real zeros.
Thus, a Lorentzian polynomial need not be stable.
For example, consider the cubic form
\[
f=2w_1^3+12w_1^2w_2+18w_1w_2^2+\theta w_2^3,
\]
where $\theta$ is a real parameter. A straightforward computation shows  that 
\[
\text{$f$ is Lorentzian if and only if $0 \le \theta \le 9$, and $f$ is stable if and only if $0 \le \theta \le 8$.}
\]
\end{example}

\begin{example}
Clearly, if $f$ is in the closure of $\mathring{\mathrm{L}}^d_n$ in $\mathrm{H}^d_n$, then $f$ has nonnegative coefficients and
\[
\text{$\partial^\alpha f$ has at most one positive eigenvalue for every $\alpha \in \Delta^{d-2}_n$.}
\]
The bivariate cubic $f=w_1^3+w_2^3$ shows that the converse fails.
In this case, $\partial_1 f$ and $\partial_2 f$ are Lorentzian, but $f$ is not Lorentzian.
\end{example}

We give alternative characterizations of $\mathring{\mathrm{L}}^2_n$. 
Similar arguments  were given in  \cite{Gregor} and  \cite[Theorem 5.3]{COSW}.

\begin{lemma}\label{StrictlyStable}
The following conditions are equivalent for any $f \in \mathrm{P}^2_n$.
\begin{enumerate}[(1)]\itemsep 5pt
\item The Hessian of $f$ has the Lorentzian signature $(+,-,\ldots, -)$, that is, $f \in \mathring{\mathrm{L}}^2_n$.
\item For any nonzero $u \in \mathbb{R}^n_{\ge 0}$, 
$(u^T \mathscr{H}_f v)^2 > (u^T \mathscr{H}_f u)(v^T \mathscr{H}_f v)$ for any $v \in \mathbb{R}^n$ not parallel to  $u$.
\item For some  $u \in \mathbb{R}^n_{\ge 0}$, 
$(u^T \mathscr{H}_f v)^2 > (u^T \mathscr{H}_f u)(v^T \mathscr{H}_f v)$ for any $v \in \mathbb{R}^n$ not parallel to  $u$.
\item For any nonzero $u \in \mathbb{R}^n_{\ge 0}$, the univariate polynomial $f(xu-v)$ in $x$ has two distinct real zeros for any $v \in \mathbb{R}^n$ not parallel to  $u$.
\item For some  $u \in \mathbb{R}^n_{\ge 0}$, the univariate polynomial $f(xu-v)$ in $x$ has two distinct real zeros for any $v \in \mathbb{R}^n$ not parallel to  $u$.
\end{enumerate}
\end{lemma}

It follows that a quadratic form with nonnegative coefficients is strictly Lorentzian if and only if it is strictly stable.
Thus,  a  quadratic form with nonnegative coefficients is Lorentzian if and only if it is stable.

\begin{proof}
We prove $(1) \Rightarrow (2)$.
Since all the entries of $\mathscr{H}_f$ are positive, $u^T \mathscr{H}_f u >0$ for any nonzero  $u \in \mathbb{R}^n_{\ge 0}$.
By Cauchy's interlacing theorem, for any $v \in \mathbb{R}^n$ not parallel to $u$,
 the restriction of $\mathscr{H}_f$ to the plane spanned by $u,v$ has signature $(+,-)$.
It follows that
\[
\det \left(\begin{array}{cc} u^T \mathscr{H}_f u&u^T \mathscr{H}_f v \\ u^T \mathscr{H}_f v&v^T \mathscr{H}_f v\end{array} \right)=(u^T \mathscr{H}_f u)(v^T \mathscr{H}_f v)-(u^T \mathscr{H}_f v)^2<0.
\]

%Clearly,  $(2) \Rightarrow (3)$. 
We prove $(3) \Rightarrow (1)$. Let $u$ be the nonnegative vector in the statement $(3)$.
Then $\mathscr{H}_f$ is negative definite on the hyperplane 
$\{v \in \mathbb{R}^n \mid u^T \mathscr{H}_f v=0\}$.
Since $f \in \mathrm{P}^2_n$, we have $u^T \mathscr{H}_f u>0$, and hence  $\mathscr{H}_f$ has the Lorentzian signature.

%We have $(2) \Leftrightarrow (4) \Rightarrow (3) \Leftrightarrow (5)$ because
The remaining implications follows from the fact that the univariate polynomial $\frac{1}{2}f(xu-v)$ has the discriminant $(u^T \mathscr{H}_f v)^2 - (u^T \mathscr{H}_f u)(v^T \mathscr{H}_f v)$.
\end{proof}

%The same argument shows that a nonzero quadratic form $f$ with nonnegative coefficients is stable if and only if 
 %$\mathscr{H}_f$ has exactly one positive eigenvalue.

%\begin{lemma}\label{LemmaStable}
%The following conditions are equivalent for a homogeneous quadratic polynomial $f$ with nonnegative coefficients.
%\begin{enumerate}[(1)]\itemsep 5pt
%\item The Hessian of $f$ has at most one positive eigenvalue.
%\item For any $u \in \mathbb{R}^n_{\ge 0}$, 
%$(u^T \mathscr{H}_f v)^2 \ge (u^T \mathscr{H}_f u)(v^T \mathscr{H}_f v)$ for all $v \in \mathbb{R}^n$.
%\item For some  $u \in \mathbb{R}^n_{\ge 0}$, 
%$(u^T \mathscr{H}_f v)^2 \ge  (u^T \mathscr{H}_f u)(v^T \mathscr{H}_f v)$ for all $v \in \mathbb{R}^n$.
%\end{enumerate}
%\end{lemma}

%\subsubsection{}

Matroid theory captures various combinatorial notions of independence. A \emph{matroid} $\MM$ on $[n]$ is a nonempty family of subsets $\BB$ of $[n]$, called the \emph{set of bases of} $\MM$, that satisfies the \emph{exchange property}: 
\begin{multline*}
\text{For any $B_1, B_2 \in \BB$ and $i \in B_1\setminus B_2$, there is  $j \in B_2\setminus B_1$ such that $(B_1 \setminus i) \cup j \in \BB$.}
\end{multline*}
We refer to \cite{Oxley} for background on matroid theory.
More generally, following \cite{Murota},
we define  a subset $\JJ \subseteq \NN^n$ to be \emph{$\MM$-convex} if it satisfies any one of the following equivalent conditions\footnote{The class of $\MM$-convex sets is essentially identical to the class of generalized polymatroids in the sense of \cite{Fujishige}. 
Some other notions in the literature that are equivalent to $\MM$-convex sets are  integral polymatroids \cite{Welsh},  discrete polymatroids \cite{HH}, and integral generalized permutohedras \cite{Postnikov}.
We refer to \cite[Section 1.3]{Murota} and \cite[Section 4.7]{Murota} for more details.}: 
\begin{enumerate}[--]\itemsep 5pt
\item For any $\alpha, \beta \in \JJ$ and any index $i$ satisfying $\alpha_i>\beta_i$, there is an index $j$ satisfying
\[
\alpha_j<\beta_j \ \ \text{and} \ \ \alpha-e_i+e_j \in \JJ.
\]
\item For any $\alpha, \beta \in \JJ$ and  any index $i$ satisfying $\alpha_i>\beta_i$, there is an index $j$ satisfying
\[
\alpha_j<\beta_j \ \ \text{and} \ \ \alpha-e_i+e_j \in \JJ \ \ \text{and} \ \ \beta-e_j+e_i \in \JJ.
\]
\end{enumerate}
The first condition is called the \emph{exchange property} for $\MM$-convex sets,
and the second condition is called the \emph{symmetric exchange property} for $\MM$-convex sets. 
A proof of the equivalence can be found in \cite[Chapter 4]{Murota}.
Note that any  $\MM$-convex subset of $\NN^n$ is necessarily contained in the discrete simplex $\Delta^d_n$ for some $d$.
%An $\MM$-convex subset of $\{0,1\}^n$ is the set of bases of a matroid on $[n]$ and conversely.
We refer to \cite{Murota} for a comprehensive treatment of $\MM$-convex sets.

Let $f$ be a polynomial in $\mathbb{R}[w_1,\ldots,w_n]$.
We write $f$ in the normalized form
\[
f=\sum_{\alpha \in \NN^n} \frac{c_\alpha}{\alpha!} \hspace{0.5mm}w^\alpha,  \ \ \text{where $\alpha!=\prod_{i=1}^n \alpha_i!$.}
\]
%where $\alpha!$ is the product of factorials $\prod_{i=1}^n \alpha_i!$.
The \emph{support} of the polynomial $f$ is the subset of $\NN^n$ defined by
\[
\mathrm{supp}(f)= \Big\{ \alpha \in \NN^n \mid c_\alpha \neq 0\Big\}.
\]
We write $\mathrm{M}^d_n$ for the set of all degree $d$ homogeneous polynomials in $\mathbb{R}_{\ge 0}[w_1,\ldots,w_n]$ 
whose supports are  $\MM$-convex.
Note that, in our convention, the empty subset of $\NN^n$ is an $\MM$-convex set.
Thus, the zero polynomial belongs to $\mathrm{M}^d_n$,
and $f \in \mathrm{M}^d_n$ implies $\partial_i f \in \mathrm{M}^{d-1}_n$.
It follows from \cite[Theorem 3.2]{Branden} that $\mathrm{S}^d_n \subseteq \mathrm{M}^d_n$.

%\[
%\mathrm{L}^2_n=\Big\{f \in \mathrm{M}^2_n\mid \text{$\HE_f$  has at most one positive eigenvalue} \Big\}.
%\]

\begin{definition}\label{SecondDefinition}
We set $\mathrm{L}^0_n=\mathrm{S}^0_n$, $\mathrm{L}^1_n=\mathrm{S}^1_n$, and  $\mathrm{L}^2_n=\mathrm{S}^2_n$.
For $d$ larger than $2$, we define %$\mathrm{L}^d_n$  by setting
\[
\mathrm{L}^d_n=\Big\{f \in \mathrm{M}^d_n\mid \text{$\partial_i f \in \mathrm{L}^{d-1}_n$ for all $i \in [n]$} \Big\}
= \Big\{f \in \mathrm{M}^d_n\mid \text{$\partial^\alpha f \in \mathrm{L}^{2}_n$ for every $\alpha \in \Delta^{d-2}_n$} \Big\}.
\]
\end{definition}

Clearly, $\mathrm{L}^d_n$ contains $\mathring{\mathrm{L}}^d_n$.
In Theorem \ref{chars}, we show that  $\mathrm{L}^d_n$ is the closure of $\mathring{\mathrm{L}}^d_n$ in $\mathrm{H}^d_n$.
In other words,  $\mathrm{L}^d_n$ is exactly the set of degree $d$ Lorentzian polynomials in $n$ variables.
In this section, we show that $\mathring{\mathrm{L}}^d_n$ is contractible and its  closure 
contains $\mathrm{L}^d_n$. 
%and  that $\mathring{\mathrm{L}}^d_n$ and $\mathrm{L}^d_n$ are connected.
The following proposition plays a central role in our analysis of $\mathrm{L}^d_n$.
Analogous statements, in the context of hyperbolic polynomials and stable polynomials, appear in  \cite{Nuij} and \cite{LiebSokal}.
We fix  a degree $d$ homogeneous polynomial $f$  in $n$ variables and  indices $i$, $j$ in $[n]$.

\begin{proposition}\label{deform}
If $f \in \mathrm{L}^d_n$, then $\big(1+\theta w_i\partial_j \big)f \in \mathrm{L}^d_n$
for every nonnegative real number $\theta$.
\end{proposition}

We prepare the proof of Proposition \ref{deform} with two lemmas.

\begin{lemma}\label{support}
If $f \in \mathrm{M}^d_n$, then $\big(1+\theta w_i\partial_j \big) f \in \mathrm{M}^d_n$ for every nonnegative real number $\theta$.
\end{lemma}

\begin{proof}
We may suppose $\theta=1$ and $j=n$. 
We use two combinatorial lemmas from \cite{KMT}.
Introduce a new variable $w_{n+1}$, and set
\[
g(w_1,\ldots,w_n,w_{n+1})=f(w_1,\ldots,w_n+w_{n+1})=\sum_{k = 0}^d \frac{1}{k!}w_{n+1}^k\partial_n^k f(w_1,\ldots,w_n).
\]
By  \cite[Lemma 6]{KMT},  the support of $g$ is $\MM$-convex.
In terms of \cite{KMT}, the support of $g$ is obtained from the support of $f$ by an elementary splitting,
and the operation of splitting preserves $\MM$-convexity.
Therefore, $g$ belongs to $\mathrm{M}^d_{n+1}$.
Since the intersection of an $\mathrm{M}$-convex set with a cartesian product of intervals is $\mathrm{M}$-convex,
it follows that
\[
\big(1+w_{n+1}\partial_n\big)f \in \mathrm{M}^d_{n+1}.
\]
By \cite[Lemma 9]{KMT}, the above displayed inclusion implies
\[
\big(1+w_{i}\partial_n\big)f \in \mathrm{M}^d_{n}.
\]
In terms of \cite{KMT}, the support of $\big(1+w_{i}\partial_n\big)f$ is obtained from the support of $\big(1+w_{n+1}\partial_n\big)f$ by an elementary aggregation, and the operation of aggregation preserves $\MM$-convexity.
\end{proof}

For stable polynomials $f$ and $g$  in $\mathbb{R}[w_1,\ldots,w_n]$, we define a relation $f \prec g$ by
\[
f \prec g \Longleftrightarrow
\text{$g+w_{n+1}f$ is a stable polynomial in $\mathbb{R}[w_1,\ldots,w_n,w_{n+1}]$}. 
\]
If $f$ and $g$ are univariate polynomials with  leading coefficients of the same sign, then $f \prec g$  if and only if the zeros of $f$ interlace the zeros of $g$ \cite[Lemma 2.2]{BB10}. 
In general, we have
%$f \prec g$ is equivalent to the   following statement on univariate polynomials in the variable $x$:
\[
f \prec g \Longleftrightarrow
\text{$f(xu-v) \prec g(xu-v)$ for all $u \in \mathbb{R}^n_{>0}$ and $v \in \mathbb{R}^n$.}
\]
For later use, we record here basic properties of stable polynomials and the relation $\prec$.

\begin{lemma}\label{closure}
Let $f,g_1,g_2,h_1,h_2$ be stable polynomials  satisfying $h_1 \prec f \prec g_1$ and  $h_2\prec f \prec g_2$. 
\begin{enumerate}[(1)]\itemsep 5pt
\item The derivative $\partial_1 f$ is stable  and $\partial_1 f \prec f$.
\item The diagonalization $f(w_1,w_1,w_3,\ldots,w_n)$ is stable.
\item The dilation $f(a_1w_1,\ldots,a_nw_n)$ is stable for any $a \in \mathbb{R}^n_{\ge 0}$.
\item If $f$ is not identically zero,  then $f \prec \theta_1 g_1+\theta_2 g_2$ for any $\theta_1,\theta_2 \ge 0$.
\item If $f$ is not identically zero,  then $\theta_1 h_1+\theta_2 h_2 \prec f$ for any $\theta_1,\theta_2 \ge 0$.
\end{enumerate}
\end{lemma}

The statement $\partial_1 f \prec f$ appears, for example, in \cite[Section 4]{BBL}.
It follows that, if $f$ is stable, then $(1+\theta w_i \partial_j) f$ is stable for every nonnegative real number $\theta$.
The remaining proof of Lemma \ref{closure} can be found in \cite[Section 2]{Wagner11} and \cite[Section 2]{BB10}.
%We will only use Lemma \ref{closure} when $f,g_1,g_2$ are homogeneous quadratic polynomials.

\begin{proof}[Proof of Proposition \ref{deform}]
When $d = 2$, %(1) and (2) of 
Lemma \ref{closure}  implies Proposition \ref{deform}.
Suppose $d \ge 3$, and set
\[
g=\big(1+\theta w_i\partial_j \big)f.
\]
By Lemma \ref{support}, the support of $g$ is $\MM$-convex. 
Therefore, it is enough to prove that
$\partial^\alpha g$ is stable
for all $\alpha \in \Delta^{d-2}_n$.
We give separate arguments when $\alpha_i=0$ and $\alpha_i>0$.
If $\alpha_i=0$,  then
\[
\partial^\alpha g=\partial^\alpha f+\theta w_i  \partial^{\alpha+e_j} f.
\]
In this case,  (1), (2), and (3) of 
Lemma \ref{closure}  for $\partial^\alpha f$ show that $\partial^\alpha g$ is stable.
If $\alpha_i>0$, then
\begin{align*}
\partial^{\alpha} g&=\partial^{\alpha} f +\theta \alpha_i \partial^{\alpha-e_i+e_j} f +\theta w_i \partial^{\alpha+e_j}  f \\
&=\partial_i \Big(\partial^{\alpha-e_i} f\Big) +\theta \alpha_i \partial_j \Big( \partial^{\alpha-e_i} f\Big) +\theta w_i \partial_i \partial_j \Big(\partial^{\alpha-e_i}  f\Big).
\end{align*}
In this case, (1) of  Lemma \ref{closure} applies to the stable polynomials $\partial^\alpha f$ and  $\partial^{\alpha-e_i+e_j} f$:
\[
 \partial_i \partial_j \Big(\partial^{\alpha-e_i}  f\Big) \prec \partial_i \Big(\partial^{\alpha-e_i} f\Big)  \ \ \text{and} \ \   \partial_i \partial_j \Big(\partial^{\alpha-e_i}  f\Big) \prec \partial_j \Big(\partial^{\alpha-e_i} f\Big).
\]
Therefore, unless $\partial^{\alpha+e_j}f$ is identically zero, $\partial^\alpha g$ is stable by (2) and (4) of Lemma \ref{closure}.

It remains to prove that, whenever $\alpha_i$ is positive  and  $\partial^{\alpha+e_j}f$ is identically zero, 
\[
\text{$\partial_i \Big(\partial^{\alpha-e_i} f\Big) +\phi \  \partial_j \Big( \partial^{\alpha-e_i} f\Big)$ is stable for every nonnegative real number $\phi$.}
\]
Since the cubic form $\partial^{\alpha-e_i} f$ is in $\mathrm{L}^3_n$, it is enough to prove the statement when $d=3$ and $\alpha=e_i$.

We show that, if $f$ is in $\mathrm{L}^3_n$ and $\partial_i \partial_j f$ is identically zero, then
\[
\text{$ \partial_i f + \phi\ \partial_j f$ is stable for every nonnegative real number $\phi$.}
\]
The statement is clear when $(\partial_i f)( \partial_j f)$ is identically zero.
If otherwise, there are monomials of the form $w_iw_{i'}w_{i''}$ and $w_jw_{j'}w_{j''}$ in the support of $f$.
We apply  the symmetric exchange property to the support of $f$, the monomials $w_iw_{i'}w_{i''}$, $w_jw_{j'}w_{j''}$, and the variable $w_i$:
We see that the monomial $w_jw_{i'}w_{i''}$ must be in the support of $f$, since no monomial in the support of $f$ is divisible by $w_iw_j$.
For a positive real parameter $s$, set
\[
h_s=\big( 1+s w_i \partial_{i'} \big) f.
\]
Since  $\partial_i\partial_{i'} f$ is not identically zero,
 the argument in the first paragraph shows  that $h_s$ is in $\mathrm{L}^3_n$.
 %because we still just have to check the case when $\alpha=e_i$ [CHANGED].
Similarly, since $\partial_i \partial_j h_s$ is not identically zero, we have
\[
\big( 1+\phi  w_i \partial_j \big) h_s \in \mathrm{L}^3_n \ \ \text{for every nonnegative real number $\phi$.}
\]
Since  stability is a closed condition, it follows that
\[
\text{$\lim_{s \to 0} \partial_i\Big( h_s +\phi   w_i \partial_j h_s\Big)= %  \partial_i\Big( f +\phi  \ w_i \partial_j f\Big)=
 \partial_i f + \phi\ \partial_j f$ is stable for every nonnegative real number $\phi$.} \qedhere
\]
%This completes the proof of Proposition \ref{deform}.
\end{proof}

We use Proposition \ref{deform} to show that any nonnegative linear change of variables preserves $\mathrm{L}^d_n$.

\begin{theorem}\label{flow}
If $f(w) \in \mathrm{L}^d_n$, then $f(Av)\in \mathrm{L}^d_m$ for any $n \times m$ matrix $A$ with nonnegative entries.
\end{theorem}

\begin{proof}
Fix $f=f(w_1,\ldots,w_n)$ in $\mathrm{L}^d_n$.
Note that Theorem \ref{flow} follows from its three special cases:
\begin{enumerate}\itemsep 5pt
\item[(I)]  the elementary splitting $f(w_1,\ldots,w_{n-1},w_n+w_{n+1})$ is in $\mathrm{L}^d_{n+1}$, 
\item[(II)]  the dilation $f(w_1,\ldots,w_{n-1},\theta w_n)$  is in $\mathrm{L}^d_n$ for any $\theta \ge 0$, 
\item[(III)] the diagonalization $f(w_1,\ldots,w_{n-2},w_{n-1},w_{n-1})$ is in $\mathrm{L}^d_{n-1}$,
\end{enumerate}
As observed in the proof of Lemma \ref{support}, an elementary splitting preserves $\mathrm{M}$-convexity:
\[
f(w_1,\ldots,w_{n-1},w_n+w_{n+1}) \in \mathrm{M}^d_{n+1}.
\] 
Therefore,\footnote{It is necessary to check the inclusion in $\mathrm{M}^d_{n+1}$ in advance because we have not yet proved that $\mathrm{L}^d_{n+1}$ is closed.} the first statement follows from Proposition \ref{deform}:
\[
\lim_{k \to \infty} \left(1+\frac {w_{n+1} \partial_n}{k}\right)^k \!\!\!\! f = f(w_1,\ldots, w_{n-1},w_n+w_{n+1})\in \mathrm{L}_{n+1}^d. 
\]

For the second statement, 
note from the definition of $\MM$-convexity that
\[
f(w_1,\ldots,w_{n-1},0) \in \mathrm{M}^d_n.
\]
Thus the second statement for $\theta=0$ follows from the case $\theta>0$, which is trivial to verify.

The proof of the third statement is similar to that of the first statement. 
As observed in the proof of Lemma \ref{support}, an elementary aggregation preserves  $\mathrm{M}$-convexity, and hence
\[
f(w_1,\ldots, w_{n-1},w_{n-1}+w_{n})\in\mathrm{M}^d_{n}.
\] 
Therefore, Proposition \ref{deform} implies that
\[
\lim_{k \to \infty} \left(1+\frac {w_{n-1} \partial_n}{k}\right)^k \!\!\!\! f = f(w_1,\ldots, w_{n-1},w_{n-1}+w_{n})\in \mathrm{L}_{n}^d. 
\]
By the second statement, we may substitute $w_n$ in the displayed equation by zero.
\end{proof}

Theorem \ref{flow} can be used to show that taking directional derivatives in nonnegative directions takes polynomials in $ \mathrm{L}^d_n$ to polynomials in $ \mathrm{L}^{d-1}_n$.

\begin{corollary}\label{derivatives}
If $f \in \mathrm{L}^d_n$, then $\sum_{i=1}^n a_i\partial_i f \in \mathrm{L}^{d-1}_n$ for any $a_1,\ldots,a_n \ge 0$.
\end{corollary}

\begin{proof}
We apply Theorem \ref{flow} to $f$ and the matrix with column vectors $e_1,\ldots,e_n$ and $\sum_{i=1}^n a_ie_i$:
\[
g\coloneqq f(w_1+a_1w_{n+1},\ldots,w_n+a_nw_{n+1}) \in \mathrm{L}^d_{n+1}, \ \text{and hence} \   \partial_{n+1}g \in \mathrm{L}^{d-1}_{n+1}.
\]
Applying Theorem \ref{flow} to $ \partial_{n+1}g$ and the matrix with column vectors $e_1,\ldots,e_n$ and $0$,  we get
\[
\partial_{n+1}g|_{w_{n+1}=0}=\sum_{i=1}^na_i\partial_if\in \mathrm{L}^{d-1}_{n}.\qedhere
\]
%This completes the proof of Corollary \ref{derivatives}.
\end{proof}

%Let $f$ be an element of $\mathrm{L}_n^d$, and 
Let $\theta$ be a nonnegative real parameter.
We define a  linear  operator $T_n(\theta,-)$ by
\[
T_{n} (\theta,f) = \Bigg( \prod_{i=1}^{n-1}\big(1+\theta w_i\partial_n \big)^{d} \Bigg) f.
\]
By Proposition \ref{deform},  if $f \in \mathrm{L}_n^d$, then $T_{n} (\theta,f) \in \mathrm{L}_n^d$.
In addition, if $f \in \mathrm{P}_n^d$, then $T_{n} (\theta,f) \in \mathrm{P}_n^d$.
Most importantly, the operator $T_n$ satisfies  the following Nuij-type homotopy lemma.
For a similar argument in the context of hyperbolic polynomials, see  the proof of the main theorem in \cite{Nuij}.

\begin{lemma}\label{split}
If $f \in \mathrm{L}_n^d \cap  \mathrm{P}_n^d$, then 
$T_n (\theta,f) \in \mathring{\mathrm{L}}_n^d$ for every positive real number $\theta$.
\end{lemma}

\begin{proof}
%By Theorem \ref{deform}, we have $T_n(\theta,f) \in \mathrm{L}_n^d \cap  \mathrm{P}_n^d$.
Let $e_i$ be the $i$-th standard unit vector in $\mathbb{R}^n$, and let
 $v$ be any vector  in $\mathbb{R}^n$ not parallel to $e_n$.
 From here on, in this proof, all polynomials are restricted to the line $xe_n-v$
and considered as univariate polynomials in $x$. 

Let $\alpha$ be an arbitrary element of $\Delta^{d-2}_n$.
By Lemma \ref{StrictlyStable}, it is enough to show that the quadratic polynomial $\partial^\alpha T_n(\theta,f)$ has two distinct real zeros.
Using Proposition \ref{deform}, we can deduce the preceding statement  from the following claims:
%For any $f \in \mathrm{L}^d_n$ and any index $i<n$, we have
 \begin{enumerate}\itemsep 5pt
\item[(I)] If $\partial^\alpha f$ has two distinct real zeros, then  $\partial^\alpha\big(1+\theta w_i\partial_n \big)f$ has two distinct zeros.
\item[(II)] If $v_i$ is nonzero, then $\partial^\alpha\big(1+\theta w_i\partial_n \big)^{d}f$ has two distinct real zeros.
\end{enumerate}
We first prove  (I). Suppose $\partial^\alpha f$ has two distinct real zeros, and set
$g=\big(1+\theta w_i\partial_n \big)f$. 
Note that
\[
\partial^\alpha g= \partial^\alpha f+\theta \alpha_i \partial^{\alpha-e_i+e_n} f+\theta w_i \partial^{\alpha +e_n}f.
\]
Let $c$ be the unique zero of $\partial^{\alpha+e_n} f$. 
Since $c$ strictly interlaces two distinct zeros of $\partial^\alpha f$, we have
\[
\partial^\alpha f|_{x=c} <0.
\]
Similarly, since $\partial^{\alpha-e_i+e_n} f$ has only real zeros and $\partial^{\alpha+e_n} f \prec  \partial^{\alpha-e_i+e_n} f$, we have
\[
 \partial^{\alpha-e_i+e_n} f|_{x=c} \le 0.
\]
Thus $\partial^\alpha g|_{x=c}<0$, and hence  $\partial^\alpha g$ has two distinct real zeros.
This completes the proof of (I).

Before proving (II), we strengthen (I) as follows:
 \begin{enumerate}\itemsep 5pt
\item[(III)] A multiple zero of $\partial^\alpha g$ is necessarily a multiple zero of $\partial^\alpha f$.
 \end{enumerate}
Suppose  $\partial^\alpha g$ has a multiple zero. 
Using (I), we know that $\partial^\alpha f$ has a multiple zero, say $c$.
Clearly, $c$ must be also a zero of $\partial^{\alpha+e_n} f$. 
Since $c$ interlaces the two (not necessarily distinct) zeros of $\partial^{\alpha-e_i+e_n} f$, we have
\[
\partial^\alpha g|_{x=c}= \theta \alpha_i \partial^{\alpha-e_i+e_n} f|_{x=c} \le 0.
\]
Therefore, if  $c$ is not a zero of $\partial^\alpha g$, then $\partial^\alpha g$ has two distinct zeros, contradicting the hypothesis that $\partial^\alpha g$ has a multiple zero.
This completes the proof of (III).

We prove (II).
Suppose  $\partial^\alpha \big(1+\theta w_i\partial_n \big)^{d}f$ has a multiple zero, say  $c$.
Using (III), we know that
\[
\text{the number $c$ is a multiple zero of $\partial^\alpha\big(1+\theta w_i\partial_n \big)^{k}f$ for all $0 \leq k \leq d$.}
\]
Expanding the $k$-th power and using the linearity of $\partial^\alpha$, we deduce that
\[
\text{the number $c$ is a zero of $\partial^\alpha w_i^k\partial_n^{k}\hspace{0.5mm} f$ for all $0 \le k \le d$.}
\]
However, since $f$ has positive coefficients,
the value of $\partial^\alpha w_i^{\alpha_i+2}\partial_n^{\alpha_i+2}\hspace{0.5mm} f$ at $c$  is a positive multiple of $v_i^2$,
and hence $v_i$ must be zero.
This completes the proof of (II).
\end{proof}

We use Lemma \ref{split} to prove the main result of this subsection.

\begin{theorem}\label{conn}
The  closure of $\mathring{\mathrm{L}}^d_n$  in $\mathrm{H}^d_n$ contains  $\mathrm{L}_n^d$. %Moreover, $\mathrm{L}_{n}^{d}$ is  in the closure of $\mathring{\mathrm{L}}_n^d$. 
\end{theorem}

\begin{proof}
Let $f$ be a polynomial in $\mathrm{L}_n^d$ that is not identically zero, and let $\theta$ be a real parameter satisfying 
$0 \le \theta \le 1$.
By Theorem \ref{flow}, we have
\[
S(\theta,f)\coloneqq \frac{1}{|f|_1}
 f\Big((1-\theta) w_1+\theta (w_1+\cdots+w_n) , \ldots, (1-\theta) w_n+\theta (w_1+\cdots+w_n)\Big)  \in \mathrm{L}_n^d,
\]
where $|f|_1$ is the sum of all coefficients of $f$.
Since $S(\theta,f)$  belongs to $\mathrm{P}_n^d$ when $0<\theta \le 1$, 
Lemma \ref{split} shows that
we have a homotopy 
\[
T_n\Big(\theta, S(\theta, f)\Big) \in \mathring{\mathrm{L}}^d_n,  \quad 0 < \theta \leq 1,
\]
that deforms $f$ to  the polynomial $T_n\big(1, (w_1+\cdots+w_n)^d\big)$.
%that belongs to $\mathring{\mathrm{L}}^d_n$ for $0< \theta \le 1$.
It follows that
 %$\mathring{\mathrm{L}}^d_n$ is contractible and
  the closure of  $\mathring{\mathrm{L}}^d_n$ in $\mathrm{H}^d_n$ contains  $\mathrm{L}_n^d$.
%Thus $\mathrm{L}_{n}^{d}$ is contained in the closure of $\mathring{\mathrm{L}}_n^d$, and both $\mathring{\mathrm{L}}^d_n$ and $\mathrm{L}_n^d$ are connected.
\end{proof}

We show in Theorem \ref{chars} that the closure of $\mathring{\mathrm{L}}^d_n$ in $\mathrm{H}^d_n$ is, in fact, equal to $\mathrm{L}^d_n$.

\subsection{Hodge--Riemann relations for Lorentzian polynomials}\label{secHR}

Let $f$ be a nonzero degree $d \ge 2$ homogeneous polynomial with nonnegative coefficients in variables $w_1,\ldots,w_n$.
The following proposition may be seen as an analog of the Hodge--Riemann relations for homogeneous stable polynomials.\footnote{We refer to \cite{Huh} for a survey of the Hodge--Riemann relations in combinatorial contexts.}

\begin{proposition}\label{hrstable}
If $f$ is  in $\mathrm{S}^d_n \setminus 0$, then $\HE_f(w)$ has exactly one positive eigenvalue for all  $w\in \mathbb{R}_{>0}^n$. 
Moreover, if $f$ is in the interior of $\mathrm{S}^d_n$, then  $\HE_f(w)$ is nonsingular for all  $w\in \mathbb{R}_{>0}^n$. 
\end{proposition}

\begin{proof}
Fix a vector $w\in \mathbb{R}_{>0}^n$. By Lemma \ref{StrictlyStable}, the Hessian of $f$ has exactly one positive eigenvalue at $w$ if and only if the following quadratic polynomial in $z$ is stable:
\[
z^T \mathscr{H}_f(w) z= \sum_{1 \le i,j \le n} z_iz_j \partial_i\partial_j f(w).
\]
The above is the quadratic part of the stable polynomial with nonnegative coefficients $f(z+w)$, and hence is stable by \cite[Lemma 4.16]{BBL}. 

Moreover, if $f$ is strictly stable, then $f_\epsilon = f \pm \epsilon(w_1^d +\cdots+w_n^d)$ is stable for all sufficiently small positive $\epsilon$. 
Therefore, by the result obtained in  the previous paragraph, the matrix 
\[
\HE_{f_\epsilon}(w)= \HE_{f}(w) \pm d(d-1)\epsilon \ \mathrm{diag}(w_1^{d-2},\ldots, w_n^{d-2})
\]
has exactly one positive eigenvalue for all sufficiently small positive $\epsilon$,
and hence $\HE_{f}(w)$ is nonsingular. 
\end{proof}

In Theorem \ref{HRTheorem}, we extend  the above result to Lorentzian polynomials.

\begin{lemma}\label{LefschetzInduction}
If $\mathscr{H}_{\partial_if}(w)$ has exactly one positive eigenvalue for every $i\in [n]$ and  $w \in \mathbb{R}^{n}_{>0}$,
then
\[
\ker \mathscr{H}_{f}(w)=\bigcap_{i=1}^{n} \ker \mathscr{H}_{\partial_if}(w) \ \ \text{for every $w \in \mathbb{R}^{n}_{>0}$.}
\]
\end{lemma}

\begin{proof}
We may suppose $d \ge 3$.
Fix $w \in \mathbb{R}^{n}_{>0}$, and write $ \mathscr{H}_{f}$ for $\mathscr{H}_{f}(w)$.
We will use Euler's formula for homogeneous functions:
\[
d \hspace{0.2mm} f=\sum_{i=1}^n w_i \hspace{0.5mm} \partial_i f.
\]
It follows that the Hessians of $f$ and $\partial_i f$ satisfy the relation
\[
(d-2) \hspace{0.5mm}\mathscr{H}_{f}=\sum_{i=1}^{n} w_i\hspace{0.5mm}\mathscr{H}_{\partial_if},
\]
and hence the kernel of $\mathscr{H}_f$ contains the intersection of the kernels of $\mathscr{H}_{\partial_if}$.

For the other inclusion, 
let $z$ be a vector in the kernel of $\mathscr{H}_f$. 
By Euler's formula again, 
 \[
(d-2) \hspace{0.5mm}e_i^T \mathscr{H}_f
= w^T\mathscr{H}_{\partial_i f} \ \ \text{for every $i \in [n]$,}
 \]
% where $e_i$ is the $i$-th standard basis vector in $\mathbb{R}^{n}$. 
and hence $w^T\mathscr{H}_{\partial_if} z=0$ for every $i \in [n]$.
We have $w^T  \mathscr{H}_{\partial_if}  w >0$
because $\partial_if$ is nonzero and has nonnegative coefficients.
Since $\mathscr{H}_{\partial_if}(w)$ has exactly one positive eigenvalue, it follows that $ \mathscr{H}_{\partial_if} $ is negative semidefinite on the kernel of $w^T \mathscr{H}_{\partial_if} $.
In particular, 
\[
\text{$z^T  \mathscr{H}_{\partial_if}  z \le 0$, with equality if and only if $ \mathscr{H}_{\partial_if} z=0$.}
\]
To conclude, we write zero as the positive linear combination
\[
0=(d-2)\Big(z^T  \mathscr{H}_{f}  z\Big) = \sum_{i =1}^{n} w_i \Big(z^T \mathscr{H}_{\partial_if} z\Big).
\]
Since every summand in the right-hand side is non-positive by the previous analysis, we must have $z^T\mathscr{H}_{\partial_if} z=0$ for every $i \in [n]$,
and hence $\mathscr{H}_{\partial_if} z=0$ for every $i \in. [n]$. 
\end{proof}

We  now  prove an analog of the Hodge--Riemann relation for Lorentzian polynomials. 
When $f$ is the volume polynomial of a projective variety as defined in Section \ref{SectionProjective}, then the one positive eigenvalue condition for the Hessian of $f$ at $w$ is equivalent to the validity of the Hodge--Riemann relations on the space of divisor classes of the projective variety with respect to the polarization corresponding to $w$.

\begin{theorem}\label{HRTheorem}
Let $f$ be a nonzero homogeneous polynomial in $\mathbb{R}[w_1,\ldots,w_n]$ of degree $d \ge 2$.
\begin{enumerate}[(1)]\itemsep 5pt
\item If $f$ is in $\mathring{\mathrm{L}}^d_n$, then  $\HE_f(w)$ is nonsingular for all  $w\in \mathbb{R}_{>0}^n$. 
\item If $f$ is  in $\mathrm{L}^d_n$, then $\HE_f(w)$ has exactly one positive eigenvalue for all  $w\in \mathbb{R}_{>0}^n$. 
\end{enumerate}
\end{theorem}

\begin{proof}
By Theorem \ref{conn}, $\mathrm{L}^d_n$ is in the closure of $\mathring{\mathrm{L}}^d_n$.
Note that, for any nonzero polynomial $f$ of degree $d \ge 2$ with nonnegative coefficients,
 $\HE_f(w)$ has at least one positive eigenvalue for any  $w\in \mathbb{R}_{>0}^n$. 
Therefore, we may suppose $f \in \mathring{\mathrm{L}}^d_n$ in (2).
We prove (1) and (2) simultaneously by induction on $d$ under this assumption. 
The base case $d=2$ is trivial. 
We suppose that $d \ge 3$ and that the theorem holds for $\mathring{\mathrm{L}}^{d-1}_n$.

That (1) holds for $\mathring{\mathrm{L}}^d_n$ follows from induction and Lemma \ref{LefschetzInduction}. 
Using Proposition \ref{hrstable}, we see that (2) holds for stable polynomials in $\mathring{\mathrm{L}}_n^d$. 
Since $\mathring{\mathrm{L}}^d_n$ is connected by Theorem \ref{conn}, the continuity of eigenvalues and the validity of (1) together implies (2).
\end{proof}

Theorem \ref{HRTheorem}, when combined with the following proposition, shows that 
all polynomials in $\mathrm{L}^d_n$ share a negative dependence property.
The negative dependence property will be systematically studied in the following section.

\begin{proposition}\label{RayleighLemma}
%If $f$ is a nonzero polynomial in $\mathrm{L}^d_n$, then% for any $w\in \mathbb{R}^n_{\geq 0}$  and $i,j \in [n]$, we have
If $\mathscr{H}_f(w)$ has exactly one positive eigenvalue for all $w\in \mathbb{R}^n_{> 0}$, then 
\[
 f(w)\hspace{0.5mm}\partial_i\partial_j f(w) \leq 2\Big(1-\frac{1}{d}\Big) \partial_i f(w)\hspace{0.5mm} \partial_j f(w) \ \ \text{for all $w\in \mathbb{R}^n_{\geq 0}$  and $i,j \in [n]$.}
\]%In particular, if $f$ is Lorentzian of degree $d$, then $f$ is $2(1-\frac{1}{d})$-Rayleigh.
\end{proposition}

\begin{proof}
Fix $w\in \mathbb{R}_{>0}^n$, and write $\mathscr{H}$ for $\mathscr{H}_f(w)$. 
By Euler's formula for homogeneous functions,
\[
w^T\mathscr{H}w = d(d-1)f(w) \ \ \text{and} \ \ w^T\mathscr{H}e_i= (d-1)\partial_i f(w).
\]
%where $e_i$ is the $i$-th standard unit vector in $\mathbb{R}^n$.
Let $t$ be a real parameter, and consider the restriction of $\mathscr{H}$ to the plane spanned by $w$ and $v_t=e_i+te_j$.
By Theorem \ref{HRTheorem}, $\mathscr{H}$ has exactly one positive eigenvalue.
Therefore, by Cauchy's interlacing theorem, the restriction of $\mathscr{H}$ also has exactly one positive eigenvalue.
In particular, the determinant of the restriction must be nonpositive:
\[
\left(w^T\mathscr{H} v_t \right)^2 - \left(w^T\mathscr{H}w \right) \cdot \left( v_t^T\mathscr{H} v_t \right) \ge 0  \ \ \text{for all $t \in \mathbb{R}$.}
\]
In other words, for all $t \in \mathbb{R}$, we have
\[
(d-1)^2 (\partial_i f+t \partial_j f)^2 -  d(d-1) f (\partial_i^2 f+ 2 t \partial_i\partial_j f + t^2\partial_j^2f) \ge 0.
\]
It follows that, for all $t \in \mathbb{R}$, we have
\[
(d-1)^2 (\partial_i f+t \partial_j f)^2 - 
2td (d-1)f   \partial_i\partial_j f \ge 0.
\]
Thus, the discriminant of the above quadratic polynomial in $t$ should be nonpositive:
\[
 f\partial_i\partial_j f - 2\Big(1-\frac{1}{d}\Big) \partial_i f \partial_j f \le 0. %\qedhere
\]
This completes the proof of Proposition \ref{RayleighLemma}.
\end{proof}

\subsection{Independence and negative dependence}\label{Negative}

Let $c$ be a fixed positive real number, and let $f$ be a polynomial in $\mathbb{R}[w_1,\ldots,w_n]$.
In this section, the polynomial $f$ is not necessarily homogeneous.
As before, we write $e_i$ for the $i$-th standard unit vector in $\mathbb{R}^n$.

\begin{definition}
We say that $f$ is $c$-\emph{Rayleigh} if $f$ has nonnegative coefficients and
\[
 \partial^\alpha f(w) \hspace{0.5mm}\partial^{\alpha+e_i+e_j} f(w) \leq c \hspace{0.5mm}\partial^{\alpha+e_i} f(w)  \hspace{0.5mm}\partial^{\alpha+e_j} f(w) \ \ \text{for all $i,j \in [n], \alpha \in \NN^n$, $w \in \RR_{\geq 0}^n$}. 
\]
\end{definition}

When $f$ is the partition function of a discrete probability measure $\mu$, the $c$-Rayleigh condition captures a negative dependence property of $\mu$.
More precisely, when $f$ is \emph{multi-affine}, that is, when $f$ has degree at most one in each variable, the $c$-Rayleigh condition for $f$ is equivalent to 
\[
 f(w) \hspace{0.5mm}\partial_i\partial_j  f(w)  \leq c\partial_i f(w)  \hspace{0.5mm}\partial_j  f(w) \ \ \text{for all distinct $i,j \in [n]$, and $w \in \RR_{\geq 0}^n$}. 
\]
Thus the $1$-Rayleigh property of multi-affine polynomials is equivalent to the Rayleigh property for discrete probability measures studied in  \cite{Wagner08} and \cite{BBL}.
%In Proposition \ref{HRRayleigh}, we prove that any  Lorentzian polynomial is $2$-Rayleigh.

\begin{proposition}\label{HRRayleigh}
Any polynomial in $\mathrm{L}^d_n$ is
  $2\Big(1-\frac{1}{d}\Big)$-Rayleigh.
\end{proposition}

%In particular, any polynomial in $\mathrm{L}^d_n$ is $2$-Rayleigh.

\begin{proof}
The statement follows from Theorem \ref{HRTheorem} and 
Proposition \ref{RayleighLemma}
% show that, for all  $i,j \in [n]$, $\alpha \in \mathbb{N}^n$,  $w\in \mathbb{R}^n_{\geq 0}$, we have
%\[
 %\partial^\alpha f(w)\hspace{0.5mm}\partial^{\alpha+e_i+e_j} f(w) \leq 2\Big(1-\frac{1}{d-|\alpha|_1}\Big) \partial^{\alpha+e_i} f(w)\hspace{0.5mm} \partial^{\alpha+e_j} f(w).
%\]
because  $2\left(1-\frac{1}{d}\right)$ is an increasing function of $d$.
\end{proof}

The goal of this section is to show that the support of any homogeneous $c$-Rayleigh polynomial  is $\MM$-convex (Theorem \ref{MRayleigh}).
The notion of $\MM^\natural$-convexity will be useful for the proof:
A subset $\mathrm{J}^\natural \subseteq \NN^n$ is said to be \emph{$\MM^\natural$-convex}
if there is an $\MM$-convex set $\mathrm{J}$ in $\NN^{n+1}$ such that 
\[
\mathrm{J}^\natural= \Big\{ (\alpha_1,\ldots, \alpha_n) \mid (\alpha_1,\ldots, \alpha_n,\alpha_{n+1}) \in \JJ\Big\}. 
\]
The projection from $\mathrm{J}$ to $\mathrm{J}^\natural$ should be bijective for any such $\mathrm{J}$, as 
the $\MM$-convexity of $\mathrm{J}$ implies that $\mathrm{J}$ is in $\Delta^d_n$ for some $d$.
We refer to \cite[Section 4.7]{Murota} for more on $\MM^\natural$-convex sets.
%Note that, in this paper, the empty subset of $\mathbb{N}^n$ is an $\MM^\natural$-convex set.
%In Theorem \ref{MRayleigh}, we prove  that the support of any homogeneous $c$-Rayleigh polynomial is $\MM$-convex. 
%The notion of $\MM^\natural$-convexity will be useful for the proof.
%This generalizes a result of Wagner \cite[Section 4]{Wagner08}, who proved the result for multi-affine polynomials and $c=1$. 

We prepare the proof of Theorem \ref{MRayleigh} with three lemmas.
Verification of the first lemma is routine and will be omitted.

\begin{lemma}\label{DeletionContraction}
The following polynomials are $c$-Rayleigh whenever $f$ is $c$-Rayleigh:
\begin{enumerate}[(1)]\itemsep 5pt
\item The \emph{contraction} $\partial_if$ of $f$.
\item The \emph{deletion} $f \setminus i$ of $f$, the polynomial obtained from $f$ by evaluating $w_i=0$.
%\item The \emph{inversion} $f^\vee$ of $f$, the polynomial obtained from $f$ by inverting all the variables:
%\[
%f^\vee(w_1,\ldots,w_n)=\Bigg(\prod_{i=1}^n w_i\Bigg) f(w_1^{-1},\ldots,w_n^{-1}).
%\]
%\item The \emph{diagonalization} $f(w_1,w_1,w_3,\ldots, w_n)$. 
\item[(4)] The \emph{dilation} $f(a_1w_1,\ldots,a_nw_n)$, for $(a_1,\ldots,a_n) \in \mathbb{R}^n_{\ge 0}$.
\item[(5)] The \emph{translation} $f(a_1+w_1,\ldots,a_n+w_n)$,  for $(a_1,\ldots,a_n) \in \mathbb{R}^n_{\ge 0}$.
\end{enumerate}
\end{lemma}

\begin{remark*}
Lemma \ref{DeletionContraction} in the previous version of this manuscript contained the following incorrect statement:
\begin{enumerate}[(3)]\itemsep 5pt
\item If $f(w_1,w_2,w_3,\ldots,w_n)$ is $c$-Rayleigh, then so is the \emph{diagonalization} $f(w_1,w_1,w_3,\ldots, w_n)$.
\end{enumerate}
This implies that the rank sequence of any Rayleigh measure is strongly log-concave, which is known to be not true \cite[Section 7, Counterexample 1]{BBL}.
%There is a multi-affine Rayleigh polynomial $f(w_1,w_2,\ldots,w_{20})$ such that the coefficients of $f(w_1,w_1,\ldots,w_1)$ fail to be strongly log-concave \cite[Section 7, Counterexample 1]{BBL}. 
In fact, the rank sequence of a Rayleigh measure need not even be unimodal \cite[Theorem 6]{KN10}.
Proof of Lemma \ref{PolymatroidRayleigh} below is revised accordingly.
\end{remark*}

We introduce a partial order $\le$ on $\NN^n$ by setting
\[
\alpha \le \beta  \Longleftrightarrow \text{$\alpha_i \le \beta_i$ for all $i \in [n]$}.
\]
We say that  a subset $\JJ^\natural$ of $\NN^n$ is \emph{interval convex}  if the following implication holds:
\[
\Big(\text{$\alpha \in \JJ^\natural$, $\beta \in \JJ^\natural$, $\alpha \leq \gamma \leq \beta$}\Big)  \Longrightarrow \gamma \in \JJ^\natural.
\]
The \emph{augmentation property} for $\JJ^\natural \subseteq \NN^n$  is the implication
\[
\Big(\text{$\alpha \in \JJ^\natural$, $\beta \in \JJ^\natural$,  $|\alpha|_1<|\beta|_1$}\Big) \Longrightarrow
\Big(\text{$\alpha_j<\beta_j$ and $\alpha +e_j \in \JJ^\natural$ for some $j \in [n]$}\Big).
\]

\begin{lemma}\label{polyind}
Let $\JJ^\natural$ be an interval convex subset of $\mathbb{N}^n$ containing $0$. 
Then $\JJ^\natural$ is $\MM^{\natural}$-convex if and only if  $\JJ^\natural$ satisfies the augmentation property.
\end{lemma}

Therefore, a nonempty interval convex subset of $\{0,1\}^n$ containing $0$  is $\MM^\natural$-convex 
if and only if it is the collection of independent sets of a matroid on $[n]$.

\begin{proof}
Let $d$ be any sufficiently large positive integer,  and set
\[
\mathrm{J}=\Big\{(\alpha_1,\ldots,\alpha_n,d-\alpha_1-\cdots-\alpha_n) \in \mathbb{N}^{n+1} \mid (\alpha_1,\ldots,\alpha_n) \in \mathrm{J}^\natural\Big\} .
\]
The ``only if'' direction is straightforward: If $\mathrm{J}^\natural$ is $\MM^\natural$-convex, then $\mathrm{J}$ is $\MM$-convex, and the augmentation property for $\mathrm{J}^\natural$ is a special case of the exchange property for $\mathrm{J}$.

We prove the ``if'' direction by checking the exchange property for $\mathrm{J}$.
Let $\alpha$ and $\beta$ be elements of  $\mathrm{J}$, and let $i$ be an index satisfying $\alpha_i >\beta_i$.
We claim that there is an index $j$ satisfying
\[
 \alpha_j <  \beta_j  \ \ \text{and} \ \  \alpha-e_i+e_j \in \mathrm{J}.
 \]
By the augmentation property for $\mathrm{J}^\natural$, it is enough to justify the claim when $i \neq n+1$.
When  $\alpha_{n+1} < \beta_{n+1}$, then we may take $j=n+1$,  %because $\alpha-e_i+e_{n+1} \in \mathrm{J}$ 
again by the augmentation property for $\MM^\natural$.

Suppose  $\alpha_{n+1} \ge  \beta_{n+1}$. 
In this case, we consider the element $\gamma=\alpha -e_i +e_{n+1}$. 
The element $\gamma$ belongs to $\mathrm{J}$, because $\mathrm{J}^\natural$ is 
an interval convex set containing $0$.
We have $\gamma_{n+1} > \beta_{n+1}$,  and hence
 the augmentation property for $\mathrm{J}^\natural$ gives
 an index $j$ satisfying 
 \[
\gamma_j< \beta_j \ \ \text{and} \ \ \alpha-e_i+e_j= \gamma -e_{n+1}+e_j \in \mathrm{J}.
 \]
 This index $j$ is necessarily different from $i$ because  $\alpha_i >\beta_i$.
It follows that $\alpha_j =\gamma_j< \beta_j$, and the $\MM$-convexity of $\JJ$ is proved.
\end{proof}

\begin{lemma}\label{PolymatroidRayleigh}
Let $f$ be a  $c$-Rayleigh polynomial in $\mathbb{R}[w_1,\ldots,w_n]$.
\begin{enumerate}[(1)]\itemsep 5pt
\item\label{PR1} The support of $f$  is interval convex.
%\item\label{PR3} Any two minimal elements of $\supp(f)$ have the same size.
%\item\label{PR2} Any two maximal elements of $\supp(f)$ have the same size.
\item\label{PR4} If $f(0)$ is nonzero, then $\supp(f)$ is $\MM^\natural$-convex.
\end{enumerate}
\end{lemma}

\begin{proof}
Suppose that the polynomial $f$ and the vectors $\alpha \le \gamma \le \beta$ constitute a minimal counterexample to (1) with respect to the degree and the number of variables of $f$.
We may and will suppose that $|\beta|_1$ is minimal among all such  $\alpha \le \gamma \le \beta$ for the polynomial $f$.
We have
\[
\alpha_j=0 \ \ \text{for all $j$,}
\]
since otherwise some contraction $\partial_j f$ is a smaller counterexample to (1). Similarly, we have
\[
\beta_j>0 \ \ \text{for all $j$,}
\]
since otherwise some deletion $f \setminus j$ is a smaller counterexample to (1). In addition, we may assume that $\gamma$ is a unit vector, say
\[
\gamma=e_i,
\]
since otherwise the contraction $\partial_j f$ for any $j$ satisfying $\gamma_j>0$ is a smaller counterexample to (1).
Suppose $e_j$ is in the support of $f$ for some $j$.
In this case, we should have
\[
e_i+e_j \in \textrm{supp}(f),
\]
since otherwise $\partial_j f$ is a smaller counterexample to (1). 
%{\red This is where we use the condition $\beta_j>0$ for all $j$.}
However, the above implies 
\[
\partial_i f(0)=0 \ \ \text{and} \ \ f(0) \partial_i\partial_j f(0)>0,
\]
contradicting the $c$-Rayleigh property of $f$. Therefore, no $e_j$ is in the support of $f$, 
and hence 
\[
|\delta|_1 \ge |\beta|_1 \ \ \text{for all $\delta \in \textrm{supp}(f)$,}
\]
by the minimality of $|\beta|_1$.
Thus, for any indices $k$ and $l$ satisfying $\beta \ge e_{k}+e_{l}$, 
we have
\begin{align*}
f(\epsilon,\ldots,\epsilon) &= a_1 +\text{higher order terms},\\
\partial_k f(\epsilon,\ldots,\epsilon)   &= a_2 \epsilon^{|\beta|_1-1}+\text{higher order terms},\\
\partial_l f(\epsilon,\ldots,\epsilon)   &=a_3 \epsilon^{|\beta|_1-1}+\text{higher order terms},\\
\partial_k \partial_l f(\epsilon,\ldots,\epsilon)   &= a_4 \epsilon^{|\beta|_1-2}+\text{higher order terms},
\end{align*}
for some positive constants $a_1,a_2,a_3,a_4$.
This contradicts the $c$-Rayleigh property of $f$ for sufficiently small positive $\epsilon$, proving (1).

Suppose $f$ is a counterexample to (2) that is minimal with respect to the degree and the number of variables of $f$.
By Lemma \ref{polyind} and (1) of the current lemma, we know that the support of $f$ fails to have the augmentation property.
In other words, there are $\alpha$ and $\beta$ in the support of $f$ such that $|\alpha|_1 < |\beta|_1$ and, for all $i$,
\[
\alpha_i < \beta_i \Longrightarrow \alpha+e_i \notin \textrm{supp}(f).
\]
We may and will suppose that $|\alpha|_1$ is minimal among all such  $\alpha$ and $\beta$ for the polynomial $f$.
For any $\gamma$, write $S(\gamma)$ for the set of indices $i$ such that $\gamma_i>0$.
If $i$ is in the intersection of $S(\alpha)$ and $S(\beta)$, then $\partial_i f$ is a counterexample to (2) that has degree less than that of $f$,
and hence
\[
S(\alpha) \cap S(\beta)=\varnothing.
\]
Similarly, if $i$ is not in the union of $S(\alpha)$ and $S(\beta)$, then $f \setminus i$ is a counterexample to (2) that involves less than $n$ variables,
and hence
\[
S(\alpha) \cup S(\beta)=[n].
\]
Since $\textrm{supp}(f)$ is interval convex by (1),  there is an index $i$ in $S(\alpha)$. In addition, since $f(0)$ is nonzero, we have %$\alpha-e_i$ must be in the support of $f$ for any such  $i$.
\[
\alpha-e_i \in \textrm{supp}(f).
\]
The vectors $\alpha-e_i$ and $\beta$ satisfy the augmentation property for $\textrm{supp}(f)$, by the minimality of $|\alpha|_1$.
Therefore, there is an index $j$ such that
\[
(\alpha-e_i)_j < \beta_j \ \ \text{and} \ \ \alpha-e_i+e_j \in \textrm{supp}(f).
\]
The first condition shows that  the index $j$ cannot be in $S(\alpha)$, so it must be in $S(\beta)$.
The vectors $\alpha-e_i+e_j$ and $\beta$ satisfy the augmentation property  for $\textrm{supp}(f)$, since otherwise $\partial_j f$ is a smaller counterexample to (2).
Therefore, there is an index $k$  such that
\[
(\alpha-e_i+e_j)_k < \beta_k \ \ \text{and} \ \  \alpha-e_i+e_j+e_k \in \textrm{supp}(f).
\]
The first condition shows that  the index $k$ cannot be in $S(\alpha)$, so it must be in $S(\beta)$.
On the other hand, since $j$ and $k$ are in $S(\beta)$ and not in $S(\alpha)$, the failure of the augmentation property for $\alpha$ and $\beta$ implies
\[
\alpha+e_j \notin \textrm{supp}(f) \ \ \text{and} \ \  \alpha+e_k \notin \textrm{supp}(f).
\]
Note that the $c$-Rayleigh polynomial $g\coloneq \partial^{\alpha-e_i} f$ satisfies
\[
e_i, e_j+e_k \in \textrm{supp}(g) \ \ \text{and} \ \ e_i+e_j,e_i+e_k \notin \textrm{supp}(g).
\]
The first pair of conditions shows that $g(w) \partial_j\partial_k g(w)$ is a strictly increasing function of $w_i$ if we set all the variables other than $w_i$ equal to $1$.
On the other hand, the second pair of conditions shows that $\partial_j g(w) \partial_k g(w)$ is independent of $w_i$, by  the interval convexity of $\textrm{supp}(g)$. 
This contradicts the $c$-Rayleigh property of $g$, proving (2).
\end{proof}

\begin{theorem}\label{MRayleigh}
If $f$ is homogeneous and $c$-Rayleigh, then the support of $f$  is $\MM$-convex.
\end{theorem}

\begin{proof}
%By Lemma \ref{PolymatroidRayleigh}, all the maximal monomials in the support of $f$ have the same degree, say $d$. 
%Similarly, all the minimal monomials in the support of $f$ have the same degree, say $d'$.
By (5) of Lemma \ref{DeletionContraction} and (2) of Lemma \ref{PolymatroidRayleigh}, the support of the translation 
\[
g(w_1,\ldots,w_n)=f(w_1+1, \ldots, w_n+1)
\]
 is $\MM^\natural$-convex.  
In other words, the support $\JJ$ of the homogenization of $g$ is $\MM$-convex. 
Since the intersection of an $\mathrm{M}$-convex set with a coordinate hyperplane is $\mathrm{M}$-convex, 
this implies the $\MM$-convexity of the support of $f$.
%\[
%\Big\{(\alpha_1,\ldots, \alpha_n,\alpha_{n+1}) \in \JJ \mid \alpha_{n+1}=0 \Big\}=\text{supp}(f). \qedhere
%\]
%By Lemma \ref{PolymatroidRayleigh}, the projection of $\JJ'$ to $\mathbb{N}^n$ is the support of $f$, and hence the support of $f$ is $\MM^\natural$-convex.
\end{proof}

A multi-affine polynomial $f$ is said to be  \emph{strongly Rayleigh} if
\[
 f(w) \hspace{0.5mm}\partial_i\partial_j  f(w)  \leq \partial_i f(w)  \hspace{0.5mm}\partial_j  f(w) \ \ \text{for all distinct $i,j \in [n]$, and $w \in \RR^n$}. 
\]
Clearly, any strongly Rayleigh multi-affine polynomial is $1$-Rayleigh.
Since a multi-affine polynomial is stable if and only if it is strongly Rayleigh \cite[Theorem 5.6]{Branden}, 
Theorem \ref{MRayleigh} extends the following theorem of Choe \emph{et al.} \cite[Theorem 7.1]{COSW}: If $f$ is a nonzero homogeneous stable multi-affine polynomial, then the support of $f$ is the set of bases of a matroid.

%Since $\mathrm{L}^d_n$ is the set of degree $d$ Lorentzian polynomials in $n$ variables,
%Proposition \ref{HRRayleigh} shows that 
% any degree $d$ Lorentzian polynomial is $2\Big(1-\frac{1}{d}\Big)$-Rayleigh.
Lastly, we show that the bound in  Proposition \ref{HRRayleigh} is optimal.

\begin{proposition}
When $n \le 2$, all polynomials in $\mathrm{L}^d_n$ are $1$-Rayleigh.
When $n \ge 3$, we have
\[
\Big(\text{all polynomials in $\mathrm{L}^d_n$ are $c$-Rayleigh}\Big) \Longrightarrow c \ge 2\Big(1-\frac{1}{d}\Big).
\]
In other words, for any $n \ge 3$ and any $c <2\Big(1-\frac{1}{d}\Big)$,  there is  $f\in \mathrm{L}^d_n$ that is not $c$-Rayleigh.
\end{proposition}

\begin{proof}
We first show  by induction that, for any homogeneous bivariate polynomial $f=f(w_1,w_2)$ with nonnegative coefficients, we have
\[
f(w) \hspace{0.5mm}\Big(\partial_1 \partial_2 f(w)\Big) \le \Big(\partial_1 f(w)\Big) \Big(\partial_2 f(w)\Big) \ \ \text{for any $w \in \mathbb{R}^2_{\ge 0}$.}
\]
We use the obvious fact that, for any homogeneous polynomial with nonnegative coefficients $h$,
\[
\big(\text{deg}(h)+1\big)  h \ge  \big(1+w_i \partial_i\big) h \ \ \text{for any $i \in [n]$ and $w \in \mathbb{R}^n_{\ge 0}$.}
\]
Since $f$ is bivariate, we may write $f=c_1w_1^d+c_2w_2^d+w_1w_2g$. We have%where $g$ is a homogeneous polynomial of degree $d-2$.
\begin{align*}
\partial_1 f\partial_2 f-f\partial_1 \partial_2 f&=d^2c_1c_2w_1^{d-1}w_2^{d-1}\\
&\quad +dc_1w_1^d(1+w_2\partial_2)g-c_1w_1^d(1+w_1\partial_1)(1+w_2\partial_2)g\\
&\quad +dc_2w_2^d(1+w_1\partial_1)g-c_2w_2^d(1+w_1\partial_1)(1+w_2\partial_2)g\\
&\quad +w_1w_2(g+w_1\partial_1 g)(g+w_2\partial_2 g)-w_1w_2 g(1+w_1\partial_1)(1+w_2\partial_2)g.
\end{align*} 
The summand in the second line  is nonnegative on $\mathbb{R}^2_{\ge 0}$ by the mentioned fact for
$(1+w_2\partial_2)g$.
The summand in the third line is nonnegative on $\mathbb{R}^2_{\ge 0}$  by the  mentioned fact for
$(1+w_1\partial_1)g$.
The summand in the fourth line is nonnegative  on $\mathbb{R}^2_{\ge 0}$ by the induction hypothesis applied to $g$.

We next show that, for any bivariate Lorentzian polynomial $f=f(w_1,w_2)$, we have
\[
f(w) \hspace{0.5mm}\Big(\partial_1 \partial_1 f(w)\Big) \le \Big(\partial_1 f(w)\Big) \Big(\partial_1 f(w)\Big) \ \ \text{for any $w \in \mathbb{R}^2_{\ge 0}$.}
\]
Since $f$ is homogeneous, it is enough to prove the inequality when $w_2=1$.
In this case, the inequality follows from the concavity of the function $\log f$ restricted to the line $w_2=1$.
This completes the proof that any bivariate Lorentzian polynomial is $1$-Rayleigh.

To see the second statement, consider the polynomial
\[
f=2\Big(1-\frac{1}{d}\Big)w_1^d+w_1^{d-1}w_2+w_1^{d-1}w_3+w_1^{d-2}w_2w_3.
\]
It is straightforward to check that $f$ is in $\mathrm{L}^d_n$.
If $f$ is $c$-Rayleigh, then, for any $w \in \mathbb{R}^n_{\ge 0}$,
\[
w_1^{2d-4}\Big(2\Big(1-\frac{1}{d}\Big)w_1^2+w_1w_2+w_1w_3+w_2w_3\Big) \le
c w_1^{2d-4}\Big(w_1+w_2\Big)
\Big(w_1+w_3\Big).
\]
The desired lower bound for $c$ is obtained by setting $w_1=1,w_2=0,w_3=0$.
\end{proof}

\subsection{Characterizations of Lorentzian polynomials}\label{secChar}

%For $0\leq k \leq n$, the $k$th \emph{elementary symmetric polynomial} in $w_1,\ldots, w_n$ is defined by 
%$$
%e_k(w_1,\ldots,w_n)= \sum_{1\leq i_1<\cdots <i_k \leq n} w_{i_1}\cdots w_{i_k}. 
%$$ 
%The \emph{polarization map}, $\mathcal{P}$, defined on polynomials in $\mathbb{R}[w_1,\ldots,w_n]$ of degree at most $d$, is defined by 
%$$
%\mathcal{P}(w_i^k) =  {e_k(w_{i1},\ldots, w_{id})} /{\binom d k},
%$$
%by linear extension. 
%\begin{theorem}\label{pol}
%$\mathcal{P} : \HRC_n^d \rightarrow \HRC_{dn}^d$. 
%\end{theorem}
%
%\begin{proof}
%We may assume $f \in \HR_n^d$ and thus that $f$ has full support. We prove that  $\mathcal{P}(f) \in \LL_{dn}^d$. Clearly $\mathrm{supp}(\mathcal{P}(f))$ is the set of bases of a matroid. 
%Since polarization preserves stability, see \cite{BBI}, it remains to prove that 
%$$
%\partial^{S_1} \cdots \partial^{S_n}\mathcal{P}(f) \big |_{w_{ij}=w_i}
%$$
%is stable or identically zero, where 
%$$
%\partial^{S_i}= \prod_{j=1}^{\alpha_i}{\frac \partial {\partial w_{ij}}},
%$$
%and $\alpha_1+\alpha_2+\cdots+\alpha_n=d-2$. However 
%$$
%\partial^{S_1} \cdots \partial^{S_n}\mathcal{P}(f) \big |_{w_{ij}=w_i}= 
%\frac {(d-\alpha_1)!}{d!} \cdots \frac {(d-\alpha_n)!}{d!} \partial^\alpha f,
%$$
%and the theorem follows. 
%\end{proof}
We may now give a complete and useful description of the space of Lorentzian polynomials. 
As before, we write $\mathrm{H}^d_n$ for the space of degree $d$ homogeneous polynomials in $n$ variables.
%In this section, let $f$ be a degree $d$ homogeneous polynomial in $\mathbb{R}[w_1,\ldots,w_n]$ with nonnegative coefficients.

\begin{theorem}\label{chars}
The closure of $\mathring{\mathrm{L}}^d_n$ in $\mathrm{H}^d_n$ is $\mathrm{L}_n^d$.
In particular,  $\mathrm{L}_n^d$ is a closed subset of $\mathrm{H}^d_n$. 
%In other words,  $\mathrm{L}_n^d$ is the set of degree $d$ Lorentzian polynomials in $n$ variables.
\end{theorem}

\begin{proof}
By Theorem \ref{conn}, the closure of $\mathring{\mathrm{L}}^d_n$ contains $\mathrm{L}^d_n$.
%For the other inclusion, it is enough to show that the support of any Lorentzian polynomial is $\MM$-convex.
Since any limit of $c$-Rayleigh polynomials must be $c$-Rayleigh,
the other inclusion follows from Theorem \ref{MRayleigh} and Proposition \ref{HRRayleigh}.
\end{proof}

Therefore, a degree $d$ homogeneous polynomial $f$ with nonnegative coefficients is Lorentzian if and only if 
 the support of $f$ is $\MM$-convex and  
 $\partial^\alpha f$ has at most one positive eigenvalue for every $\alpha \in \Delta^{d-2}_n$. 
 In other words, Definitions \ref{FirstDefinition} and \ref{SecondDefinition} define the same class of polynomials. 
%\begin{enumerate}[(1)]\itemsep 5pt
%\item the support of $f$ is $\MM$-convex, and  
%\item $\partial^\alpha f$ has at most one positive eigenvalue for every $\alpha \in \NN^n$ with $|\alpha|_1=d-2$. 
%\end{enumerate}

\begin{example}\label{ExampleBivariate}
A sequence of  nonnegative numbers $a_0,a_1,\ldots,a_d$ is said to be \emph{ultra log-concave} if 
\[
 \frac {a_{k}^2} {{\binom d {k}}^2} \geq \frac {a_{k-1}} {\binom d {k-1}} \frac {a_{k+1}} {\binom d {k+1}} \quad \text{for all $0<k<d$}.% \ \ 
% \text{and} \ \ \text{$a_{k_1}a_{k_3}>0 \Longrightarrow a_{k_2}>0$
%for $0\leq k_1 < k_2 < k_3 \leq d$.}
\]
The sequence is said to have \emph{no internal zeros} if
\[
% \frac {a_{k}^2} {{\binom d {k}}^2} \geq \frac {a_{k-1}} {\binom d {k-1}} \frac {a_{k+1}} {\binom d {k+1}} \ \ \text{for $0<k<d$} \ \ 
 %\text{and} \ \ 
 a_{k_1}a_{k_3}>0 \Longrightarrow a_{k_2}>0
 \quad \text{for all $0\leq k_1 < k_2 < k_3 \leq d$.}
\]
Recall from Example \ref{n2} that a bivariate homogeneous polynomial $\sum_{k=0}^d a_k w_1^kw_2^{d-k}$  is strictly Lorentzian if and only if 
the sequence $a_k$ is positive and strictly ultra log-concave.
Theorem \ref{chars} says that, in this case,
 the polynomial $\sum_{k=0}^d a_k w_1^kw_2^{d-k}$   is Lorentzian if and only if 
the sequence $a_k$ is nonnegative, ultra log-concave, and has no internal zeros.
\end{example}

\begin{example}
Using Theorem \ref{chars}, it is straightforward to check that elementary symmetric polynomials are Lorentzian.
In fact, one can show more generally that all normalized Schur polynomials are Lorentzian \cite[Theorem 3]{HMMS}.
Any elementary symmetric polynomial is stable \cite[Theorem 9.1]{COSW}, but a normalized Schur polynomial need not be stable \cite[Example 9]{HMMS}.
\end{example}

%\begin{remark}\label{Remark-Stratification}
Let $\mathbb{P}\mathrm{H}^d_n$ be the projectivization of  $\mathrm{H}^d_n$ equipped with the quotient topology.
The image  $\mathbb{P}\mathrm{L}^d_n$  of $\mathrm{L}^d_n$  in the projective space 
 is homeomorphic to the intersection of $\mathrm{L}^d_n$ with the unit sphere in $\mathrm{H}^d_n$ for the Euclidean norm on the coefficients.
%Similarly,
%the image  $\mathbb{P}\mathring{\mathrm{L}}^d_n$  of $\mathring{\mathrm{L}}^d_n$  
% is homeomorphic to the intersection of $\mathring{\mathrm{L}}^d_n$ with the unit sphere.
 
\begin{theorem}\label{BallLike}
The space $\mathbb{P}\mathrm{L}^d_n$ is compact and contractible. %and has contractible interior  $\mathbb{P}\mathring{\mathrm{L}}^d_n$.
\end{theorem}

\begin{proof}
Since $\mathbb{P}\mathrm{H}^d_n$ is compact, Theorem \ref{chars} implies that $\mathbb{P}\mathrm{L}^d_n$ is compact.
A deformation retract of
 $\mathbb{P}\mathrm{L}^d_n$ can be constructed using Theorem \ref{flow}.%follows from the homotopy constructed in the proof of Theorem \ref{conn}.
\end{proof}

We conjecture that $\mathbb{P}\mathrm{L}^d_n$ is homeomorphic to a familiar space.

\begin{conjecture}\label{Ball}
The space $\mathbb{P}\mathrm{L}^d_n$ is homeomorphic to a closed Euclidean ball.
\end{conjecture}

For other appearances of stratified Euclidean balls in the interface of analysis of combinatorics, see \cite{GKL1,GKL2} and references therein.
Prominent examples are the totally nonnegative parts of Grassmannian and other partial flag varieties.
%which one may view as a subset of $\mathbb{P}\mathrm{L}^d_n$.
%and let $\mathrm{L}_\mathrm{J}$ be the set of polynomials in $\mathrm{L}^d_n$ with nonempty support $\mathrm{J}$.
%Writing  $\mathbb{P}\mathrm{L}^d_n$, $\mathbb{P}\mathring{\mathrm{L}}^d_n$, and $\mathbb{P}\mathrm{L}_\mathrm{J}$ for the images of $\mathrm{L}^d_n$, $\mathring{\mathrm{L}}^d_n$, and $\mathrm{L}_\mathrm{J}$ in $\mathbb{P}\mathrm{H}^d_n$ respectively, we have
%\[
%\mathbb{P}\mathrm{L}^d_n = \coprod_{\mathrm{J}} \mathbb{P}\mathrm{L}_\mathrm{J}, %\quad \text{where $\mathrm{L}_\mathrm{J} =  \big\{f \in \mathrm{L}^d_n \mid \text{supp}(f)=\mathrm{J}\big\}$},
%\]
%where $ \mathrm{L}_\mathrm{J} =  \big\{f \in \mathrm{L}^d_n \mid \text{supp}(f)=\mathrm{J}\big\}$ 
%where the union is over all nonempty $\MM$-convex subsets of $\Delta^d_n$.
%By Theorems \ref{flow} and \ref{chars}, 
%$\mathbb{P}\mathrm{L}^d_n$ is a compact contractible subset of $\mathbb{P}\mathrm{H}^d_n$ with contractible interior $\mathbb{P}\mathring{\mathrm{L}}^d_n$.
%By Theorem \ref{charjump}, $\mathbb{P}\mathrm{L}_\mathrm{J}$ is nonempty for every nonempty $\MM$-convex subset $\mathrm{J}$ of $\Delta^d_n$.
%In addition, by Proposition \ref{tot},
%$\mathbb{P}\mathrm{L}_\mathrm{J}$ is contractible for every nonempty $\MM$-convex subset $\mathrm{J}$ of $\Delta^d_n$.
%It would be interesting to study the boundary structure of $\mathbb{P}\mathrm{L}^d_n$ in greater detail.
%\end{remark}

Let $f$ be a polynomial in $n$ variables with nonnegative coefficients.
In \cite{GurvitsL}, Gurvits defines $f$ to be \emph{strongly log-concave} if, for all $\alpha \in \NN^n$,
\[
\text{$\partial^\alpha f$ is identically zero or $\log (\partial^\alpha f)$ is concave on $\RR_{> 0}^n$.}
 \]
 In \cite{AGV},  Anari \emph{et al.} define  $f$ to be \emph{completely log-concave} if, for all $m \in \NN$ and any $m\times n$ matrix $(a_{ij})$ with nonnegative entries,
 \[
 \text{$\Big(\prod_{i=1}^m D_{i}\Big) f$ is identically zero or $\log\Big(\Big(\prod_{i=1}^m D_{i}\Big) f\Big)$ is concave on $\RR_{> 0}^n$,}
\]
where $D_i$ is the differential operator $\sum_{j=1}^n a_{ij} \partial_j$.
We show that the two notions agree with each other and with the Lorentzian property for homogeneous polynomials.\footnote{An implication similar to $(3) \Rightarrow (1)$ of Theorem \ref{allequal} can be found in \cite[Theorem 3.2]{ALGVIII}.}

\begin{theorem}\label{allequal}
The following conditions are equivalent for any homogeneous polynomial $f$.
\begin{enumerate}[(1)]\itemsep 5pt
\item $f$ is completely log-concave.
\item $f$ is strongly log-concave.
\item $f$ is Lorentzian.
\end{enumerate}
\end{theorem}

%An implication similar to that of $(3) \Rightarrow (1)$ in Theorem \ref{allequal}, in terms of  appears in \cite{ALGVII} and \cite{ALGVIII}.
%By Theorem \ref{chars}, 
The support of any Lorentzian polynomial is $\MM$-convex by Theorem \ref{chars}.
Thus, by Theorem \ref{allequal}, the same  holds for any strongly log-concave homogeneous polynomial.
%Combining Theorems \ref{chars} and \ref{allequal}, we see that the support of any strongly log-concave homogeneous polynomial is $\MM$-convex.
This answers a question of Gurvits \cite[Section 4.5 (iii)]{GurvitsL}.
% Gurvits asked a question which is equivalent to asking if the support of a strongly log-concave homogeneous polynomial always  is $\MM$-convex. 
%Theorems \ref{chars} and \ref{allequal} answer this question to the affirmative. 

\begin{corollary}\label{CorollarySupport}
The support of any strongly log-concave homogeneous polynomial is $\MM$-convex.
\end{corollary}

%Combining Theorems \ref{allequal} and \ref{flow}, we see that the product of strongly log-concave homogeneous polynomials is strongly log-concave.
Similarly, we can use Theorem \ref{allequal} to show that the class of strongly log-concave homogeneous polynomials is closed under multiplication.
This answers another question of Gurvits \cite[Section 4.5 (iv)]{GurvitsL} for homogeneous polynomials.
% Gurvits asked a question which is equivalent to asking if the support of a strongly log-concave homogeneous polynomial always  is $\MM$-convex. 
%Theorems \ref{chars} and \ref{allequal} answer this question to the affirmative. 

\begin{corollary}\label{CorollaryProduct}
The product of strongly log-concave homogeneous polynomials  is strongly log-concave.
\end{corollary}

\begin{proof}
Let $f(w)$ be an element of $\mathrm{L}^{d}_n$, and let $g(w)$ be an element of $\mathrm{L}^{e}_n$.
%It is enough to show that  $f(w)g(w) \in \mathrm{L}^{d+e}_n$.
It is straightforward to check that $f(w)g(u)$ is an element of  $\mathrm{L}^{d+e}_{n+n}$,  where $u$ is a set of variables disjoint from $w$.
It follows that $f(w)g(w)$ is an element of $\mathrm{L}^{d+e}_n$, since setting $u=w$ preserves the Lorentzian property by Theorem \ref{flow}.
\end{proof}

Corollary \ref{CorollaryProduct} extends the following theorem of Liggett \cite[Theorem 2]{Liggett}:
The convolution product of two ultra log-concave sequences with no internal zeros is an ultra log-concave sequence with no internal zeros.
%This recovers a theorem of Liggett \cite[Theorem 2]{Liggett}. 

%On the other hand, if $f(w)g(w)$ is Lorentzian, it does not follow that $f$ and $g$ are. Take 
%$f=x^2+2axy+y^2$ and $g=x^2+2bxy+y^2$, where $a=2$ and $b=4/5$.  

To prove Theorem \ref{allequal}, we use the following elementary observation.
Let $f$ be a homogeneous polynomial in $n \ge 2$ variables of degree $d \ge 2$.

\begin{proposition}\label{logcon}
 The following are equivalent for any $w \in \mathbb{R}^n$ satisfying $f(w)>0$.
\begin{enumerate}[(1)]\itemsep 5pt
\item The Hessian of $f^{1/d}$  is negative semidefinite at $w$.
\item The Hessian of $\log f$ is negative semidefinite at $w$.
\item The Hessian of $f$ has exactly one positive eigenvalue at $w$.
\end{enumerate}
\end{proposition}

The equivalence of (2) and (3) appears in \cite{AGV}.

\begin{proof}
We fix $w$ throughout the proof. 
For $n \times n$ symmetric matrices $\mathrm{A}$ and $\mathrm{B}$, we write $\mathrm{A} \prec \mathrm{B}$ to mean
 the following interlacing relationship between the eigenvalues of $\mathrm{A}$ and $\mathrm{B}$:
\[
\lambda_1(\mathrm{A}) \leq \lambda_1(\mathrm{B}) \leq \lambda_2(\mathrm{A}) \leq \lambda_2(\mathrm{B}) \leq \cdots \leq \lambda_n(\mathrm{A}) \leq \lambda_n(\mathrm{B}).
\]
Let $\mathscr{H}_1$, $\mathscr{H}_2$, and $\mathscr{H}_3$ for the Hessians of $f^{1/d}$, $\log f$, and $f$, respectively.
We have
\[
df^{-1/d} \mathscr{H}_1 = \mathscr{H}_2 +\frac 1 d f^{-2}(\text{grad} f)(\text{grad} f)^T \ \ \text{and} \ \ \mathscr{H}_2= f^{-1} \mathscr{H}_3 -f^{-2} (\text{grad} f)(\text{grad} f)^T.
\]
 %where $\text{grad} f$ is the gradient of $f$. %Since $f(w)>0$, we have $(\nabla f)(w) \neq 0$ by Euler's formula for homogenous functions. 
Since $(\text{grad} f)(\text{grad} f)^T$ is positive semidefinite of rank one,
Weyl's inequalities for Hermitian matrices  \cite[Theorem 6.3]{Serre} show that
\[
\mathscr{H}_2 \prec \mathscr{H}_1 \ \ \text{and} \ \ \mathscr{H}_2 \prec \mathscr{H}_3 \ \ \text{and} \ \ \mathscr{H}_1 \prec \mathscr{H}_3.
\]
Since $w^T \mathscr{H}_3 w=d(d-1)f$, $\mathscr{H}_3$ has at least one positive eigenvalue, and hence (1) $\Rightarrow$ (2) $\Rightarrow$ (3).

For (3) $\Rightarrow$ (1), suppose that $\mathscr{H}_3$ has exactly one positive eigenvalue. 
We introduce a positive real parameter $\epsilon$ and consider the polynomial
\[
f_\epsilon=f-\epsilon(w_1^d+\cdots+w_n^d).
\]
We write $\mathscr{H}_{3,\epsilon}$ for the Hessian of $f_\epsilon$, and write $\mathscr{H}_{1,\epsilon}$ for the Hessian of $f_\epsilon^{1/d}$.

Note that $\mathscr{H}_{3,\epsilon}$ is nonsingular and has exactly one positive eigenvalue for all sufficiently small positive $\epsilon$.
In addition, we have $\mathscr{H}_{1,\epsilon} \prec \mathscr{H}_{3,\epsilon}$, and hence $\mathscr{H}_{1,\epsilon}$ has at most one nonnegative eigenvalue for all sufficiently small positive $\epsilon$.
However, by Euler's formula for homogeneous functions, we have
 \[
 \mathscr{H}_{1,\epsilon} w =0, 
 \]
 so that $0$ is the only nonnegative eigenvalue of $ \mathscr{H}_{1,\epsilon}$ for any such $\epsilon$.
The implication 
(3) $\Rightarrow$ (1) now follows by limiting $\epsilon$ to $0$. 
\end{proof}

It follows that, for any nonzero degree $d \ge 2$ homogeneous polynomial $f$ with nonnegative coefficients,
the following conditions are equivalent:
\begin{enumerate}[--]\itemsep 5pt
\item The function $f^{1/d}$ is concave on $\mathbb{R}^n_{>0}$.
\item The function $\log f$ is concave on $\mathbb{R}^n_{>0}$.
\item The Hessian of $f$ has exactly one positive eigenvalue on $\mathbb{R}^n_{>0}$.
\end{enumerate}

\begin{proof}[Proof of Theorem \ref{allequal}]
We may suppose that $f$ has degree $d \ge 2$.
Clearly, completely log-concave polynomials are strongly log-concave. 

Suppose $f$ is a strongly log-concave homogeneous polynomial of degree $d$. 
By Proposition \ref{logcon}, either $\partial^\alpha f$ is identically zero or the Hessian of $\partial^\alpha f$ has exactly one positive eigenvalue on $\mathbb{R}^n_{>0}$
for all $\alpha \in \NN^n$.
%whenever $|\alpha|_1 \le d-2$.
By Proposition \ref{RayleighLemma}, $f$ is $2\big(1-\frac{1}{d}\big)$-Rayleigh, and hence, by Theorem \ref{MRayleigh}, the support of $f$ is $\MM$-convex. 
Therefore, by Theorem \ref{chars},  $f$ is Lorentzian. 

Suppose $f$ is a nonzero Lorentzian polynomial.
Theorem \ref{HRTheorem} and Proposition \ref{logcon} together show that
 $\log f$ is concave on $\mathbb{R}^n_{>0}$. 
Therefore, it is enough to prove that $\big(\sum_{i=1}^n a_i\partial_i\big)f$ is Lorentzian
 for any nonnegative numbers $a_1,\ldots,a_n$.
This is a direct consequence  of Theorem \ref{chars} and Corollary \ref{derivatives}.
\end{proof}

\section{Advanced theory}

\subsection{Linear operators preserving Lorentzian polynomials}\label{secOp}

We describe a large class of linear operators that preserve the Lorentzian property. 
An analog was achieved for the class of stable polynomials in \cite[Theorem 2.2]{BBI}, where the linear operators preserving stability were characterized.
% If $T : A \to B$, where $A$ and $B$ are linear spaces of real polynomials, we say that $T$ preserves the \emph{Lorentzian property} if $T(f)$ or $-T(f)$ is Lorentzian whenever $f \in A$ is Lorentzian.  
%\subsubsection{}\label{polarization}
For an element $\kappa$ of $\NN^n$,
we set
\begin{align*}
\RR_\kappa[w_i]&= \Big\{\text{polynomials in  $\RR[w_i]_{1 \le i \le n}$ of degree at most $\kappa_i$ in $w_i$ for every $i$}\Big\}, \\
\RR_\kappa^\textrm{a}[w_{ij}]&= \Big\{\text{multi-affine polynomials in $\RR[w_{ij}]_{ 1\leq i \leq n, 1\leq j \leq \kappa_i}$}\Big\}.
\end{align*}
The \emph{projection operator} $\Pi^\downarrow_\kappa: \RR_\kappa^\textrm{a}[w_{ij}] \to \RR_\kappa[w_i]$  is the linear map that substitutes each $w_{ij}$ by $w_i$:
\[
\Pi^\downarrow_\kappa(g) = g|_{w_{ij}=w_i}.
\]
The \emph{polarization operator} $\Pi^\uparrow_\kappa: \RR_\kappa[w_i] \to \RR_\kappa^\textrm{a}[w_{ij}]$ is the linear map that sends $w^\alpha$ to the product
%unique linear map satisfying
%\begin{enumerate}[--]\itemsep 5pt
%\item for every $f$ and  $i$, the polynomial $\Pi^\uparrow_\kappa(f)$ is symmetric in the variables $\{w_{ij}\}_{1 \le j \le \kappa_i}$, and
%\item for every $f$, the polynomial $\Pi^\downarrow_\kappa \circ \Pi^\uparrow_\kappa(f)$ is equal to $f$.
%\end{enumerate}
%Explicitly, $\Pi^\uparrow_\kappa$  is the linear map that sends a monomial $w^\alpha$ to the product
\[
%\Pi^\uparrow_\kappa(w^\alpha )=
 \frac{1}{{\kappa \choose \alpha}} \prod_{i=1}^n \big(\text{elementary symmetric polynomial of degree $\alpha_i$ in the variables $\{w_{ij}\}_{1\le j \le \kappa_i}$}\big),
\]
where $\binom \kappa \alpha$ stands for the product of binomial coefficients $\prod_{i=1}^n \binom {\kappa_i} {\alpha_i}$.  
Note that 
\begin{enumerate}[--]\itemsep 5pt
\item for every $f$, we have $\Pi^\downarrow_\kappa \circ \Pi^\uparrow_\kappa(f)=f$, and
\item for every $f$ and every $i$,  the polynomial $\Pi^\uparrow_\kappa(f)$ is symmetric in the variables $\{w_{ij}\}_{1 \le j \le \kappa_i}$.
\end{enumerate}
The above  properties characterize $\Pi^\uparrow_\kappa$ among the linear operators from $\RR_\kappa[w_i]$ to $\RR_\kappa^\textrm{a}[w_{ij}]$.

%$\Pi^\downarrow_\kappa \circ \Pi^\uparrow_\kappa$  is the identity operator on $\RR_\kappa[w_i]$.

%for any $f$ and $i$, the polynomial $\Pi^\uparrow_\kappa(f)$ is symmetric in the variables $\{w_{ij}\}_{1 \le j \le \kappa_i}$.

\begin{proposition}\label{pols}
The operators $\Pi_\kappa^{\downarrow}$ and $\Pi_\kappa^{\uparrow}$ preserve the Lorentzian property.
\end{proposition}

In other words, $\Pi_\kappa^{\uparrow}(f)$ is a Lorentzian polynomial for any Lorentzian polynomial $f \in \mathbb{R}_\kappa[w_i]$,
and  $\Pi_\kappa^{\downarrow}(g)$ is a Lorentzian polynomial for any Lorentzian polynomial $g \in \mathbb{R}_\kappa^\textrm{a}[w_{ij}]$.

\begin{proof}
 The statement for $\Pi_\kappa^{\downarrow}$ follows from Theorem \ref{flow}. 
We  prove the statement for $\Pi_\kappa^{\uparrow}$. 
It is enough to prove that
$\Pi_\kappa^{\uparrow}(f)$ is Lorentzian when $f \in \mathring{\mathrm{L}}^d_n \cap \mathbb{R}_\kappa[w_i]$ for $d \ge 2$.

Set $k=|\kappa|_1$, and identify $\NN^k$ with the set of all monomials in $w_{ij}$.
Since $f \in \mathring{\mathrm{L}}^d_n$,  we have
\[
\text{supp}\big(\Pi_\kappa^{\uparrow}(f)\big)={k \brack d},
\]
which is clearly $\MM$-convex.
Therefore, by Theorem \ref{chars}, it remains to show that the quadratic form $\partial^\beta \Pi_\kappa^{\uparrow}(f)$ is stable for any $\beta \in {k \brack d-2}$.

Define $\alpha$ by the equality 
$\Pi^\downarrow_\kappa(w^\beta)=w^\alpha$.
Note that, after renaming the variables if necessary,
the $\beta$-th partial derivative of  $\Pi^\uparrow_\kappa (f)$
is a positive  multiple of a polarization of the $\alpha$-th partial derivative of $f$:
\[
\partial^\beta \Pi_\kappa^{\uparrow}(f)=\frac{(\kappa-\alpha)!}{\kappa!}\Pi^\uparrow_{\kappa-\alpha} (\partial^\alpha f).% \  \text{after renaming the variables, if necessary.}
\]
Since the operator $\Pi^\uparrow_{\kappa-\alpha} $ preserves stability \cite[Proposition 3.4]{BBI}, the conclusion follows from the stability of the quadratic form $\partial^\alpha f$.
\end{proof}

Let $\kappa$ be an element of $\NN^n$,  let $\gamma$ be an element of $\NN^m$, 
and  set $k=|\kappa|_1$. 
In the remainder of this section, we  fix a linear operator
\[
T:\mathbb{R}_\kappa[w_i] \rightarrow \mathbb{R}_\gamma[w_i],
\]
and suppose that the linear operator $T$   is \emph{homogeneous of degree $\ell$} for some $\ell \in \mathbb{Z}$: 
\[
\Big(\text{$0 \le \alpha \le \kappa$  and $T(w^\alpha) \neq 0$}\Big)
\Longrightarrow  \deg T(w^\alpha)=\deg w^\alpha+\ell.
\]
The \emph{symbol} of 
$T$
 is a homogeneous polynomial of degree $k+\ell$  in $m+n$ variables defined by
\[
\text{sym}_T(w,u)%= T\Big( (w_1+u_1)^{\kappa_1} \cdots (w_n+u_n)^{\kappa_n} \Big) 
= \sum_{0\leq \alpha \leq \kappa}\binom \kappa \alpha T(w^\alpha) u^{\kappa-\alpha}.
\] 
We show that the homogeneous  operator $T$ preserves the Lorentzian property if its symbol $\text{sym}_T$ is Lorentzian.
%where $\binom \kappa \alpha$ is the product of binomial coefficients $\prod_{i=1}^n \binom {\kappa_i} {\alpha_i}$.  
%In \cite{BBI} linear operators preserving stability were characterized in terms of the symbol. Here we will prove an analogous theorem for Lorentzian polynomials. We prepare for the proof of the characterization with a lemma. 

\begin{theorem}\label{hrcpr}
If $\text{sym}_T \in  \mathrm{L}^{k+\ell}_{m+n}$ and $f \in \mathrm{L}^d_n \cap \mathbb{R}_\kappa[w_i]$, then $T(f) \in \mathrm{L}^{d+\ell}_m$. 
\end{theorem}

When $n=2$, Theorem \ref{hrcpr} provides a large class of linear operators that preserve the ultra log-concavity of sequences of nonnegative numbers with no internal zeros. 
We prepare the proof of Theorem \ref{hrcpr} with a special case.

\begin{lemma}\label{LS-lemma}
Let $T=T_{w_1,w_2}: \mathbb{R}_{(1,\ldots,1)}[w_i] \to  \mathbb{R}_{(1,\ldots,1)}[w_i] $ be the linear operator defined by
\[
T(w^S)=
\left\{\begin{array}{cc}
w^{S \setminus 1}&\text{if $1 \in S$ and $2 \in S$,}\\
w^{S \setminus 1}&\text{if $1 \in S$ and $2 \notin S$,}\\
w^{S \setminus 2}&\text{if $1 \notin S$ and $2 \in S$,}\\
0&\text{if $1 \notin S$ and $2 \notin S$,}
\end{array}
\right.
\ \  \text{for all $S \subseteq [n]$.}
\]
Then $T$ preserves the Lorentzian property.
\end{lemma}

\begin{proof}
It is enough to prove that
$T(f) \in \mathrm{L}^{d}_{n}$  when $f \in \mathring{\mathrm{L}}^{d+1}_n \cap \mathbb{R}_{(1,\ldots,1)}[w_i]$ for $d \ge 2$.
In this case,   
\[
\text{supp}\big(T(f)\big)=\big\{\text{$d$-element subsets of $[n]$ not containing $1$}\big\},
\]
which is clearly $\MM$-convex.
Therefore, by Theorem \ref{chars}, it suffices to show that the quadratic form $\partial^S T(f)$ is stable for any $S \in {n \brack d-2}$ not containing $1$.
We write $h$ for the Lorentzian polynomial $f|_{w_1=0}$.
Since $f$ is multi-affine, we have 
\[
f=h+w_1 \partial_1 f
\ \ \text{and} \ \ 
T(f)=\partial_2 h + \partial_1 f.
\]
We give separate arguments when  $2 \in S$ and $2 \notin S$.
If $S$ contains $2$,
then
\[
\partial^S T(f)=\partial^{S \cup 1} f,
\]
 and hence $\partial^S T(f)$ is stable.
If $S$ does not contain $2$, then
\begin{enumerate}[--]\itemsep 5pt
\item the linear form $\partial^S \partial_1 \partial_2 f=\partial^{S \cup 1 \cup 2} f$ is not identically zero, because $f \in \mathring{\mathrm{L}}^{d+1}_n$, 
\item we have $\partial^{S \cup 1 \cup 2} f \prec \partial^{S \cup 2} h$, because $\partial^{S \cup 2} f$ is stable, and
\item we have $\partial^{S \cup 1 \cup 2} f \prec \partial^{S \cup 1} f$, by Lemma \ref{closure} (1).
\end{enumerate}
Therefore, by Lemma \ref{closure} (4), the quadratic form $\partial^S T(f)=\partial^{S \cup 2} h + \partial^{S \cup 1} f$ is stable.
\end{proof}

\begin{proof}[Proof of Theorem \ref{hrcpr}]
The \emph{polarization} of $T:\mathbb{R}_\kappa[w_i] \rightarrow \mathbb{R}_\gamma[w_i]$
 is the  operator $\Pi^\uparrow(T)$ defined by
\[
\Pi^\uparrow(T)=   \Pi_\gamma^{\uparrow}\circ T \circ \Pi_\kappa^{\downarrow}.%:\mathbb{R}_\kappa^\textrm{a}[w_i] \longrightarrow \mathbb{R}_\gamma^\textrm{a}[w_i].
\]
We write $\gamma \oplus \kappa $ for the concatenation of $\gamma$ and $\kappa$ in $\NN^{m+n}$.
By \cite[Lemma 3.5]{BBI}, the symbol of the polarization is the polarization of the symbol\footnote{The statement was proved in \cite[Lemma 3.5]{BBI} when $m=n$.
Clearly, this special case implies the general case.}:
\[
\text{sym}_{\Pi^\uparrow(T)}=\Pi_{\gamma \hspace{0.3mm}\oplus\hspace{0.3mm} \kappa}^\uparrow(\text{sym}_T).
\]
%Since $T= \Pi^\downarrow_\gamma \circ \Pi^\uparrow(T) \circ \Pi^\uparrow_\kappa$, 
Therefore, by Proposition \ref{pols}, the proof reduces to the case $\kappa=(1,\ldots,1)$ and $\gamma=(1,\ldots,1)$.

Suppose $f(v)$ is a multi-affine  polynomial in $\mathrm{L}^d_n$ 
and $\text{sym}_T(w,u)$ is a multi-affine polynomial in $\mathrm{L}^{\ell+n}_{m+n}$.
Since the product of Lorentzian polynomials is Lorentzian by Corollary \ref{CorollaryProduct}, we have
\[
\text{sym}_T(w,u)f(v) = \sum_{S \subseteq [n]} T(w^S)u^{[n]\setminus S}f(v) \in \mathrm{L}^{d+\ell +n}_{m+n+n}.
\]
Applying the operator in Lemma \ref{LS-lemma} for the pair of variables $(u_i,v_i)$ for $i=1,\ldots,n$, we have
\[
\prod_{i=1}^n T_{u_i,v_i} \big( \text{sym}_T(w,u)f(v) \big)=\sum_{S \subseteq [n]} T(w^S)(\partial^S f)(v) \in \mathrm{L}^{d+\ell}_{m+n}.
\]
We substitute every $v_i$ by zero in  the displayed equation to get%preserves the Lorentzian property by Theorem \ref{flow}, 
\[
\Big[\sum_{S \subseteq [n]} T(w^S)(\partial^S f)(v)\Big]_{v=0}=T\big(f(w)\big). %\qedhere
\]
Theorem \ref{flow} shows that the right-hand side belongs to $\mathrm{L}^{d+\ell}_{m}$, completing the proof.
%It is straightforward to check that the right-hand side is equal to $T\big(f(w+v)\big)$ if
%\[
%f(w)=\sum_{R \in {n \brack d}} c_R w^R
%\]
%\[
%f(w+v)=\sum_{R \in {n \brack d}} c_R (w+v)^R
%=\sum_{R \in {n \brack d}} \sum_{S \subseteq R} c_R w^S v^{R \setminus S}
%\]
%\[
%T\Big(f(w+v)\Big)
%=\sum_{R \in {n \brack d}} \sum_{S \subseteq R} c_R T(w^S) v^{R \setminus S}
%=\sum_{S \subseteq [n]}  \sum_{S \subseteq R\in {n \brack d}} T(w^S) \partial^S (c_R v^{R})
%=\sum_{S \subseteq [n]}  T(w^S)  \partial^S  \sum_{S \subseteq R\in {n \brack d}}  (c_R v^{R})
%=\sum_{S \subseteq [n]}  T(w^S)  \partial^S f(v)
%\]
 %we see that 
%\[
%\sum_{S \subseteq [n]} T(w^S)(\partial^S f)(v) = T \Big( f(w+v) \Big) \in \HR_{(1,\ldots,1)}.
%\]
\end{proof} 

%For any $f=\sum_\alpha c_\alpha w^\alpha$, we write $N(f)$ for the polynomial  $\sum_\alpha c_\alpha \frac{w^\alpha}{\alpha!}$.

%\begin{remark}
We remark that there are homogeneous linear operators $T$ preserving the Lorentzian property whose symbols are not Lorentzian. 
This contrasts the  analog of Theorem \ref{hrcpr} for stable polynomials \cite[Theorem 2.2]{BBI}.
As an example, consider the  linear operator $T:\mathbb{R}_{(1,1)}[w_1,w_2] \rightarrow \mathbb{R}_{(1,1)}[w_1,w_2]$ defined by
\[
T(1)=0, \quad T(w_1)=w_1, \quad T(w_2)=w_2, \quad T(w_1w_2)=w_1w_2.
\]
The symbol of $T$ is not Lorentzian because its support is not $\MM$-convex.
The operator $T$ preserves Lorentzian polynomials but does not preserve (non-homogeneous) stable polynomials. %with nonnegative coefficients.
%\end{remark}

%For stable polynomials, the converse of the analog of Theorem~\ref{hrcpr} is true provided that the rank of $T$ is at least three, see \cite[Theorem 2.2]{BBI}. 
%The converse of Theorem~\ref{hrcpr} is not true under such circumstances. Let $f$ be a strictly stable homogeneous polynomial of degree $d-1$ and consider the set of polynomials 
%$$
%\mathcal{A}_f = \{ g \mid g \mbox{ is stable, homogeneous, } \deg g =d, \mbox{ and } f \prec g\}.
%$$
%Then $\mathcal{A}_f$ is a full-dimensional convex cone in $\HR_m^{d}$, by Lemma \ref{closure}. Hence any linear operator  $T :  \RR_{\kappa}[w_1,\ldots, w_n] \rightarrow \RR[w_1,\ldots, w_m]$ whose image is contained in $\mathcal{A}_f$ will preserve the Lorentzian property. We may choose such a $T$ with arbitrary rank greater than two such that the support of $G_T$ is not $\MM$-convex, and hence $\pm G_T(\pm w,v)$ fails to be Lorentzian for such  $T$. Indeed, in 
%$$
%G_T = \sum_{\alpha \leq \kappa} \binom \kappa \alpha T(w^\alpha) v^{\kappa-\alpha}
%$$
%we have the freedom of choosing whether $T(w^\alpha) \in \mathcal{A}_f \setminus \{0\}$ or 
%$T(w^\alpha) \equiv 0$ freely. 

%\begin{remark}
%Note that when $n=2$, Theorem \ref{hrcpr} provides a large class of linear operators preserving ultra log-concavity. 
%\end{remark}

\begin{theorem}\label{stab-lor}
If $T$ is a homogeneous linear operator that preserves stable polynomials and polynomials with nonnegative coefficients, 
then $T$ preserves  Lorentzian polynomials. 
\end{theorem}

\begin{proof}
According to \cite[Theorem 2.2]{BBI},  $T$ preserves stable polynomials if and only if either
\begin{enumerate}[(I)]\itemsep 5pt
\item the rank of $T$ is not greater than two and $T$ is of the form
\[
T(f)=\alpha(f) P+\beta(f) Q,
\]
where $\alpha,\beta$ are linear functionals and $P,Q$ are stable polynomials satisfying $P \prec Q$,
\item the  polynomial $\text{sym}_T(w,u)$ is stable, or
\item the  polynomial $\text{sym}_T(w,-u)$ is stable.
\end{enumerate}
Suppose one of the three conditions, and suppose in addition that $T$ preserves polynomials with nonnegative coefficients.

Suppose (I) holds.  
In this case, the image of $T$ is contained in the set of stable polynomials \cite[Theorem 1.6]{BB10}.
By Proposition \ref{PropositionStableLorentzian},  homogeneous stable polynomials with nonnegative coefficients are Lorentzian.
Since $T$ preserves polynomials with nonnegative coefficients,  $T(f)$ is Lorentzian whenever $f$ is a homogeneous polynomial with nonnegative coefficients.

Suppose (II) holds. 
Since $T$ preserves polynomials with nonnegative coefficients,    $\text{sym}_T(w,u)$  is Lorentzian by Proposition  \ref{PropositionStableLorentzian}.
Therefore, by Theorem \ref{hrcpr}, $T(f)$ is Lorentzian whenever $f$ is Lorentzian.

Suppose (III) holds.
Since all the nonzero coefficients of a homogeneous stable polynomial have the same sign \cite[Theorem 6.1]{COSW},
%Therefore, since $T$ preserves polynomials with nonnegative coefficients, we have 
we have
\[
\text{sym}_T(w,-v)=\text{sym}_T(w,v) \ \ \text{or} \ \ \text{sym}_T(w,-v)=-\text{sym}_T(w,v).
\]
In both cases, $\text{sym}_T(w,v)$ is stable and has nonnegative coefficients.
Thus $\text{sym}_T(w,v)$ is Lorentzian, %by Proposition \ref{PropositionStableLorentzian}, 
and the conclusion follows from Theorem \ref{hrcpr}.
\end{proof}
%The following closure properties now follows because the corresponding linear operators preserve stability, see \cite{BBI}. 
%\begin{corollary}\label{closprop}
%Suppose $f(w_1,\ldots, w_n)$ is Lorentzian. Then so are 
%\begin{enumerate}
%\item $f(w_1,\ldots, w_{j-1}, \lambda_1 w_1+\cdots+\lambda_nw_n + \lambda_{n+1}w_{n+1}, w_{j+1}, \ldots, w_n)$, where $\lambda_i \geq 0$, $1 \leq i \leq n+1$, and $w_{j+1}$ is a new variable, 
%\item $\lambda_1\partial_1 f+\cdots+\lambda_n \partial_n f$, where $\lambda_i \geq 0$, $1 \leq i \leq n$. 
%\end{enumerate}
%\end{corollary}

In the remainder of this section, we record some useful operators that preserves the Lorentzian property.
The \emph{multi-affine part} of a polynomial $\sum_{\alpha \in \NN^n} c_\alpha w^\alpha$
is the polynomial  $\sum_{\alpha \in \{0,1\}^n} c_\alpha w^\alpha$.

\begin{corollary}\label{Corollary-Multi-Affine}
The multi-affine part of any Lorentzian polynomial is a Lorentzian polynomial.
\end{corollary}

\begin{proof}
Clearly, taking the multi-affine part is a homogeneous linear operator that preserves polynomials with nonnegative coefficients.
Since this operator also preserves stable polynomials \cite[Proposition 4.17]{COSW},
the proof follows from Theorem \ref{stab-lor}.
\end{proof}

\begin{remark}\label{BallLikeMultiaffine}
Corollary \ref{Corollary-Multi-Affine} can be used to obtain a multi-affine analog of Theorem \ref{BallLike}.
Write $\underline{\mathrm{H}}^{d}_{n}$ for the space of multi-affine degree $d$ homogeneous polynomials in $n$ variables,
and write $\underline{\mathrm{L}}^{d}_{n}$ for the corresponding set of multi-affine Lorentzian polynomials.
Let $\mathbb{P}\underline{\mathrm{H}}^{d}_{n}$ be the projectivization of the vector space $\underline{\mathrm{H}}^d_n$,
and let $\underline{\mathrm{L}}_\mathrm{B}$ be the set of polynomials in $\underline{\mathrm{L}}^d_n$ with support $\mathrm{B}$.
We identify a rank $d$ matroid $\mathrm{M}$ on $[n]$ with its set of bases $\mathrm{B} \subseteq {n \brack d}$.
Writing $\mathbb{P}\underline{\mathrm{L}}^d_n$ and $\mathbb{P}\underline{\mathrm{L}}_\mathrm{B}$ for the images of $\underline{\mathrm{L}}^d_n \setminus 0$ and $\underline{\mathrm{L}}_\mathrm{B}$ in $\mathbb{P}\underline{\mathrm{H}}^d_n$
respectively,
we have
\[
\mathbb{P}\underline{\mathrm{L}}^d_n = \coprod_{\mathrm{B}} \mathbb{P}\underline{\mathrm{L}}_\mathrm{B}, %\quad \text{where $\mathrm{L}_\mathrm{J} =  \big\{f \in \mathrm{L}^d_n \mid \text{supp}(f)=\mathrm{J}\big\}$},
\]
%where $ \mathrm{L}_\mathrm{J} =  \big\{f \in \mathrm{L}^d_n \mid \text{supp}(f)=\mathrm{J}\big\}$ 
where the union is over all rank $d$ matroids on $[n]$.
By Theorem \ref{flow} and Corollary \ref{Corollary-Multi-Affine}, 
$\mathbb{P}\underline{\mathrm{L}}^d_n$ is a compact contractible subset of $\mathbb{P}\underline{\mathrm{H}}^d_n$. %with contractible interior $\mathbb{P}\mathrm{L}_{{n \brack d}}$.
By Theorem \ref{charjump}, $\mathbb{P}\underline{\mathrm{L}}_\mathrm{B}$ is nonempty for every matroid $\mathrm{B} \subseteq {n \brack d}$. %nonempty $\MM$-convex subset $\mathrm{J}$ of $\Delta^d_n$.
In addition, by Proposition \ref{tot}, $\mathbb{P}\underline{\mathrm{L}}_\mathrm{B}$ is contractible for every  matroid $\mathrm{B} \subseteq {n \brack d}$. %It would be interesting to study the boundary structure of $\mathbb{P}\underline{\mathrm{L}}^d_n$ in greater detail.
\end{remark}

Let $N$ be the linear operator defined by the condition
$N(w^\alpha)=\frac{w^\alpha}{\alpha!}$.
The normalization operator $N$ turns generating functions into exponential generating functions.
%For its use  in algebraic combinatorics, see, for example,  \cite{HMMS}.

\begin{corollary}\label{CorollaryNormalization}
If $f$ is a Lorentzian polynomial, then $N(f)$ is a Lorentzian polynomial.
\end{corollary}

It is shown in \cite[Theorem 3]{HMMS} that the normalized Schur polynomial $N(s_\lambda(w_1,\ldots,w_n))$ is Lorentzian for any partition $\lambda$.
Note that the converse of Corollary \ref{CorollaryNormalization} fails in general.
For example, complete symmetric polynomials, which are special cases of Schur polynomials, need not be Lorentzian.

\begin{proof}
Let $\kappa$ be any element of $\mathbb{N}^n$. 
%It is enough to prove that the linear operator $N_1$ defined by the condition $N_1(w^\alpha)=\frac{w^\alpha}{\alpha_1!}$
%preserves the Lorentzian property.
%For this, 
%and let $N$ be the normalization operator  $w^\alpha \mapsto \frac{w^\alpha}{\alpha!}$ acting on the space of polynomials of degree $\le \kappa$. 
By Theorem \ref{hrcpr}, it suffices to show that the symbol
%\[
%\text{sym}_{N_1}(w,u)
%= \sum_{0\leq \alpha \leq \kappa}\binom \kappa \alpha \frac{w^\alpha}{\alpha_1!} u^{\kappa-\alpha}
%= \Bigg[ \sum_{0\leq \alpha_1 \leq \kappa_1} \binom {\kappa_1} {\alpha_1} \frac{w_1^{\alpha_1}}{\alpha_1!} u_1^{\kappa_1-\alpha_1}\Bigg] (w+u)^{\kappa-\kappa_1 e_1}
%\] %$g_N$ of the normalization operator $w^\alpha \mapsto \frac{w^\alpha}{\alpha!}$ on the space of polynomials of degree $\le \kappa$ is Lorentzian  for any $\kappa$, where
\[
\text{sym}_N(w,u)%= T\Big( (w_1+u_1)^{\kappa_1} \cdots (w_n+u_n)^{\kappa_n} \Big) 
= \sum_{0\leq \alpha \leq \kappa}\binom \kappa \alpha \frac{w^\alpha}{\alpha!} u^{\kappa-\alpha}
= \prod_{j=1}^n \Bigg( \sum_{0\leq \alpha_j \leq \kappa_j} \binom {\kappa_j} {\alpha_j} \frac{w_j^{\alpha_j}}{\alpha_j!} u_j^{\kappa_j-\alpha_j}\Bigg)
\] 
is a Lorentzian polynomial.
Since the product of Lorentzian polynomials is Lorentzian %(Theorem \ref{allequal} and 
by Corollary \ref{CorollaryProduct},
the proof is reduced to the case when the symbol is bivariate.
In this case, using the characterization of bivariate Lorentzian polynomials in Example \ref{ExampleBivariate}, 
we get the Lorentzian property  from the log-concavity of the sequence $1/k!$.
%\[
%\Big(\frac{1}{k!}\Big)^2 >\frac{1}{(k+1)!} \frac{1}{(k+1)!}
%\]
%The statement is straightforward to check using Theorem \ref{chars}.
%A straightforward computation shows that,  for any $\beta \in \Delta_{\ n+n}^{|\kappa|-2}$,
%the quadratic form $\partial^\beta \text{Sym}_\mathrm{N}(w,u)$ is either identically zero or has exactly one positive and one negative eigenvalues.
%The support of $\text{Sym}_\mathrm{N}(w,u)$ is identical to that of $(w+u)^\kappa$, and hence it is
% $\mathrm{M}$-convex.
% The conclusion follows from Theorem \ref{chars}.
\end{proof}

Corollary \ref{CorollaryConvolution} below extends the classical fact that the convolution product of two log-concave sequences with no internal zeros is a log-concave sequence with no internal zeros.
For early proofs of the classical fact, see \cite[Chapter 8]{Karlin} and \cite{Menon}. 

\begin{corollary}\label{CorollaryConvolution}
If $N(f)$ and $N(g)$ are Lorentzian polynomials, then  $N(fg)$ is a Lorentzian polynomial.
\end{corollary}

Note that the analogous statement for stable polynomials fails to hold in general.
For example, when $f=x^3+x^2y+xy^2+y^3$,
the polynomial $N(f)$ is stable but $N(f^2)$ is not.

\begin{proof}
Suppose that $f$ and $g$ belong to $\mathbb{R}_\kappa[w_i]$.
We consider the linear operator
\[
T:\mathbb{R}_\kappa[w_i] \longrightarrow \mathbb{R}_{}[w_i],  \qquad N(h) \longmapsto N(hg).
\]
By Theorem \ref{hrcpr}, it is enough to show that its symbol
%\[
%\text{Sym}_{\mathrm{T}(g)}(w,u)= \sum_{0\leq \alpha \leq \kappa}\binom \kappa \alpha\hspace{0.5mm} \alpha!\hspace{0.5mm} \mathrm{N}(w^\alpha g) u^{\kappa-\alpha}
%=\kappa! \hspace{0.5mm} \mathrm{N}\Big( \Big( \sum_{0\le \alpha \le \kappa} w^\alpha u^{\kappa-\alpha}\Big) g\Big)
%\]
\[
\text{sym}_T(w,u)= \kappa!\sum_{0\leq \alpha \leq \kappa} N(w^\alpha g) \frac{u^{\kappa-\alpha}}{(\kappa-\alpha)!}
\]
is a Lorentzian polynomial in $2n$ variables. For this, we consider the linear operator
%\[
%\mathrm{S}:\mathbb{R}_\kappa[w_i] \longrightarrow \mathbb{R}[w_i,u_i], \qquad  \mathrm{N}(h) \longmapsto  \kappa! \hspace{0.5mm} \mathrm{N}\Big( \Big( \sum_{0\le \alpha \le \kappa} w^\alpha u^{\kappa-\alpha}\Big) h\Big).
%\]
\[
S:\mathbb{R}_\kappa[w_i] \longrightarrow \mathbb{R}[w_i,u_i], \qquad  N(h) \longmapsto  \sum_{0\leq \alpha \leq \kappa} N(w^\alpha h) \frac{u^{\kappa-\alpha}}{(\kappa-\alpha)!}.
\]
By Theorem \ref{hrcpr}, it is enough to show that its symbol
\begin{align*}
\text{sym}_S(w,u,v)
%&=\sum_{0 \le \beta \le \kappa} \binom \kappa \beta \beta!  \Big( \sum_{0\leq \alpha \leq \kappa}\binom \kappa \alpha\hspace{0.5mm} \alpha!\hspace{0.5mm} \mathrm{N}(w^{\alpha+\beta}) u^{\kappa-\alpha}\Big) v^{\kappa-\beta}\\
%&=\sum_{0 \le \beta \le \kappa} \binom \kappa \beta \beta!  \Big( \sum_{0\leq \alpha \leq \kappa}\binom \kappa \alpha\hspace{0.5mm} \alpha!\hspace{0.5mm} \mathrm{N}(w^{\alpha+\beta}) u^{\kappa-\alpha}\Big) v^{\kappa-\beta}\\
&=  \kappa! \sum_{0 \le \beta \le \kappa} \sum_{0 \le \alpha \le \kappa} \frac{w^{\alpha+\beta}}{(\alpha+\beta)!} \frac{u^{\kappa-\alpha}}{(\kappa-\alpha)!} \frac{v^{\kappa-\beta}}{(\kappa-\beta)!}
\end{align*}
is a Lorentzian polynomial in $3n$ variables.
The statement is straightforward to check using Theorem \ref{chars}.
See Theorem \ref{charjump} below for a more general statement.
\end{proof}

%\subsubsection{Partial symmetrization}
The \emph{symmetric exclusion process}  is one of the main models considered in interacting particle systems. 
It is a continuous time Markov chain which models particles that jump symmetrically between sites, where each site may be occupied by at most one particle \cite{Liggett10}. 
A problem that has attracted much attention is to find negative dependence properties that are preserved under the symmetric exclusion process. 
In \cite[Theorem 4.20]{BBL}, it was proved that strongly Rayleigh measures are preserved under the symmetric exclusion process. 
%An important property of the class of strongly Rayleigh measures is that it is closed under the symmetric exclusion process \cite[Theorem 4.20]{BBL}. 
In other words, %the class of stable multi-affine polynomials is that it is closed under partial symmetrization:
if $f=f(w_1,w_2,\ldots,w_n)$ is a stable multi-affine polynomial with nonnegative coefficients, then the multi-affine polynomial $\Phi^{1,2}_\theta(f)$ defined by
\[
\Phi^{1,2}_\theta(f)=
(1-\theta) f(w_1,w_2,w_3,\ldots,w_n) +\theta f(w_2,w_1,w_3,\ldots,w_n)
\]
is stable for all $0 \le \theta \le 1$.
%The same is true for Lorentzian polynomials by 
%Therefore, by Corollary \ref{stab-lor}, the same holds for Lorentzian polynomials. 
We prove an analog for Lorentzian  polynomials.

\begin{corollary}\label{partsym}
Let $f=f(w_1,w_2,\ldots,w_n)$ be a multi-affine polynomial with nonnegative coefficients.
If the homogenization of $f$ 
 is a Lorentzian polynomial, then the homogenization of
$\Phi^{1,2}_\theta(f)$ is a Lorentzian polynomial for all $0 \le \theta \le 1$.
\end{corollary}

\begin{proof}
Recall that a polynomial with nonnegative coefficients is stable if and only if its homogenization is stable \cite[Theorem 4.5]{BBL}.
Clearly, $\Phi^{1,2}_\theta$ is homogeneous and preserves polynomials with nonnegative coefficients.
Since $\Phi^{1,2}_\theta$ preserves stability of multi-affine polynomials by \cite[Theorem 4.20]{BBL}, the statement follows from Theorem \ref{stab-lor}.
%The statement immediately follows from Corollary \ref{stab-lor} and  \cite[Theorem 4.20]{BBL}. %applied to $\Phi^{1,2}_\theta$.
\end{proof}

%We remark that the same implication holds under the weaker assumption %for any Lorentzian  $f$ with $\partial_1^2f \equiv \partial_2^2 f \equiv 0$.

%\begin{remark}
%In \cite{Wagner11} Wagner proved that strong Rayleigh measures are preserved under $\mathrm{SEP}$ also when particles are allowed to created and annihilated at sites. This is true also for Lorentzian measures by Corollary \ref{stab-lor}, since the corresponding operators on the partition functions preserve stability. 
%\end{remark}

%\begin{proof}
%It is not hard to see that if $\supp(f)$ is $\MM$-convex, then so is $\supp(g)$. 
%For each $|\alpha|=d-2$ where $(\alpha_1,  \alpha_2) \in \{0,1\}^2$, we need to prove that $\partial^\alpha g$ is stable or identically zero. The cases when $\alpha_1=\alpha_2$ follows from the fact that 
%$\Phi_\theta$ preserves stability. For the case when $\alpha_1=0$, $\alpha_2=1$, write 
%$$
%f=A+Bw_1+Cw_2+Dw_1w_2, \ \ A,B,C,D \in \RR[w_3,\ldots, w_n].
%$$
%By Lemma \ref{deform} we may perturb $f$ slightly so that $D$ has maximal possible support. If $\alpha'=\alpha-e_2$, then 
%$$
%\partial ^\alpha [f(w_1,w_2,w_3,\ldots,w_n)] = \partial^{\alpha'} C + w_1 \partial^{\alpha'} D
%$$
%and 
% $$
%\partial ^\alpha [f(w_2,w_1,w_3,\ldots,w_n)] = \partial^{\alpha'} B + w_1 \partial^{\alpha'} D
%$$
%are stable by assumption. But then so is 
%$$
%\partial^\alpha \Phi_\theta(f) = (1-\theta)  \partial^{\alpha'} C+ \theta \partial^{\alpha'} B +  w_1\partial^{\alpha'} D
%$$
%by Lemma \ref{closure}. 
%\end{proof}

\subsection{Matroids, $\MM$-convex sets, and Lorentzian polynomials}\label{secM}

The \emph{generating function} of a subset $\mathrm{J} \subseteq \NN^n$ is, by definition,
\[
f_{\JJ}= \sum_{\alpha \in \JJ} \frac {w^{\alpha}}{\alpha!}, \ \ \text{where} \ \ \alpha!=\prod_{i=1}^n\alpha_i!.
\]
We  characterize matroids and $\MM$-convex sets in terms of their generating functions. 

\begin{theorem}\label{charjump}
The following are equivalent for any nonempty $\mathrm{J} \subseteq \NN^n$. 
\begin{enumerate}[(1)]\itemsep 5pt
\item There is a Lorentzian polynomial whose support is $\JJ$.
\item There is a homogeneous $2$-Rayleigh polynomial whose support is $\JJ$.
\item There is a homogeneous $c$-Rayleigh polynomial whose support is $\JJ$ for some $c>0$.
\item The generating function $f_\JJ$ is a Lorentzian polynomial.
\item The generating function $f_\JJ$ is a homogeneous $2$-Rayleigh polynomial.
\item The generating function $f_\JJ$ is a homogeneous $c$-Rayleigh polynomial for some $c>0$.
\item $\JJ$ is  $\MM$-convex. 
\end{enumerate}
When $\mathrm{J} \subseteq \{0,1\}^n$, any one of the above conditions is equivalent to 
\begin{enumerate}[(1)]\itemsep 5pt
\item[(8)] $\JJ$ is the set of bases of a matroid on $[n]$.
\end{enumerate}
\end{theorem}

The statement that the basis generating polynomial $f_\mathrm{J}$  is log-concave on the positive orthant can be found in \cite[Theorem 4.2]{AGV}.
An equivalent statement that the Hessian $f_\mathrm{J}$ has exactly one positive eigenvalue on the positive orthant has been noted earlier in \cite[Remark 15]{HW}.
%The implication (8) $\Rightarrow$ (4) %$\Rightarrow$ (5) 
%goes back to \cite[Remark 15]{HW}. See also \cite[Theorem 4.2]{AGV} and \cite[Section 3]{HSW}.
%These proofs use \cite{AHK}.
The equivalence of the conditions (4) and (7)  will be generalized to $\MM$-convex functions in Theorem \ref{classical}.

We prepare the proof  of Theorem \ref{charjump} with an analysis of the quadratic case.

\begin{lemma}\label{deg2}
The following conditions are equivalent for any  $n \times n$ symmetric matrix $A$ with entries in $\{0,1\}$.
\begin{enumerate}[(1)]\itemsep 5pt
%\item The matrix $A$ has exactly one positive eigenvalue.
\item The quadratic polynomial $w^T A w$ is Lorentzian.
%$$
%f_A= w^TAw= \sum_{i,j=1}^n a_{ij}w_iw_j
%$$
\item The support of the quadratic polynomial $w^T A w$ is $\MM$-convex. 
\end{enumerate}
\end{lemma}

\begin{proof}
Theorem \ref{chars} implies (1) $\Rightarrow$ (2).
We prove (2) $\Rightarrow$ (1). 
We may and will suppose that no column of $A$ is zero.
Let $\JJ$ be the $\MM$-convex support of $w^T A w$, and set
\[
S=\Big\{i \in [n] \mid 2e_i \in \JJ \Big\}.
\]
The exchange property for $\JJ$ shows that $e_i+e_j \in \JJ$ for every $i \in S$ and $j \in [n]$.
In addition, again by the exchange property for $\JJ$,
\[
\mathrm{B} \coloneqq \Big\{e_i+e_j  \in \JJ \mid \text{$i \notin S$ and $j \notin S$}\Big\}
\]
 is the set of bases of a rank $2$ matroid on $[n] \setminus S$ without loops.
Writing $S_1 \cup \cdots \cup S_k$ for the decomposition of $[n] \setminus S$ into parallel classes in the matroid,
we have
%\[
%\JJ= \Big\{e_i+e_j  \in \JJ \mid \text{$i \in S$ and $j \in [n]$}\Big\} \cup  \Big\{e_i+e_j  \mid \text{$i \in S_l$ and $j \in S_m$ for $l \neq m$}\Big\}.
%\]
%Therefore, we have
\[
w^TAw=\Big(\sum_{j \in [n]}w_j\Big)^2 - \Big(\sum_{j \in S_1}w_j\Big)^2-\cdots -\Big(\sum_{j \in S_k}w_j\Big)^2, 
\]
and hence $w^TAw$ is a Lorentzian polynomial.
\end{proof}

\begin{proof}[Proof of Theorem \ref{charjump}]
Theorem \ref{MRayleigh}, Theorem \ref{chars},  and Proposition \ref{HRRayleigh} show that
\[
(1) \Rightarrow (2) \Rightarrow (3) \Rightarrow (7)\ \ \text{and} \ \   (4) \Rightarrow (5) \Rightarrow (6) \Rightarrow (7).
\]
Since (4) $\Rightarrow$ (1), we only need to prove (7) $\Rightarrow$ (4).

If $\JJ$ is an $\MM$-convex subset of $\NN^n$,
then  $f_\JJ$ is a homogeneous polynomial of some degree $d$.
  Suppose $d \ge 2$, and let $\alpha$ be an element of $\Delta^{d-2}_n$.
Note that, in general, the support of $\partial^\alpha f_\JJ$ is $\MM$-convex whenever the support of $f_\JJ$ is $\MM$-convex.
%\[
%\text{$\text{supp}(f_\mathrm{J})$  is $\MM$-convex} \Longrightarrow 
%\text{$\text{supp}(\partial^\alpha f_\mathrm{J})$  is $\MM$-convex}.
%\]
Therefore,
  $\partial^\alpha f_\mathrm{J}$  is Lorentzian by Lemma \ref{deg2},
and hence $f_\JJ$ is Lorentzian by Theorem \ref{chars}.
\end{proof}

Let $\JJ$ be the set of bases of a  matroid $\MM$ on $[n]$.
If $\MM$ is regular \cite{FM}, if $\MM$ is representable over the finite fields $\mathbb{F}_3$ and $\mathbb{F}_4$ \cite{COSW},
if  the rank of $\mathrm{M}$ is at most $3$ \cite{Wagner05}, or if  the number of elements $n$ is at most $7$ \cite{Wagner05}, then
$f_\JJ$ is  $1$-Rayleigh.
Seymour and Welsh found the first example of a matroid whose basis generating function is not $1$-Rayleigh \cite{SW}.
We propose the following improvement of Theorem \ref{charjump}.

\begin{conjecture}
The following conditions are equivalent for any nonempty  $\JJ \subseteq \{0,1\}^n$.
\begin{enumerate}[(1)]\itemsep 5pt
\item  $\JJ$ is the set of bases of a matroid on $[n]$.
\item  The generating function $f_\JJ$ is a homogeneous $\frac{8}{7}$-Rayleigh polynomial.
\end{enumerate}
 \end{conjecture}
 
The constant $\frac{8}{7}$ is best possible: For any positive real number $c<\frac{8}{7}$, there is a matroid whose basis generating function is not $c$-Rayleigh \cite[Theorem 7]{HSW}.

\subsection{Valuated matroids, $\MM$-convex functions, and Lorentzian polynomials}\label{tropsec}

%\subsubsection{}

%Let $\mathrm{M}$ be a matroid on $[n]$ with the set of bases $\mathrm{B}$.
%A \emph{valuated matroid} on $\MM$ is a function $\nu:\mathrm{B} \to \mathbb{R}$ that satisfies the \emph{exchange property}:
%\[
%\text{For any $B_1,B_2 \in \mathrm{B}$ and $i \in B_1 \setminus B_2$, there is  $j \in B_2 \setminus B_1$ such that
%$\nu(B_1)+\nu(B_2) \le \nu((B_1 \setminus i) \cup j)+\nu((B_2\setminus j) \cup i)$.}
%\]
Let $\nu$ be a function from $\NN^n$ to $\mathbb{R} \cup \{\infty\}$.
The \emph{effective domain} of $\nu$ is, by definition,
\[
\text{dom}(\nu)=\Big\{\alpha \in \NN^n \mid \nu(\alpha) <\infty\Big\}.
\]
The function $\nu$ is said to be \emph{$\MM$-convex} if satisfies the \emph{symmetric exchange property}:
\begin{enumerate}[(1)]\itemsep 5pt
\item[(1)] For any $\alpha,\beta \in \text{dom}(\nu)$ and any  $i$ satisfying $\alpha_i >\beta_i$, there is  $j$ satisfying
\[
\alpha_j <\beta_j \ \ \text{and} \ \  \nu(\alpha)+\nu(\beta) \ge \nu(\alpha-e_i+e_j)+\nu(\beta-e_j+e_i).
\]
\end{enumerate}
%The condition is called the \emph{symmetric exchange property} for $\MM$-convex functions.
Note that the effective domain of an $\MM$-convex function on $\NN^n$ is an $\MM$-convex subset of $\NN^n$.
In particular, the effective domain of an $\MM$-convex function on $\NN^n$  is contained in $\Delta^d_n$ for some $d$. 
In this case, we identify $\nu$ with its restriction to $\Delta^d_n$.
When the effective domain of $\nu$ is is $\MM$-convex, 
the symmetric exchange property for $\nu$ is equivalent to the following \emph{local exchange property}:
\begin{enumerate}[(1)]\itemsep 5pt
\item[(2)] For any $\alpha,\beta \in \text{dom}(\nu)$ with $|\alpha-\beta|_1=4$, there are $i$ and $j$ satisfying 
\[
\alpha_i >\beta_i, \ \  \alpha_j <\beta_j \ \ \text{and} \ \  \nu(\alpha)+\nu(\beta) \ge \nu(\alpha-e_i+e_j)+\nu(\beta-e_j+e_i).
\]
\end{enumerate}
A proof of the equivalence of the two exchange properties can be found in \cite[Section 6.2]{Murota}.

\begin{example}
The indicator function of  $\mathrm{J} \subseteq \NN^n$ is the function $\nu_\mathrm{J}:\NN^n \rightarrow \mathbb{R} \cup \{\infty\}$ defined by
\[
\nu_\mathrm{J}(\alpha)=\left\{\begin{array}{cc} 0 & \text{if $\alpha \in \mathrm{J}$,}\\  \infty & \text{if $\alpha \notin \mathrm{J}$.}\end{array}\right.
\]
Clearly,  $\mathrm{J} \subseteq \NN^n$ is $\MM$-convex if and only if the indicator function $\nu_\mathrm{J}$  is $\MM$-convex.
\end{example}

A function $\nu: \NN^n \to \mathbb{R} \cup \{-\infty\}$ is said to be \emph{$\MM$-concave} if $-\nu$ is $\MM$-convex.
The \emph{effective domain} of an $\MM$-concave function is
\[
\text{dom}(\nu)=\Big\{\alpha \in \NN^n \mid \nu(\alpha) >-\infty\Big\}.
\]
A \emph{valuated matroid} on $[n]$ is an $\MM$-concave function on $\NN^n$ whose effective domain is a nonempty subset of  $\{0,1\}^n$ .
The effective domain of a valuated matroid $\nu$ on $[n]$ is the set of bases of a matroid on $[n]$, the \emph{underlying matroid} of $\nu$.

In this section, we prove that the class of tropicalized Lorentzian polynomials coincides with the class of $\MM$-convex functions.
The tropical connection is used to produce Lorentzian polynomials from $\MM$-convex functions.
First, we state a classical version of the  result.
For  any function $\nu:\Delta^{d}_n \to \mathbb{R} \cup \{\infty\}$ and a positive real number $q$, we define
\[
f^{\nu}_{q}(w)=\sum_{\alpha \in \text{dom}(\nu)} \frac{q^{\nu(\alpha)}}{\alpha!} w^\alpha \ \ \text{and} \ \  g^{\nu}_{q}(w)=\sum_{\alpha \in \text{dom}(\nu)} 
{\delta \choose \alpha} q^{\nu(\alpha)} w^\alpha,
\]
where  $\delta=(d,\ldots,d)$ and
${\delta \choose \alpha}$
is the product of binomial coefficients $\prod_{i=1}^n {d_{\ } \choose \alpha_i}$.
%Note that, when $\text{dom}(\nu) \subseteq {n \brack d}$, we have
%\[
%g_{\nu,q}(w)=d^d f_{\nu,q}(w)=\sum_{\alpha \in \text{dom}(\nu)} q^{\nu(\alpha)} w^\alpha
%\]
When $\nu$ is the indicator function of $\mathrm{J} \subseteq \NN^n$, 
the polynomial $f^\nu_q$ is independent of $q$ and equal to the generating function $f_\mathrm{J}$ considered in Section \ref{secM}.

\begin{theorem}\label{classical}
The following conditions are equivalent for   $\nu:\Delta^d_n \to \mathbb{R} \cup \{\infty\}$.
\begin{enumerate}[(1)]\itemsep 5pt
\item The function $\nu$ is $\MM$-convex.
\item The polynomial $f^{\nu}_{q}(w)$ is Lorentzian for all $0 < q \le 1$.
\item The polynomial $g^{\nu}_{q}(w)$ is Lorentzian for all $0 < q \le 1$.
%\item The polynomial $f_{\nu,q}(w)$ is Lorentzian for all sufficiently small positive numbers $q$.
%\item The polynomial $g_{\nu,q}(w)$ is Lorentzian for all sufficiently small positive numbers $q$.
\end{enumerate}
\end{theorem}

%A proof of Theorem \ref{classical} will be given at the end of this subsection.
The proof of Theorem \ref{classical}, which relies on the theory of phylogenetic trees and the problem of isometric embeddings of finite metric spaces in Euclidean spaces,
will be given at the end of this subsection.

\begin{example}\label{examplerank}
A function $\mu$ from $\mathbb{N}^n$ to $\mathbb{R} \cup \{\infty\}$ is said to be \emph{$\mathrm{M}^\natural$-convex}  if, for some positive integer $d$,
the function $\nu$ from $\mathbb{N}^{n+1}$ to $\mathbb{R} \cup \{\infty\}$ defined by
\[
\nu(\alpha_0,\alpha_1,\ldots,\alpha_n)=\left\{\begin{array}{cc} \mu(\alpha_1,\ldots,\alpha_n) & \text{if $\alpha \in \Delta^d_{n+1}$,} \\ \infty &  \text{if $\alpha \notin \Delta^d_{n+1}$,} \end{array}\right.
\]
is $\mathrm{M}$-convex. %where $\pi$ is the projection onto the last $n$ coordinates.
The condition does not depend on  $d$, and $\mathrm{M}^\natural$-concave functions are defined similarly.
We refer to \cite[Chapter 6]{Murota} for more on $\mathrm{M}^\natural$-convex  and $\mathrm{M}^\natural$-concave functions.

It can be shown that every matroid rank function $\text{rk}_\mathrm{M}$,  viewed as a function on $\mathbb{N}^n$ with the effective domain $\{0,1\}^n$,
is  $\mathrm{M}^\natural$-concave.
See \cite[Section 3]{Shioura} for an elementary proof and other related results.
Thus, by Theorem \ref{classical}, the normalized rank generating function
\[
\sum_{A \subseteq [n]} \frac{1}{c(A)!}q^{-\text{rk}_\mathrm{M}(A)}w^A w_0^{c(A)}, \ \  \text{where $w=(w_1,\ldots,w_n)$ and $c(A)=n-|A|$},
\]
is Lorentzian for all $0<q \le 1$.
We will obtain a sharper result on $\text{rk}_\mathrm{M}$ %on matroid rank functions 
in Section \ref{SecqP}.
\end{example}

%We give a proof in Section \ref{Proofs}. %A proof will be given in Section \ref{Proofs}.
Theorem \ref{classical} provides a useful sufficient condition for a homogeneous polynomial  to be Lorentzian.
Let $f$ be an arbitrary  homogeneous polynomial with nonnegative real coefficients written in the normalized form
\[
f=\sum_{\alpha \in \Delta^d_n} \frac{c_\alpha}{\alpha!}  w^\alpha.
\]
We define a discrete function $\nu_f$ using natural logarithms of the normalized coefficients:
 \[
 \nu_f:\Delta^d_n \longrightarrow \mathbb{R} \cup \{-\infty\}, \qquad \alpha \longmapsto \log(c_\alpha).
 \]
 
 \begin{corollary}\label{normalizedcoefficients}
 If $\nu_f$ is an $\MM$-concave function, then $f$ is a Lorentzian polynomial.
 \end{corollary}
 
 \begin{proof}
 By Theorem \ref{classical}, the polynomial
 $ \sum_{\alpha \in \text{dom}(\nu_f)} \frac{q^{-\nu_f(\alpha)}}{\alpha!} w^\alpha$
 is Lorentzian when $q=e^{-1}$.
 \end{proof}
 
We note that the converse of Corollary \ref{normalizedcoefficients} does not hold. For example, the  polynomial
\[
f=\prod_{i=1}^{n-1} (w_i+w_n)
\]
is Lorentzian, being a product of Lorentzian polynomials. However, $\nu_f$ fails to be $\MM$-concave when $n > 2$.

%\subsubsection{}

We formulate a tropical counterpart of Theorem \ref{classical}.
Let $\mathbb{C}((t))_{\text{conv}}$ be the field of  Laurent series with complex coefficients that have a positive radius of convergence around $0$.
By definition, any nonzero element of $\mathbb{C}((t))_{\text{conv}}$ is a series of the form
\[
s(t)=c_1t^{a_1}+c_2t^{a_2}+c_3t^{a_3}+\cdots, 
\]
where $c_1,c_2,\ldots$ are nonzero complex numbers and $a_1<a_2<\cdots$ are integers,
that converges on a punctured open disk centered at $0$.
Let $\mathbb{R}((t))_{\text{conv}}$ be the subfield of elements that have real coefficients.
We define the fields of real and complex convergent Puiseux series\footnote{The main statements in this section are valid over the field of formal  Puiseux  series as well.}
 by
\[
\mathbb{K}=  \bigcup_{k \ge 1} \mathbb{R}((t^{1/k}))_{\text{conv}} \ \ \text{and} \ \ \overline{\mathbb{K}}= \bigcup_{k \ge 1} \mathbb{C}((t^{1/k}))_{\text{conv}}.
\]
Any nonzero element of $\overline{\mathbb{K}}$ is a series of the form
\[
s(t)=c_1t^{a_1}+c_2t^{a_2}+c_3t^{a_3}+\cdots, 
\]
where $c_1,c_2,\ldots$ are nonzero complex numbers and $a_1<a_2<\cdots$ are rational numbers that have a common denominator.
The \emph{leading coefficient} of $s(t)$  is $c_1$, and the \emph{leading exponent} of $s(t)$ is $a_1$.
A nonzero element of $\mathbb{K}$ is \emph{positive} if its leading coefficient is positive.
The \emph{valuation map} is the function
\[
\text{val}:\overline{\mathbb{K}} \longrightarrow \mathbb{R} \cup \{\infty\},
\]
that takes the zero  element to $\infty$ and a nonzero element to its leading exponent.
For a nonzero element $s(t) \in \mathbb{K}$, we have
\[
\text{val} \big(s(t)\big)=\lim_{t \to 0^+} \log_t \big(s(t)\big).
\]

The field $\overline{\mathbb{K}}$ is algebraically closed, and the field $\mathbb{K}$ is real closed.
See \cite[Section 1.5]{Sp} and references therein.
Since the theory of real closed fields has quantifier elimination \cite[Section 3.3]{Marker},
for any first-order formula $\varphi(x_1,\ldots,x_m)$ in the language of ordered fields and
any $s_1(t),\ldots,s_m(t)\in \mathbb{K}$, we have
\begin{multline*}
\Big(\text{$\varphi(s_1(t),\ldots,s_m(t))$ holds in $\mathbb{K}$}\Big) 
\Longleftrightarrow\\
\Big(\text{$\varphi(s_1(q),\ldots,s_m(q))$ holds in $\mathbb{R}$ for all sufficiently small positive real numbers $q$}\Big). 
\end{multline*}
In particular, Tarski's principle holds for $\mathbb{K}$: A first-order sentence in the language of ordered fields holds in $\mathbb{K}$ if and only if it holds in $\mathbb{R}$.

\begin{definition}
Let $f_t=\sum_{\alpha \in \Delta^d_n} s_\alpha(t) w^\alpha$ be a nonzero  homogeneous polynomial with coefficients in $\mathbb{K}_{\ge 0}$.
%\[
%f_t=\sum_{\alpha \in \Delta^d_n} s_\alpha(t) w^\alpha \in \mathbb{K}_{\ge 0}[w_1,\ldots,w_n].
%\]
The \emph{tropicalization} of $f_t$ is the discrete function defined by
\[
\text{trop}(f_t):\Delta^d_n \longrightarrow \mathbb{R} \cup \{\infty\}, \qquad \alpha \longmapsto \text{val}\big(s_\alpha(t)\big).
\]
We say that $f_t$ is  \emph{log-concave on $\mathbb{K}^n_{> 0}$} if 
the function $\log(f_q)$ is concave on $\mathbb{R}^n_{>0}$ for all sufficiently small positive real numbers $q$.
%\begin{enumerate}[--]\itemsep 5pt
%\item The Hessian of $f_t$ has exactly one positive eigenvalue on $\mathbb{K}^n_{>0}$.
%\item The Hessian of $f_q$ has exactly one positive eigenvalue on $\mathbb{R}^n_{>0}$ and all sufficiently small positive numbers $q$.
%\item The function $\log(f_q)$ is concave on $\mathbb{R}^n_{>0}$ for all sufficiently small positive numbers $q$.
%\end{enumerate}
\end{definition}
 
Note that the support of $f_t$ is the effective domain of the tropicalization of $f_t$.
 %\[
 %\text{supp}(f_t)=\Big\{\alpha \in \NN^n \mid s_\alpha(t) \neq 0\Big\}.
 %\]
 We write $\mathrm{M}^d_n(\mathbb{K})$ for the set of all degree $d$ homogeneous polynomials in $\mathbb{K}_{\ge 0}[w_1,\ldots,w_n]$  whose support is $\MM$-convex.
 
 \begin{definition}[Lorentzian polynomials over $\mathbb{K}$]
 We set $ \mathrm{L}^0_n(\mathbb{K})= \mathrm{M}^0_n(\mathbb{K})$, $ \mathrm{L}^1_n(\mathbb{K})= \mathrm{M}^1_n(\mathbb{K})$, and
\[
 \mathrm{L}^2_n(\mathbb{K})=\Big\{f_t \in \mathrm{M}^2_n(\mathbb{K}) \mid \text{The Hessian of $f_t$ has at most one eigenvalue in $\mathbb{K}_{>0}$}\Big\}.
 \]
 For $d \ge 3$, we define $  \mathrm{L}^d_n(\mathbb{K})$ by setting
 \[
  \mathrm{L}^d_n(\mathbb{K})=\Big\{f_t \in \mathrm{M}^d_n(\mathbb{K}) \mid \text{$\partial^\alpha f_t \in  \mathrm{L}^2_n(\mathbb{K})$ for all $\alpha \in \Delta^{d-2}_n$}\Big\}.
  \]
The polynomials in $ \mathrm{L}^d_n(\mathbb{K})$ will be called \emph{Lorentzian}.
\end{definition}

By Proposition \ref{logcon}, the log-concavity of  homogeneous polynomials can be expressed in the first-order language of ordered fields.
It follows that the analog of Theorem \ref{allequal} holds for any homogeneous polynomial $f_t$ with coefficients in $\mathbb{K}_{\ge 0}$.

\begin{theorem}
The following conditions are equivalent for $f_t$.
\begin{enumerate}[(1)]\itemsep 5pt
\item For any $m \in \NN$ and any $m\times n$ matrix $(a_{ij})$ with entries in $\mathbb{K}_{\ge 0}$,
 \[
 \text{$\Big(\prod_{i=1}^m D_{i}\Big) f_t$ is identically zero or $\Big(\prod_{i=1}^m D_{i}\Big) f_t$ is log-concave on $\mathbb{K}_{> 0}^n$,}
\]
where $D_i$ is the differential operator $\sum_{j=1}^n a_{ij} \partial_j$.
% support of $f_t$ is $\MM$-convex and $\partial^\alpha f_t$ has at most one positive eigenvalue for all $\alpha \in \NN^n$ that makes $\partial^\alpha f_t$ a quadratic form.
\item For any $\alpha \in \NN^n$, the polynomial $\partial^\alpha f_t$ is identically zero or log-concave on $\mathbb{K}_{> 0}^n$.
\item The polynomial $f_t$ is Lorentzian.
\end{enumerate}
\end{theorem}

%We have the following tropical counterpart of Theorem \ref{classical}.
The field $\mathbb{K}$ is real closed, and the field $\overline{\mathbb{K}}$ is algebraically closed \cite[Section 1.5]{Sp}.
Any element $s(t)$ of $\overline{\mathbb{K}}$ can be written as a sum
\[
s(t)=p(t)+i \hspace{0.5mm}q(t),
\]
where $p(t) \in \mathbb{K}$ is the \emph{real part} of $s(t)$ and $q(t) \in \mathbb{K}$ is the \emph{imaginary part} of $s(t)$.
The \emph{open upper half plane in $\overline{\mathbb{K}}$}  is the set of elements in $\overline{\mathbb{K}}$ with positive imaginary parts.
A polynomial $f_t$ in $\mathbb{K}[w_1,\ldots,w_n]$ is \emph{stable} if $f_t$ is non-vanishing on $\mathscr{H}_{\overline{\mathbb{K}}}^n$ or identically zero, 
where $\mathscr{H}_{\overline{\mathbb{K}}}$ is the open upper half plane in $\overline{\mathbb{K}}$.
According to \cite[Theorem 4]{SIAM}, tropicalizations of homogeneous stable polynomials  over $\mathbb{K}$ are $\MM$-convex functions.\footnote{In \cite{SIAM}, the field of formal Puiseux series with real exponents $\mathbb{R}\{t\}$ containing $\mathbb{K}$ was used. The tropicalization used in \cite{SIAM} differs from ours by a sign.}  %and that not all $\MM$-convex functions are tropicalizations of stable polynomials. 
Here we prove that tropicalizations of Lorentzian polynomials  over $\mathbb{K}$ are $\MM$-convex, and that \emph{all} $\MM$-convex functions are limits of tropicalizations of Lorentzian polynomials over $\mathbb{K}$.\footnote{If $\mathbb{R}\{t\}$ is used instead of $\mathbb{K}$, then all $\MM$-convex functions are tropicalizations of Lorentzian polynomials. More precisely,  a discrete function $\nu$ with values in $\mathbb{R} \cup \{\infty\}$ is $\MM$-convex if and only if there is a Lorentzian polynomial over $\mathbb{R}\{t\}$ whose tropicalization is $\nu$. In this setting, the Dressian of a matroid $\mathrm{M}$ can be identified with the set of tropicalized Lorentzian polynomials $f_t$ with $\text{supp}(f_t)=\mathrm{B}$, where $\mathrm{B}$ is the set of bases of $\mathrm{M}$.
} 

\begin{theorem}\label{tropical}
The following conditions are equivalent for any function  $\nu:\Delta^d_n \to \mathbb{Q} \cup \{\infty\}$.
\begin{enumerate}[(i)]\itemsep 5pt
%\item If $f$ is in $\mathrm{L}^d_n(\mathbb{K})$, then $\text{trop}(f)$ is $\MM$-convex.
%\item If $\nu$ is $\MM$-convex, then there is $f$ in $\mathrm{L}^d_n(\mathbb{K})$ such that $\text{trop}(f)=\nu$.
\item The function $\nu$ is $\MM$-convex.
%\item The polynomial $f_{\nu,t}$ is Lorentzian.
\item There is a Lorentzian polynomial in $\mathbb{K}[w_1,\ldots,w_n]$ whose tropicalization is $\nu$.
\end{enumerate}
\end{theorem}

Let $\MM$ be a matroid with the set of bases $\mathrm{B}$.
The \emph{Dressian} of $\MM$, denoted $\text{Dr}(\mathrm{M})$, is the tropical variety in $\mathbb{R}^\mathrm{B}$
obtained by intersecting the tropical hypersurfaces of the  Pl\"ucker relations in $\mathbb{R}^\mathrm{B}$ \cite[Section 4.4]{MS}.
Since $\text{Dr}(\mathrm{M})$ is a rational polyhedral fan whose points bijectively correspond to the valuated matroids with underlying matroid $\MM$, Theorem \ref{tropical} shows that
\[
\text{Dr}(\mathrm{M})=\text{closure}\Big\{-\text{trop}(f_t) \mid \text{$f_t$ is a Lorentzian polynomial with $\text{supp}(f_t)=\mathrm{B}$}\Big\}.
\]
%where $\pi$ is the projection onto $\mathbb{R}^\mathrm{B}$. %\footnote{The equality holds without taking the closure if the field of formal Puiseux series with real exponents $\mathbb{R}\{t\}$ is used.}
We note that the corresponding statement for stable polynomials fails to hold.
For example, when $\mathrm{M}$ is the Fano plane,
 there is no stable polynomial whose support is $\mathrm{B}$ \cite[Section 6]{Branden}.

%\subsubsection{}\label{quadratic}

We prove Theorems \ref{classical} and \ref{tropical} together after reviewing the needed results on the space of phylogenetic trees and the isometric embeddings of  finite metric spaces in Euclidean spaces.
A \emph{phylogenetic tree} with $n$ leaves is a tree with $n$ labelled leaves and no vertices of degree $2$.
A function $\mathrm{d}: {n \brack 2} \to \mathbb{R}$ is a \emph{tree distance} if there is a phylogenetic tree $\tau$ with $n$ leaves and edge weights  $\ell_e \in \mathbb{R}$
such that
\[
\mathrm{d}(i,j)=\Big(\text{the sum of all $\ell_e$ along the unique path in $\tau$ joining the leaves $i$ and $j$}\Big).
\]
%When all the edge weights are positive, $\delta$ defines a metric on $[n]$.
The \emph{space of phylogenetic trees} $\mathscr{T}_n$ is the set of all tree distances in $\mathbb{R}^{n \choose 2}$.
The Fundamental Theorem of Phylogenetics shows that
\[
\mathscr{T}_n=\text{Dr}(2,n),
\]
where $\text{Dr}(2,n)$ is the Dressian of the rank $2$ uniform matroid on $[n]$ \cite[Section 4.3]{MS}.

We give a spectral characterization of tree distances.
For any function $\mathrm{d}:{n \brack 2} \to \mathbb{R}$
and any positive real number $q$,
we define an $n \times n$ symmetric matrix $\mathrm{H}_q(\mathrm{d})$ by
\[
\mathrm{H}_q(\mathrm{d})_{ij}=\left\{\begin{array}{cc} 0&\text{if $i=j$,} \\ q^{\mathrm{d}(i,j)}& \text{if $i \neq j$.}\end{array}\right.
\]
We say that an $n \times n$ symmetric matrix $\mathrm{H}$ is  \emph{conditionally negative definite} if
\[
(1,\ldots,1) w=0 \Longrightarrow w^T \mathrm{H} w \le 0.
\]
Basic properties of conditionally negative definite matrices are collected in \cite[Chapter 4]{BR}.

\begin{lemma}\label{ultrametricHR}
The following conditions are equivalent for any function $\mathrm{d}:{n \brack 2} \to \mathbb{R}$.
\begin{enumerate}[(1)]\itemsep 5pt
%\item The matrix $\mathrm{H}_q(\mathrm{d})$ is conditionally negative semidefinite for all $q \ge 1$.
\item The matrix $\mathrm{H}_q(\mathrm{d})$ has exactly one positive eigenvalue for all $q \ge 1$.
\item The function $\mathrm{d}$ is a tree distance.
\end{enumerate}
\end{lemma}

Lemma \ref{ultrametricHR} is closely linked to the problem of isometric embeddings of ultrametric spaces in Hilbert spaces.
Let $\mathrm{d}$ be a metric on $[n]$.
Since $\mathrm{d}(i,i)=0$ and $\mathrm{d}(i,j)=\mathrm{d}(j,i)$ for all $i$, we may identify $\mathrm{d}$ with a function ${n \brack 2} \to \mathbb{R}$.
We define an $n \times n$ symmetric matrix $\mathrm{E}(\mathrm{d})$ by
\[
\mathrm{E}(\mathrm{d})_{ij}=\mathrm{d}(i,j)^2.
\]
We say that $\mathrm{d}$  admits an \emph{isometric embedding} into $\RR^m$ if there is $\phi : [n] \rightarrow \RR^m$ such that 
\[
\mathrm{d}(i,j) = | \phi(i)-\phi(j) |_2 \ \ \text{for all} \ \  i, j \in [n],
\]
where $| \cdot |_2$ is the standard Euclidean norm on $\mathbb{R}^m$. The following theorem of Schoenberg \cite{Sch2} characterizes  metrics on $[n]$ that admit an isometric embedding into some $\mathbb{R}^m$.

\begin{theorem}\label{schoenthm}
A metric $\mathrm{d}$ on $[n]$ admits an isometric embedding into some $\mathbb{R}^m$
if and only if the matrix
$\mathrm{E}(\mathrm{d})$ is conditionally negative semidefinite.
\end{theorem}

Recall that an \emph{ultrametric} on $[n]$ is a metric $\mathrm{d}$ on $[n]$ such that
\[
\mathrm{d}(i,j) \le \max\Big\{\mathrm{d}(i,k),\mathrm{d}(j,k)\Big\} \ \ \text{for any $i,j,k \in [n]$}.
\]
Equivalently, $\mathrm{d}$ is an ultrametric if the maximum of $\mathrm{d}(i,j),\mathrm{d}(i,k),\mathrm{d}(j,k)$ is attained at least twice for any $i,j,k \in [n]$.
Any ultrametric is a tree distance given by a phylogenetic tree \cite[Section 4.3]{MS}.
In \cite{TV}, Timan and Vestfrid proved that any separable ultrametric space is isometric to a subspace of $\ell_2$.
We use the following special case.

\begin{theorem}\label{embedding}
Any ultrametric on $[n]$ admits an isometric embedding into $\RR^{n-1}$. 
\end{theorem}

\begin{proof}[Proof of Lemma \ref{ultrametricHR}]
We prove (1) $\Rightarrow$ (2). 
We may suppose that $\mathrm{d}$ takes rational values.
If (1) holds, then the quadratic polynomial
$w^T \mathrm{H}_q(\mathrm{d})w$  is stable for all $q \ge 1$.
Therefore, by the quantifier elimination for the theory of real closed fields,
the quadratic form
\[
\sum_{i<j}t^{-\mathrm{d}(i,j)}w_iw_j \in \mathbb{K}[w_1,\ldots,w_n]
\]
is stable.
By \cite[Theorem 4]{SIAM},
 tropicalizations of stable polynomials are $\MM$-convex\footnote{The tropicalization used in \cite{SIAM} differs from ours by a sign.}, and hence 
the function $-\mathrm{d}$ is $\MM$-convex.
In other words, we have
\[
\mathrm{d}\in\text{Dr}(2,n)= \mathscr{T}_n.
\]

For (2) $\Rightarrow$ (1), we first consider the special case when $\mathrm{d}$ is an ultrametric on $[n]$.
In this case, $q^\mathrm{d}$ is also an ultrametric on $[n]$ for all $q \ge 1$.
It follows from Theorems \ref{schoenthm}  and \ref{embedding} that $\mathrm{H}_q(\mathrm{d})$ is conditionally negative definite for all $q \ge 1$, and Cauchy's interlacing theorem shows that conditionally negative definite matrices have at most one positive eigenvalue.
In the general case,
we use that $\mathscr{T}_n$
is the sum of its linearity space with the space of ultrametrics on $[n]$ \cite[Lemma 4.3.9]{MS}.
Thus, for any tree distance $\mathrm{d}$ on $[n]$,
there is an ultrametric $\underline{\mathrm{d}}$ on $[n]$ and real numbers $c_1,\ldots,c_n$ such that
\[
\mathrm{d}= \underline{\mathrm{d}}+\sum_{i=1}^n c_i \Big(\sum_{i \neq j} e_{ij}\Big) \in \mathbb{R}^{n \choose 2}.
\]
Therefore, the symmetric matrix $\mathrm{H}_q(\mathrm{d})$ is congruent to $\mathrm{H}_q(\underline{\mathrm{d}})$,
and the conclusion follows from the case of ultrametrics.
\end{proof}

%\subsubsection{}\label{Proofs}

We start the proof of Theorems \ref{classical} and \ref{tropical} with a linear algebraic lemma.
Let $(a_{ij})$ be an $n \times n$ symmetric matrix with entries in $\mathbb{R}_{> 0}$.

\begin{lemma}\label{power}
If $(a_{ij})$ has exactly one positive eigenvalue, then $(a_{ij}^{\hspace{0.5mm} p})$ has exactly one positive eigenvalue for $0 \le p \le 1$.
\end{lemma}

\begin{proof}
If $(v_i)$ is the Perron eigenvector of $(a_{ij})$, then $(\frac{a_{ij}}{v_iv_j})$ is conditionally negative definite \cite[Lemma 4.4.1]{BR}.
Therefore,  $(\frac{a_{ij}^{\hspace{0.5mm} p}}{v_i^{\hspace{0.5mm} p}v_j^{\hspace{0.5mm} p}})$ is conditionally negative definite  \cite[Corollary 2.10]{BCR},
and the conclusion follows.
\end{proof}

Let $f$ be a degree $d$ homogeneous polynomial 
written in the normalized form
\[
f= \sum_{\alpha \in \text{supp}(f)} \frac{c_\alpha}{\alpha!} w^\alpha.
\]
For any nonnegative real number $p$, we define
\[
R_p(f)= \sum_{\alpha \in \text{supp}(f)} \frac{c_\alpha^{\hspace{0.5mm} p}}{\alpha!} w^\alpha.
\]
%We show that $R_p$ provides a homotopy from $f$ to the generating function of $\text{supp}(f)$ in $\mathrm{L}^d_n$.
We use Lemma \ref{power} to construct a homotopy from any Lorentzian polynomial to
the generating function of its support.
The following proposition was proved in \cite{ALGVII} for strongly log-concave multi-affine polynomials.

\begin{proposition}\label{tot}
If $f$ is Lorentzian, then $R_p(f)$ is Lorentzian for all $0 \le p \le 1$.
\end{proposition}

\begin{proof}
Using the characterization of Lorentzian polynomials in Theorem \ref{chars},
the proof reduces to the case of quadratic polynomials.
Using Theorem \ref{flow}, the proof further reduces to the case $f \in \mathrm{P}^2_n$.
In this case, the assertion is Lemma \ref{power}.
\end{proof}

Set $m=nd$,
and let $\nu:\Delta^d_n \rightarrow \mathbb{R} \cup\{\infty\}$  and $\mu:\Delta^d_m \to \mathbb{R} \cup \{\infty\}$ be arbitrary functions.
Write  $e_{ij}$ for the  standard unit vectors in $\mathbb{R}^{m}$ with $1 \le i \le n$ and $1 \le j \le d$, and let $\phi$ be the linear map
\[
\phi:\mathbb{R}^{m} \longrightarrow \mathbb{R}^n, \qquad  e_{ij} \longmapsto e_i.
\]
We define the \emph{polarization} of $\nu$ to be the  function $\Pi^\uparrow \nu: \Delta^d_{m} \rightarrow \mathbb{R} \cup \{\infty\}$  satisfying
\[
\text{dom}\big( \Pi^\uparrow \nu \big)\subseteq {m \brack d} \ \ \text{and} \ \ \Pi^\uparrow \nu= \nu\circ  \phi   \ \ \text{on ${m \brack d}$}.
\]
We define the \emph{projection} of $\mu$ to be the function $\Pi^\downarrow \mu:\Delta^d_n \to \mathbb{R} \cup \{\infty\}$ satisfying
\[
\Pi^\downarrow \mu (\alpha)= \min \Big\{\mu(\beta) \mid \phi(\beta)=\alpha\Big\}.
\]
%We have the following discrete analog of Proposition \ref{pols}.
 It is straightforward to check the symmetric exchange properties of $\Pi^\uparrow\nu$ and $\Pi^\downarrow \mu$ from  the symmetric exchange properties of $\nu$ and $\mu$.\footnote{In the language of \cite{KMT}, the polarization of $\nu$ is obtained from $\nu$ by splitting of variables and restricting to ${m \brack d}$, and the projection of $\mu$ is obtained from $\mu$ by aggregation of variables.}

\begin{lemma}\label{DiscretePolarization}
Let  $\nu:\Delta^d_n \rightarrow \mathbb{R} \cup\{\infty\}$  and $\mu:\Delta^d_m \to \mathbb{R} \cup \{\infty\}$ be arbitrary functions.
\begin{enumerate}[(1)]\itemsep 5pt
\item If $\nu$ is an $\MM$-convex function, then $ \Pi^\uparrow \nu $ is an $\MM$-convex function.
\item If $\mu$ is an $\MM$-convex function, then $ \Pi^\downarrow \mu $ is an $\MM$-convex function.
\end{enumerate}
\end{lemma}

%\begin{proof}
 %It is straightforward to check that  $\Pi^\uparrow(\nu)$ satisfies the symmetric exchange property.
%\end{proof}

As a final preparation for the proof of Theorems \ref{classical} and \ref{tropical},
we show that any $\MM$-convex function on $\Delta^d_n$ can be approximated by $\MM$-convex functions whose effective domain is $\Delta^d_n$.

\begin{lemma}\label{regularization}
For any $\MM$-convex function $\nu:\Delta^d_n \rightarrow \mathbb{R} \cup\{\infty\}$, there is a sequence of $\MM$-convex functions $\nu_k:\Delta^d_n \rightarrow \mathbb{R}$ such that
\[
\lim_{k \to \infty}\nu_k(\alpha)=\nu(\alpha) \ \ \text{for all $\alpha \in \Delta^d_n$}.
\]
The sequence $\nu_k$ can be chosen so that $\nu_k=\nu$ in $\text{dom}(\nu)$ and $\nu_k<\nu_{k+1}$ outside $\text{dom}(\nu)$.
\end{lemma}

\begin{proof}
It is enough to prove the case when  $\nu$ is not the constant function $\infty$.
Write  $e_{ij}$ for the  standard unit vectors in $\mathbb{R}^{n^2}$.
Let $\varphi:  \Delta^d_{n^2} \to  \Delta^d_n$ and $\psi:  \Delta^d_{n^2} \to  \Delta^d_n$ be the restrictions of the linear maps from $\mathbb{R}^{n^2}$ to $\mathbb{R}^n$ given by
\[
\varphi(e_{ij})=e_i \ \ \text{and} \ \ \psi(e_{ij})=e_j.
\]
%\begin{align*}
%\varphi&:\mathbb{R}^{n^2} \longrightarrow \mathbb{R}^n, \qquad  e_{ij} \longmapsto e_i,\\
%\psi&:\mathbb{R}^{n^2} \longrightarrow \mathbb{R}^n, \qquad  e_{ij} \longmapsto e_j.
%\end{align*}
For any function $\mu:\Delta^d_n \to \mathbb{R} \cup \{\infty\}$, we define the function $\varphi^*\mu:\Delta^d_{n^2} \to \mathbb{R} \cup \{\infty\}$ by
\[
\varphi^*\mu(\beta)=\mu\Big(\varphi(\beta)\Big).
\]
For any function $\mu:\Delta^d_{n^2} \to \mathbb{R} \cup \{\infty\}$, we define the function $\psi_{*}\mu:\Delta^d_n \to \mathbb{R} \cup \{\infty\}$ by
\[
\psi_{*}\mu(\alpha)=\min\Big\{ \mu(\beta) \mid \psi(\beta)=\alpha\Big\}.
\]
Recall that the operations of splitting \cite[Section 4]{KMT} and aggregation \cite[Section 5]{KMT} preserve $\MM$-convexity of discrete functions.
Therefore, $\varphi^*$ and $\psi_{*}$ preserve $\MM$-convexity.
Now, given  $\nu$, set
\[
\nu_k=\psi_{*}(\ell_k+\varphi^*\nu),
\]
where $\ell_k$ is the restriction of the linear function on $\mathbb{R}^{n^2}$ defined by
\[
\ell_k(e_{ij})=\left\{\begin{array}{cc}
0 &\text{if $i=j$,} \\
k & \text{if $i \neq j$.}
\end{array}\right.
\]
 The existence theorem for nonnegative matrices with given row and column sums
 shows that the restriction of $\psi$ to any fiber of $\varphi$ is surjective  \cite[Corollary 1.4.2]{Brualdi}. 
Thus, the assumption that $\nu$ is not identically $\infty$ implies that  $\nu_k<\infty$  for every $k$.
It is straightforward to check that the sequence $\nu_k$ has the other required properties for large enough $k$.
\end{proof}

\begin{proof}[Proof of Theorem \ref{tropical}, (ii) $\Rightarrow$ (i)]
Let $f_t$ be a polynomial in $\mathrm{L}^d_n(\mathbb{K})$
whose tropicalization is $\nu$.
We show the $\MM$-convexity of $\nu$ by checking the local exchange property:
For any $\alpha,\beta \in \text{dom}(\nu)$ with $|\alpha-\beta|_1=4$, there are $i$ and $j$ satisfying 
\[
\alpha_i >\beta_i, \ \  \alpha_j <\beta_j \ \ \text{and} \ \  \nu(\alpha)+\nu(\beta) \ge \nu(\alpha-e_i+e_j)+\nu(\beta-e_j+e_i).
\]
Since  $|\alpha-\beta|_1=4$, 
we can find $\gamma$ in $\Delta^{d-2}_n$ and  indices $p,q,r,s$ in $[n]$ such that
such that
\[
\alpha=\gamma+e_p+e_q \ \ \text{and} \ \  \beta=\gamma+e_r+e_s  \ \ \text{and} \ \ \{p,q\} \cap \{r,s\}=\emptyset.
\]
Since $\partial^\gamma f_t$ is  stable,
the tropicalization of $\partial^\gamma f_t$ is $\MM$-convex by \cite[Theorem 4]{SIAM}.
The conclusion follows from the local exchange property for the tropicalization of $\partial^\gamma f_t$.
\end{proof}

\begin{proof}[Proof of Theorem \ref{classical}]
We prove (1) $\Rightarrow$ (3).
We first show the implication  in the special case
\[
\text{dom}(\nu)={n \brack d}.
\]
Since $\text{dom}(\nu)$ is $\MM$-convex, 
it is enough to prove that $\partial^\alpha g^\nu_{q}$ is has exactly one 
positive eigenvalue for all $\alpha \in {n \brack d-2}$ and all $0<q \le 1$.
Since $\mathscr{T}_n=\mathrm{Dr}(2,n)$ by  \cite[Theorem 4.3.5]{MS} and \cite[Definition 4.4.1]{MS},
the desired statement follows from Lemma \ref{ultrametricHR}.
This proves the first special case.
Now consider the second special case
\[
\text{dom}(\nu)=\Delta^d_n.
\]
%Write  $e_{ij}$ for the  standard unit vectors in $\mathbb{R}^{nd}$, and let $\phi$ be the linear map
%\[
%\phi:\mathbb{R}^{nd} \longrightarrow \mathbb{R}^n, \qquad  e_{ij} \longmapsto e_i.
%\]
%We define the \emph{polarization} of $\nu$ to be the discrete function $\Pi^\uparrow(\nu): \Delta^d_{nd} \rightarrow \mathbb{R} \cup \{\infty\}$  satisfying
%\[
%\text{dom}\big( \Pi^\uparrow(\nu)\big)={nd \brack d} \ \ \text{and} \ \ \Pi^\uparrow(\nu)= \nu\circ  \phi   \ \ \text{on ${nd \brack d}$}.
%\]
%\[
%\beta \longmapsto \left\{\begin{array}{cc}  \nu\circ \phi(\beta)& \text{if $\beta \in {nd \brack d}$,} \\  \infty& \text{if $\beta \notin {nd \brack d}$.} \end{array}\right.
%\]
%Recall that the operation of aggregation preserves $\MM$-convexity of discrete functions  \cite[Section 5]{KMT}.
%Therefore, the $\MM$-convexity of $\nu$ and the $\MM$-convexity of ${nd \brack d}$ together shows that
 By Lemma \ref{DiscretePolarization},
 %It is straightforward to check that the polarization  $\Pi^\uparrow(\nu)$ satisfies the symmetric exchange property.
the polarization  $\Pi^\uparrow \nu $ is an $\MM$-convex function with effective domain ${nd \brack d}$,
and hence we may apply the known implication  (1) $\Rightarrow$ (3) for $\Pi^\uparrow \nu$.
Therefore,
\[
\text{$\Pi^\uparrow_\delta(g^\nu_{q})=\frac{1}{d^d} \ g^{\Pi^\uparrow \nu}_{q}$ is a Lorentzian polynomial for $0<q \le 1$,}
\]
where $\delta=(d,\ldots,d)$.
Thus, by Proposition \ref{pols}, the polynomial  $g^{\nu}_{q}$ is  Lorentzian for all  $0 < q \le 1$, and
 the second special case is proved.
Next consider the third special case
\[
\text{$\text{dom}(\nu)$ is an arbitrary $\MM$-convex set and $q=1$.}
\]
 By Lemma \ref{DiscretePolarization}, the effective domain of $\Pi^\uparrow\nu$ is an $\MM$-convex set.
Therefore,   by Theorem \ref{charjump},  
\[
\text{$\Pi^\uparrow_\delta(g^\nu_{1})=\frac{1}{d^d} \ g^{\Pi^\uparrow\nu}_{1}$ is a Lorentzian polynomial.}
\]
Thus, by Proposition \ref{pols}, the polynomial $g^{\nu}_{1}$ is  Lorentzian, and
the the third special case is proved.
In the remaining case when $q<1$ and the effective domain of $\nu$ is arbitrary,
we express $\nu$ as the limit of $\MM$-convex functions $\nu_k$ with effective domain $\Delta^d_n$ using Lemma \ref{regularization}.
Since $q<1$, we have
\[
g^{\nu}_{q}= \lim_{k \to \infty} g^{\nu_k}_{q}.
\]
Thus the conclusion follows from the second special case applied to each $g^{\nu_k}_{q}$.

We prove (1) $\Rightarrow$ (2).
Introduce a positive real number $p$,
and consider the $\MM$-convex function $\frac{\nu}{p}$.
Applying  the known implication (1) $\Rightarrow$ (3), we see that the polynomial $g^{\nu/p}_{q}$ is Lorentzian for all $0<q \le 1$.
Therefore, by Proposition \ref{tot}, 
\[
R_p(g^{\nu/p}_{q})=\sum_{\alpha\in \text{dom}(\nu)}  (\alpha!)^p \hspace{0.5mm} {\delta \choose \alpha}^p \hspace{0.5mm} \frac{q^{\nu(\alpha)}}{\alpha!}w^\alpha %, \quad \delta=\sum_{i=1}^n de_i,
\]
is Lorentzian for all $0<p \le 1$.
Taking the limit $p$ to zero, we have (2).
%Therefore, the polynomial
%\[
%f_{\nu,q}=\lim_{p \to 0} R_p(g_{\frac{\nu}{p},q})
%\]
%is Lorentzian for all $0<q \le 1$.

We prove (2) $\Rightarrow$ (1) and (3) $\Rightarrow$ (1).
By the quantifier elimination for the theory of real closed fields,
 the polynomial
$f^{\nu}_{t}$  with coefficients in $\mathbb{K}$ is Lorentzian
if  (2) holds.
Similarly,
the polynomial $g^{\nu}_{t}$ is Lorentzian
if (3) holds.
Since
\[
\nu=\text{trop}(f^{\nu}_{t})=\text{trop}(g^{\nu}_{t}),
\]
the conclusion follows from (ii) $\Rightarrow$ (i) of Theorem \ref{tropical}.
\end{proof}

\begin{proof}[Proof of Theorem \ref{tropical}, (i) $\Rightarrow$ (ii)]
By Theorem \ref{classical}, $f^{\nu}_{q}$ is Lorentzian for all sufficiently small positive real numbers $q$.
Therefore, by the quantifier elimination for the theory of real closed fields, the polynomial $f^{\nu}_{t}$ %in $\mathbb{K}[w_1,\ldots,w_n]$
 is Lorentzian over $\mathbb{K}$.
Clearly, the tropicalization of $f^{\nu}_{t}$ is $\nu$.
\end{proof}

\begin{corollary}\label{CorollaryTropical}
Tropicalizations of Lorentzian polynomials  over $\mathbb{K}$ are $\MM$-convex, and all $\MM$-convex functions are limits of tropicalizations of Lorentzian polynomials over $\mathbb{K}$.
\end{corollary}

\begin{proof}
By Theorem \ref{tropical}, it is enough to show that any $\MM$-convex function $\nu: \Delta^d_n \to \mathbb{R} \cup \{\infty\}$
is a limit of $\MM$-convex functions  $\nu_k: \Delta^d_n \to \mathbb{Q} \cup \{\infty\}$.
By Lemma \ref{regularization}, we may suppose that
\[
\text{dom}(\nu)=\Delta^d_n.
\]
In this case, by Lemma \ref{DiscretePolarization}, the polarization $\Pi^\uparrow\nu$ is $\MM$-convex function satisfying
\[
\text{dom}\big(\Pi^\uparrow\nu\big)={nd \brack d}.
\]
In other words, $-\Pi^\uparrow\nu$ is a valuated matroid whose underlying matroid is uniform of rank $d$ on $nd$ elements.
Since the Dressian of the matroid is a rational polyhedral fan \cite[Section 4.4]{MS},
there are $\MM$-convex functions $\mu_k: \Delta^d_{nd} \to \mathbb{Q} \cup\{\infty\}$ satisfying
\[
\Pi^\uparrow\nu=\lim_{k \to \infty} \mu_k.
\]
By Lemma \ref{DiscretePolarization},  
 $\nu=\Pi^\downarrow \Pi^\uparrow \nu$ is the limit of $\MM$-convex functions $\Pi^\downarrow\mu_k: \Delta^d_{n} \to \mathbb{Q} \cup\{\infty\}$.
\end{proof}

\section{Examples and applications}

\subsection{Convex bodies and Lorentzian polynomials}\label{SectionConvex}

For any collection of convex bodies $\mathrm{K}=(\mathrm{K}_1,\ldots,\mathrm{K}_n)$ in $\mathbb{R}^d$,
consider the function 
\[
\text{vol}_\mathrm{K}:\mathbb{R}^n_{\ge 0} \longrightarrow \mathbb{R}, \qquad w  \longmapsto \mathrm{vol}(w_1\mathrm{K}_1+\cdots+w_n\mathrm{K}_n),
\]
where $w_1\mathrm{K}_1+\cdots+w_n\mathrm{K}_n$ is the \emph{Minkowski sum} and $\mathrm{vol}$ is the Euclidean volume. 
Minkowski noticed that the function $\text{vol}_\mathrm{K}$ is a degree $d$ homogeneous polynomial in  $w=(w_1,\ldots,w_n)$ with nonnegative coefficients.
We may write 
\[
\mathrm{vol}_\mathrm{K}(w)= \sum_{1 \le i_1,\ldots, i_d \le n} V(\mathrm{K}_{i_1},\ldots, \mathrm{K}_{i_d}) w_{i_1}\cdots w_{i_d}
=\sum_{\alpha\in \Delta^d_n} \frac{d!}{\alpha!} V_\alpha(\mathrm{K}) w^\alpha, 
\]
where   $V_\alpha(\mathrm{K})$ is, by definition, the \emph{mixed volume}
\[
V_\alpha(\mathrm{K})=V(\underbrace{\mathrm{K}_1,\ldots,\mathrm{K}_1}_{\alpha_1},\ldots, \underbrace{\mathrm{K}_n,\ldots,\mathrm{K}_n}_{\alpha_n})\coloneq \frac{1}{d!}\partial^\alpha \mathrm{vol}_\mathrm{K}.
\]
For any convex bodies $\mathrm{C}_0,\mathrm{C}_1,\ldots,\mathrm{C}_d$ in $\mathbb{R}^d$,
the mixed volume $V(\mathrm{C}_1,\mathrm{C}_2,\ldots,\mathrm{C}_d)$ is symmetric in its arguments and satisfies the relation
\[
V(\mathrm{C}_0+\mathrm{C}_1,\mathrm{C}_2,\ldots,\mathrm{C}_d)=V(\mathrm{C}_0,\mathrm{C}_2,\ldots,\mathrm{C}_d)+V(\mathrm{C}_1,\mathrm{C}_2,\ldots,\mathrm{C}_d).
\]
We refer to \cite{Schneider} for background on mixed volumes.
%We record here the following consequence of the Brunn-Minkowski theory.

\begin{theorem}\label{volumeLorentzian}
The \emph{volume polynomial} $\text{vol}_\mathrm{K}$
is a Lorentzian polynomial for any $\mathrm{K}=(\mathrm{K}_1,\ldots,\mathrm{K}_n)$. 
\end{theorem}

When combined with Theorem \ref{chars}, Theorem \ref{volumeLorentzian} implies the following statement.

\begin{corollary}\label{volumeSupport}
The support of   $\text{vol}_\mathrm{K}$ is an $\MM$-convex for any $\mathrm{K}=(\mathrm{K}_1,\ldots,\mathrm{K}_n)$. 
\end{corollary}

In other words, % for any $d$-dimensional projective variety $Y$, %and nef divisors $\mathrm{H}_1,\ldots,\mathrm{H}_n$ on $Y$,
the set of all $\alpha \in \Delta^d_n$ satisfying the non-vanishing condition
\[
V(\underbrace{\mathrm{K}_1,\ldots,\mathrm{K}_1}_{\alpha_1},\ldots, \underbrace{\mathrm{K}_n,\ldots,\mathrm{K}_n}_{\alpha_n})\neq 0
 \]
 is $\MM$-convex  for any convex bodies $\mathrm{K}_1,\ldots,\mathrm{K}_n$ in $\mathbb{R}^d$.
 
 \begin{remark}\label{RemarkConvexBodies}
The mixed volume $V(\mathrm{C}_{1},\ldots, \mathrm{C}_{d})$ is positive precisely when there are line segments $\ell_{i} \subseteq \mathrm{C}_{i}$ with linearly independent directions \cite[Theorem 5.1.8]{Schneider}.
Thus, when $\mathrm{K}$ consists of $n$ line segments in $\mathbb{R}^d$, Corollary \ref{volumeSupport} states the familiar fact that, for any configuration of $n$ vectors $\mathscr{A} \subseteq \mathbb{R}^d$, the collection of linearly independent $d$-subsets of $\mathscr{A}$ is the set of bases of a matroid.

The same reasoning shows that, in fact,  %not every Lorentzian polynomial is a volume polynomial of convex bodies.
 the basis generating polynomial of a  matroid on $[n]$ is the volume polynomial of $n$ convex bodies precisely when the matroid is regular.
 In particular, not every Lorentzian polynomial is a volume polynomial of convex bodies.
For example, the elementary symmetric polynomial
\[
w_1w_2+w_1w_3+w_1w_4+w_2w_3+w_2w_4+w_3w_4
\]
is not the volume polynomial of four convex bodies in the plane. %basis generating polynomial of the Fano matroid is not the volume polynomial of seven convex bodies in $\mathbb{R}^3$.
By the compactness theorem of Shephard for the affine equivalence classes of $n$ convex bodies \cite[Theorem 1]{Shephard}, the image of 
the set of volume polynomials of convex bodies in $\mathbb{P}\mathrm{L}^d_n$ is compact.
Thus, the displayed elementary symmetric polynomial is not even the limit of volume polynomials of convex bodies.

On the other hand, a collection $\mathrm{J} \subseteq {n \brack d}$ is the support of a volume polynomial of $n$ convex bodies in $\mathbb{R}^d$
if and only if $\mathrm{J}$ is the set of basis of a rank $d$ matroid on $[n]$ that is representable over $\mathbb{R}$.
For example, there are no seven convex bodies in $\mathbb{R}^3$ whose volume polynomial has the support  given by the set of bases of the Fano matroid.
\end{remark}

\begin{proof}[Proof of Theorem \ref{volumeLorentzian}]
By continuity of the volume functional \cite[Theorem 1.8.20]{Schneider}, we may suppose that  every convex body in  $\mathrm{K}$ is $d$-dimensional.
In this case, every coefficient of $\text{vol}_\mathrm{K}$ is positive.
Thus, by Theorem \ref{chars},
 it is enough to show that $\partial^\alpha \text{vol}_\mathrm{H}$ is Lorentzian for every $\alpha \in \Delta_n^{d-2}$.
For this we use a special case of the  Brunn-Minkowski theorem \cite[Theorem 7.4.5]{Schneider}: 
For any convex bodies $\mathrm{C}_3,\ldots,\mathrm{C}_d$ in $\mathbb{R}^d$,
the function
\[
%g:\mathbb{R}^n \longrightarrow \mathbb{R}, \qquad 
w \longmapsto V\Bigg(\sum_{i=1}^n w_i\mathrm{K}_i,\sum_{i=1}^n w_i\mathrm{K}_i,\mathrm{C}_3,\ldots,\mathrm{C}_d\Bigg)^{1/2}
\]
is concave on $\mathbb{R}^{n}_{>0}$.
In particular, the function 
\[
%w \longmapsto 
\Big(\hspace{0.5mm}\frac{2!}{d!} \hspace{0.5mm}\partial^\alpha \text{vol}_\mathrm{K}(w)\Big)^{1/2}= V\Bigg(\sum_{i=1}^n w_i\mathrm{K}_i,\sum_{i=1}^n w_i\mathrm{K}_i,\underbrace{\mathrm{K}_1,\ldots,\mathrm{K}_1}_{\alpha_1},\ldots, \underbrace{\mathrm{K}_n,\ldots,\mathrm{K}_n}_{\alpha_n}\Bigg)^{1/2}
\]
 is concave on $\mathbb{R}^n_{>0}$ for every $\alpha \in \Delta_n^{d-2}$.
The conclusion follows from Proposition \ref{logcon}.
\end{proof}

The \emph{Alexandrov--Fenchel inequality}  \cite[Section 7.3]{Schneider} states that
\[
V(\mathrm{C}_1,\mathrm{C}_2,\mathrm{C}_3,\ldots, \mathrm{C}_d)^2 \geq V(\mathrm{C}_1,\mathrm{C}_1,\mathrm{C}_3,\ldots, \mathrm{C}_d)V(\mathrm{C}_2,\mathrm{C}_2,\mathrm{C}_3,\ldots, \mathrm{C}_d).
\]
We show that an analog %of the Alexandrov-Fenchel inequality 
holds for any Lorentzian polynomial.
%We fix an arbitrary homogeneous polynomial
%\[
%f= \sum_{\alpha \in \Delta^d_n} \frac {c_\alpha}{\alpha!}w^\alpha.
%\]

\begin{proposition}\label{af}
If $f=\sum_{\alpha \in \Delta^d_n} \frac {c_\alpha}{\alpha!}w^\alpha$  is a Lorentzian polynomial, then
\[
c_\alpha^2\geq c_{\alpha+e_i-e_j} c_{\alpha-e_i+e_j} \ \ \text{for any $i,j \in [n]$ and any $\alpha \in \Delta^d_n$.}
\]
\end{proposition}

%Simple examples show that the converse does not hold. 
 
\begin{proof}
Consider the Lorentzian polynomial $\partial^{\alpha-e_i-e_j} f$.
Substituting $w_k$ by zero for all $k$  other than $i$ and $j$, 
we get the bivariate quadratic polynomial
\[
\frac{1}{2}c_{\alpha +e_i-e_j} w_i^2+c_\alpha w_iw_+\frac{1}{2}c_{\alpha -e_i+e_j} w_j^2.
\]
The displayed polynomial is Lorentzian by Theorem \ref{flow}, and hence $c_\alpha^2\geq c_{\alpha+e_i-e_j} c_{\alpha-e_i+e_j}$.
 \end{proof}

We may reformulate Proposition \ref{af} as follows. 
Let $f$ be a homogeneous polynomial of degree $d$ in $n$ variables.
The \emph{complete homogeneous form} of $f$ is the multi-linear function $F_f : (\RR^n)^d \rightarrow \RR$ defined by 
\[
F_f(v_1,\ldots, v_d) %= \frac 1 {d!} D_{v_1}\cdots D_{v_d} f
=  \frac 1 {d!} \frac {\partial}{\partial x_1}\cdots \frac {\partial}{\partial x_d}f(x_1v_1+\cdots+x_dv_d).
\]
Note that the complete homogeneous form of $f$ is symmetric in its arguments. 
By Euler's formula for homogeneous functions, we have
\[
F_f(w,w,\cdots, w)=f(w).
\]

\begin{proposition}\label{af2}
If $f$ is Lorentzian, then, for any $v_1 \in \RR^n$ and $v_2,\ldots,v_d \in \RR_{\geq 0}^n$, 
\[
F_f(v_1,v_2,v_3,\ldots, v_d)^2 \geq F_f(v_1,v_1,v_3,\ldots, v_d)F_f(v_2,v_2,v_3,\ldots, v_d).
\]
\end{proposition}

\begin{proof}
For every $k=1,\ldots,d$, we write $v_k=(v_{k1},v_{k2},\ldots,v_{kn})$, and set
\[
D_{k}=v_{k1}\frac{\partial}{\partial w_1}+v_{k2}\frac{\partial}{\partial w_2}+\cdots+v_{kn}\frac{\partial}{\partial w_n}.
\]
By Corollary \ref{derivatives},
the quadratic polynomial $D_{3}\cdots D_{d} f$ is Lorentzian.
We may suppose that the Hessian $\mathscr{H}$ of the quadratic polynomial is not identically zero and $v_2^T\mathscr{H} v_2 >0$. 
Note that
\[
v_i^T \mathscr{H} v_j= D_{i}D_{j}D_{3}\cdots D_{d} f=d! F_f(v_i,v_j,v_3,\ldots,v_d) \ \ \text{for any $i$ and $j$.}
\]
Since $\mathscr{H}$ has exactly one positive eigenvalue, 
the conclusion follows from Cauchy's interlacing theorem.
%by Cauchy's interlacing theorem,
%\[
%\text{
%$\left(\begin{smallmatrix}
%v_1^T \mathscr{H} v_1 & v_1^T \mathscr{H} v_2 \\ v_1^T \mathscr{H} v_2 & v_2^T \mathscr{H} v_2
%\end{smallmatrix} \right)$ has exactly one positive eigenvalue.}
%\]
%The conclusion follows by taking the determinant of the $2 \times  2$ matrix.
\end{proof}

\subsection{Projective varieties and Lorentzian polynomials}\label{SectionProjective}

Let $Y$ be a $d$-dimensional irreducible projective variety over an algebraically closed field $\mathbb{F}$.
If $\mathrm{D}_1,\ldots,\mathrm{D}_d$ are Cartier divisors on $Y$, the \emph{intersection product} $(\mathrm{D}_1\cdot \ldots \cdot \mathrm{D}_d)$ is an integer  defined 
by the following properties:
\begin{enumerate}[--]\itemsep 5pt
\item the product $(\mathrm{D}_1\cdot \ldots \cdot \mathrm{D}_d)$ is symmetric and multilinear as a function of its arguments,
\item  the product $(\mathrm{D}_1\cdot \ldots \cdot \mathrm{D}_d)$ depends only on the linear equivalence classes of the $D_i$, and
\item if $\mathrm{D}_1,\ldots,\mathrm{D}_n$ are effective divisors meeting transversely at smooth points of $Y$, then
\[
(\mathrm{D}_1\cdot \ldots \cdot \mathrm{D}_d)=\# \mathrm{D}_1 \cap \ldots \cap \mathrm{D}_d.
\]
\end{enumerate}
Given an irreducible subvariety $X \subseteq Y$ of dimension $k$, the intersection product 
\[
(\mathrm{D}_1 \cdot \ldots \cdot \mathrm{D}_k \cdot X)
\]
is then defined by replacing each divisor $\mathrm{D}_i$ with a linearly equivalent Cartier divisor whose support does not contain $X$ and intersecting the restrictions of $\mathrm{D}_i$ in $X$.
The definition of the intersection product linearly extends to $\mathbb{Q}$-linear combination of Cartier divisors, called $\mathbb{Q}$-divisors \cite[Section 1.3]{Lazarsfeld}.
If $\mathrm{D}$ is a $\mathbb{Q}$-divisor on $Y$, we write $(\mathrm{D})^d$ for the self-intersection  $(\mathrm{D} \cdot \ldots \cdot \mathrm{D})$.
For a gentle introduction to Cartier divisors and their intersection products, we refer to \cite[Section 1.1]{Lazarsfeld}.  
See \cite{Fulton} for a comprehensive study.

Let $\mathrm{H}=(\mathrm{H}_1,\ldots,\mathrm{H}_n)$ be a collection of $\mathbb{Q}$-divisors on $Y$. 
We define the \emph{volume polynomial} of $\mathrm{H}$ by 
\[
\text{vol}_\mathrm{H}(w)=(w_1H_1+\cdots+w_nH_n)^d
=\sum_{\alpha\in \Delta^d_n} \frac{d!}{\alpha!} V_\alpha(\mathrm{H}) w^\alpha, 
\]
where  $V_\alpha(\mathrm{H})$ is the intersection product
\[
V_\alpha(\mathrm{H})=(\underbrace{\mathrm{H}_1 \cdot \ldots \cdot \mathrm{H}_1}_{\alpha_1} \cdot \ldots \cdot  \underbrace{\mathrm{H}_n \cdot \ldots \cdot \mathrm{H}_n}_{\alpha_n})
=\frac{1}{d!}\partial^\alpha \mathrm{vol}_\mathrm{H}.
\]
A $\mathbb{Q}$-divisor $\mathrm{D}$ on $Y$ is said to be \emph{nef} if $(\mathrm{D} \cdot C) \ge 0$ for every irreducible curve $C$ in $Y$ \cite[Section 1.4]{Lazarsfeld}.

\begin{theorem}\label{intersectionLorentzian}
If $\mathrm{H}_1,\ldots,\mathrm{H}_n$ are nef divisors on $Y$, then 
$\text{vol}_\mathrm{H}(w)$
is a Lorentzian polynomial.
\end{theorem}

When combined with Theorem \ref{chars}, Theorem \ref{intersectionLorentzian} implies the statement. %following analog of Corollary \ref{volumeSupport}.%theorem of Castillo, Li, and Zhang
%\cite[Proposition 5.4]{CLZ}.

\begin{corollary}\label{corollaryIntersection}
If $\mathrm{H}_1,\ldots,\mathrm{H}_n$ are nef divisors on $Y$, then 
the support of $\text{vol}_\mathrm{H}(w)$
is $\MM$-convex.
\end{corollary}

In other words, % for any $d$-dimensional projective variety $Y$, %and nef divisors $\mathrm{H}_1,\ldots,\mathrm{H}_n$ on $Y$,
the set of all $\alpha \in \Delta^d_n$ satisfying the non-vanishing condition
\[
 (\underbrace{\mathrm{H}_1 \cdot \ldots \cdot \mathrm{H}_1}_{\alpha_1} \cdot \ldots \cdot  \underbrace{\mathrm{H}_n \cdot \ldots \cdot \mathrm{H}_n}_{\alpha_n}) \neq 0
 \]
 is $\MM$-convex  for any $d$-dimensional projective variety $Y$ and any nef divisors $\mathrm{H}_1,\ldots,\mathrm{H}_n$ on $Y$.
 Corollary \ref{corollaryIntersection} implies a result of Castillo \emph{et al.} \cite[Proposition 5.4]{CLZ},
 which says that the support of the multidegree of any irreducible mutiprojective variety is a discrete polymatroid.
%\[
%\Big\{\alpha \in \NN^n \mid (\underbrace{\mathrm{H}_1 \cdot \ldots,\mathrm{H}_1}_{\alpha_1} \cdot \ldots \cdot  \underbrace{\mathrm{H}_n \cdot \ldots \cdot \mathrm{H}_n}_{\alpha_n}) \neq 0 \Big\}
%\]

\begin{remark}
Let $\mathscr{A}=\{v_1,\ldots,v_n\}$ be a collection of $n$ vectors in $\mathbb{F}^d$.
In \cite[Section 4]{HW}, one can find a $d$-dimensional projective variety $Y_\mathscr{A}$ and nef divisors $\mathrm{H}_1,\ldots,\mathrm{H}_n$ on $Y_\mathscr{A}$ such that
\[
\text{vol}_\mathrm{H}(w)=\sum_{\alpha \in {n \brack d}} c_\alpha w^\alpha,
\]
where $c_\alpha=1$ if $\alpha$ corresponds to a linearly independent subset of $\mathscr{A}$ and $c_\alpha=0$ if otherwise.
Thus, in this case, Corollary \ref{corollaryIntersection} states the familiar fact that the collection of linearly independent $d$-subsets of $\mathscr{A} \subseteq \mathbb{F}^d$ is the set of bases of a matroid.
\end{remark}

\begin{proof}[Proof of Theorem \ref{intersectionLorentzian}]
By Kleiman's theorem \cite[Section 1.4]{Lazarsfeld}, every nef divisor is a limit of ample divisors, and 
we may suppose that every divisor in $\mathrm{H}$ is very ample.
In this case, every coefficient of $\text{vol}_\mathrm{H}$ is positive.
Thus, by Theorem \ref{chars},
 it is enough to show that $\partial^\alpha \text{vol}_\mathrm{H}$ is Lorentzian for every $\alpha \in \Delta_n^{d-2}$.
 Note that
\[
%w \longmapsto 
\frac{2!}{d!} \hspace{0.5mm} \partial^\alpha \text{vol}_\mathrm{H}(w)=\Bigg(\sum_{i=1}^n w_i\mathrm{H}_i \cdot \sum_{i=1}^n w_i\mathrm{H}_i \cdot \underbrace{\mathrm{H}_1 \cdot \ldots \cdot \mathrm{H}_1}_{\alpha_1} \cdot \ldots \cdot \underbrace{\mathrm{H}_n\cdot \ldots \cdot \mathrm{H}_n}_{\alpha_n}\Bigg).
\]
By Bertini's theorem \cite[Section 3.3]{Lazarsfeld}, there is an irreducible surface $S \subseteq Y$ such that
\[
\frac{2!}{d!} \hspace{0.5mm}  \partial^\alpha \text{vol}_\mathrm{H}(w)=\Bigg(\sum_{i=1}^n w_i\mathrm{H}_i \cdot \sum_{i=1}^n w_i\mathrm{H}_i \cdot S\Bigg).
\]
If $S$ is smooth, then the Hodge index theorem \cite[Theorem V.1.9]{Hartshorne}
%Now the Hodge index theorem \cite[Example 15.2.4]{Fulton} for any resolution of singularities of $S$ 
shows that
the displayed quadratic form has exactly one positive eigenvalue.
% is Lorentzian for every $\alpha \in \Delta_n^{d-2}$.
%The conclusion follows from Proposition \ref{logcon}.
In general, the Hodge index theorem applied to any resolution of singularities of $S$
implies the one positive eigenvalue condition, by the projection formula \cite[Example 2.4.3]{Fulton}.
\end{proof}

How large is the set of volume polynomials of projective varieties within the set of Lorentzian polynomials?
We formulate various precise versions of this question.
%We write $\mathrm{V}^d_n$ for the set of volume polynomials of $n$ nef divisors on a $d$-dimensional projective variety.
Let $\mathrm{V}^d_n(\mathbb{F})$  be the set of volume polynomials of $n$ nef divisors on a $d$-dimensional projective variety over $\mathbb{F}$,
and let $\mathrm{V}^d_n=\bigcup_{\mathbb{F}} \mathrm{V}^d_n(\mathbb{F})$,
where the union is over all algebraically closed fields.
%be the union of $\mathrm{V}^d_n(\mathbb{F})$  over all $\mathbb{F}$.
%polynomials of the form $\text{vol}_\mathrm{H}(w)$, where $\mathrm{H}$ is a collection of $n$ nef divisors on a $d$-dimensional projective variety over $\mathbb{F}$.

\begin{question}\label{RealizationQuestion}
Fix any algebraically closed field $\mathbb{F}$.
\begin{enumerate}[(1)]\itemsep 5pt
\item Is there a polynomial in $\mathrm{L}^d_n$ that is not in  the closure of $\mathrm{V}^d_n$?%the limit of volume polynomials of projective varieties?
\item Is there a polynomial in $\mathrm{L}^d_n$ that is not in  the closure of $\mathrm{V}^d_n(\mathbb{F})$?%the limit of volume polynomials of projective varieties?
\item Is there a polynomial in $\mathrm{L}^d_n \cap \mathbb{Q}[w]$ that is not in $\mathrm{V}^d_n$?%the limit of volume polynomials of projective varieties?
\item Is there a polynomial in $\mathrm{L}^d_n \cap \mathbb{Q}[w]$ that is not in $\mathrm{V}^d_n(\mathbb{F})$?%the limit of volume polynomials of projective varieties?
%\item Is there a polynomial in $\mathrm{L}^d_n$ that is not in  $\mathrm{V}^d_n$?%the limit of volume polynomials of projective varieties?
%\item Is there a polynomial in $\mathrm{L}^d_n$ that is not in  $\mathrm{V}^d_n(\mathbb{F})$?%the limit of volume polynomials of projective varieties?
\end{enumerate}
\end{question}

Shephard's construction in \cite[Section 3]{Shephard} shows that every polynomial in $\mathrm{L}^d_{2} \cap \mathbb{Q}[w]$
is the volume polynomial of a pair of $d$-dimensional convex polytopes with rational vertices. 
Thus, by \cite[Section 5.4]{Toric}, we have
\[
\text{$\mathrm{L}^d_{2}\cap \mathbb{Q}[w]=\mathrm{V}^d_{2}=\mathrm{V}^d_2(\mathbb{F})$ for any $d$ and any $\mathbb{F}$.}
\]
A similar reasoning based on the construction of \cite[Section I]{Heine} shows that
\[
\text{$\mathrm{L}^2_{3}\cap \mathbb{Q}[w]=\mathrm{V}^2_{3}=\mathrm{V}^2_3(\mathbb{F})$ for  any $\mathbb{F}$.}
\]
When $n \ge 4$, 
not every Lorentzian polynomial is the limit of a sequence of volume polynomials of rational convex polytopes (Remark \ref{RemarkConvexBodies}),
and we do not know how to answer any of the above questions.\footnote{After the completion of this paper, we noticed that the closure of $\mathrm{V}^3_3(\mathbb{F})$ is strictly smaller than $\mathrm{L}^3_3$ for any $\mathbb{F}$. This answers Question \ref{RealizationQuestion}. Specifically, the Lorentzian cubic  $f=14x^3+6x^2y+24x^2z+12xyz+6xz^2+3yz^2$ is not in the closure of $\mathrm{V}^3_3(\mathbb{F})$. That $f$ is not in the closure of $\mathrm{V}^3_3(\mathbb{F})$ can be shown using the \emph{reverse Khovanskii-Teissier inequality} \cite[Theorem 5.7]{Lehmann-Xiao}:
For any nef divisors $\mathrm{H}_1,\mathrm{H}_2,\mathrm{H}_3$ on a $d$-dimensional  projective variety and any $k \le d$, we have
\[
{d \choose k} \ (\mathrm{H}_2^k \cdot \mathrm{H}_1^{d-k}) \ (\mathrm{H}_1^k \cdot  \mathrm{H}_3^{d-k}) \ge (\mathrm{H}_1^d) \ (\mathrm{H}_2^k \cdot \mathrm{H}_3^{d-k}).
\]
The complex analytic proof of the inequality in \cite{Lehmann-Xiao} relies on the Calabi-Yau theorem \cite{Yau}. The algebraic proof of the inequality in \cite{JiangLi} using Okounkov bodies works over any algebraically closed field.
The theory of toric varieties shows that the volume polynomial of any set of convex bodies 
is the limit of a sequence of volume polynomials of nef divisors on projective varieties \cite[Section 5.4]{Toric}.
Thus, the Lorentzian cubic $f$ provides a counterexample to Gurvits' conjecture that a strongly log-concave homogeneous polynomial in three variables with nonnegative coefficients is the volume polynomial of three convex bodies
\cite[Conjecture 4.1]{GurvitsL}.} %That $f$ is in $\mathrm{L}^3_3$ is a straightforward computation.}

\subsection{Potts model partition functions and Lorentzian polynomials}\label{SecqP}

The \emph{$q$-state Potts model}, or the \emph{random-cluster model}, of a graph is a much studied class of measures introduced by Fortuin and Kasteleyn \cite{FK}. 
We refer to \cite{Grimmett} for a comprehensive introduction to random-cluster models. 
%A \emph{matroid $q$-state Potts model} is a measure whose partition function is 
%$$
%\mathrm{Z}_\mathrm{M}(x_1w_1, \ldots, x_nw_n, 1)/\mathrm{Z}_\mathrm{M}(x_1, \ldots, x_n, 1), \ \ \ \ x\in \RR_{>0}^n, q > 0, 
%$$
%where we for simplicity have assumed $E=\{1,\ldots,n\}$. 

% Random cluster measures for graphs are conjectured to be negative dependent for $0< q\leq 1$, see \cite{Grimmett, Pem, Wagner08}, %(for example $\mathrm{PNC}_+$ and $\mathrm{NA}_+$), 
% but so far no substantial results on negative dependence have been proved.

Let $\mathrm{M}$ be a matroid on $[n]$, and let
  $\mathrm{rk}_\mathrm{M}$ be the rank function of $\mathrm{M}$.
For a nonnegative integer $k$ and a positive real parameter $q$, consider the degree $k$ homogeneous polynomial in $n$ variables
\[
\mathrm{Z}^k_{q,\mathrm{M}}(w)=%\mathrm{Z}^k_{\mathrm{M}}(w_1,\ldots,w_n)
\sum_{A \in {n \brack k}} q^{-\mathrm{rk}_\mathrm{M}(A)} w^A, \quad w=(w_1,\ldots,w_n).
\]
%where the sum is over all $k$-element subsets $A$ of $[n]$. 
We define the \emph{homogeneous multivariate Tutte polynomial} of $\mathrm{M}$ by
\[
\mathrm{Z}_{q,\mathrm{M}}(w_0,w_1,\ldots,w_n)= \sum_{k=0}^n
\mathrm{Z}_{q,\mathrm{M}}^{k}(w) \hspace{0.5mm} w_0^{n-k},
\]
which is a homogeneous polynomial of degree $n$ in $n+1$ variables. 
When $\mathrm{M}$ is the cycle matroid of a graph $G$, the polynomial obtained from $\mathrm{Z}_{q,\mathrm{M}}$ by setting $w_0=1$ is the partition function of the $q$-state Potts model associated to $G$ \cite{Sokal}.

Since the rank function of a matroid is $\mathrm{M}^\natural$-concave,
the normalized rank generating function of $\mathrm{M}$ is Lorentzian when the parameter $q$ satisfies $0<q\le 1$, see
Example \ref{examplerank}.
In this subsection, we prove the following refinement.

\begin{theorem}\label{mainpotts}
For any matroid $\mathrm{M}$ and $0<q \le 1$, the polynomial $\mathrm{Z}_{q,\mathrm{M}}$ is Lorentzian. 
\end{theorem}

We prepare the proof with two simple lemmas.

\begin{lemma}\label{PottsSupport}
The support of  $\mathrm{Z}_{q,\mathrm{M}}$ is $\MM$-convex for all $0<q \le 1$. 
\end{lemma}

\begin{proof}
Writing $\mathrm{Z}^\natural_{q,\mathrm{M}}$ for the polynomial obtained from $\mathrm{Z}_{q,\mathrm{M}}$ by setting $w_0=1$, we have
\[
\text{supp}\big(\mathrm{Z}_{q,\mathrm{M}}^\natural\big)=\big\{0,1\big\}^n. %\ \ \text{where  $\mathrm{Z}^\natural_{q,\mathrm{M}} \coloneq \mathrm{Z}_{q,\mathrm{M}}|_{w_0=1}$.}
\]
%It is clear that the support of $\mathrm{Z}_{q,\mathrm{M}}^\natural$ is $\MM^\natural$-convex.
It is straightforward to verify the augmentation property in Lemma \ref{polyind} for $\{0,1\}^n$.
\end{proof}

For a nonnegative integer $k$ and a subset $S \subseteq [n]$, 
we define a degree $k$ homogeneous polynomial $e^k_S(w)$ by the equation
\[
\sum_{k=0}^n e^k_S(w)=\sum_{A \subseteq S} w^A.
\]
In other words,   $e^k_S(w)$ is the $k$-th elementary symmetric polynomial
in the variables $\{w_i\}_{i \in S}$.

\begin{lemma}\label{Cauchy-Schwarz}
If $S_1 \sqcup \ldots \sqcup S_m$ is a partition of $[n]$ into $m$ nonempty parts, then
\[
\frac{1}{n}e_{[n]}^1(w)^2 \le e^{1}_{S_1}(w)^2+\cdots+e^1_{S_m}(w)^2 \ \ \text{for all $w \in \mathbb{R}^n$.}
\]
\end{lemma}

\begin{proof}
Since $m \le n$, it is enough to prove the statement when $m=n$.
In this case, we have
\[
(w_1+\cdots+w_n)^2 \le n (w_1^2 +\cdots+w_n^2),
\]
by the Cauchy-Schwarz inequality for the vectors $(1,\ldots,1)$ and $(w_1,\ldots,w_n)$ in $\mathbb{R}^n$.
\end{proof}

\begin{proof}[Proof of Theorem \ref{mainpotts}]
Let $\alpha$ be an element of $ \Delta^{n-2}_{n+1}$.
 By  Theorem \ref{chars} and Lemma \ref{PottsSupport},
 the proof reduces to the statement that
 the quadratic form
 $\partial^\alpha \mathrm{Z}_{q,\mathrm{M}}$ is stable.
We prove the statement by induction on $n$.
The assertion is clear when $n=1$, so suppose  $n \ge 2$.
When $i \neq 0$, we have
\[
\partial_i\mathrm{Z}_{q,\mathrm{M}}= q^{-\text{rk}_\mathrm{M}(i)}\mathrm{Z}_{q,\mathrm{M}/i},
\]
where $\mathrm{M}/i$ is the contraction of $\mathrm{M}$ by $i$ \cite[Chapter 3]{Oxley}.
%which is stable by induction on $n$.
%Therefore, it is enough to consider the case 
%\[
%\alpha=(n-2,0,\ldots,0).
%\]
%In this case, $\partial^\alpha \mathrm{Z}_{q,\mathrm{M}}$ is the quadratic form
Thus, it is enough to prove that the following quadratic form is stable:
\[
%\partial^\alpha \mathrm{Z}_{q,\mathrm{M}}=
\frac{n!}{2} w_0^2+(n-1)! \hspace{0.5mm}\mathrm{Z}_{q,\mathrm{M}}^1(w)w_0+ (n-2)!\hspace{0.5mm} \mathrm{Z}_{q,\mathrm{M}}^2(w).
\]
%Recall that a homogeneous polynomial $f$ with nonnegative coefficient in $n+1$ variables is stable if and only if the univariate polynomial $f(xu-v)$ has only real zeros for all $v \in \mathbb{R}^{n+1}$ for some $u \in \mathbb{R}^{n+1}_{\ge 0}$ satisfying $f(u)>0$.
Recall that a homogeneous polynomial $f$ with nonnegative coefficients in $n+1$ variables is stable if and only if the univariate polynomial
 $f(xu-v)$ has only real zeros for all $v \in \mathbb{R}^{n+1}$ for some $u \in \mathbb{R}^{n+1}_{\ge 0}$ satisfying $f(u)>0$.
Therefore, % it suffices to show that the discriminant of the displayed quadratic form with respect to $w_0$ is nonnegative:
%\[
%\mathrm{p}^1_\mathrm{A}(w)^2 \ge  \frac{2n}{n-1} \mathrm{p}^2_\mathrm{A}(w) \ \ \text{for all $w \in \mathbb{R}^n$.}
%\]
%As in the proof of Theorem \ref{M-matrix},
 it suffices to show that the discriminant of the displayed quadratic form with respect to $w_0$ is nonnegative:
\[
 \mathrm{Z}_{q,\mathrm{M}}^1(w)^2 \geq 2 \frac n {n-1}  \mathrm{Z}_{q,\mathrm{M}}^2(w) \ \  \text{for all  $w \in \mathbb{R}^n$}.
\]
We prove the inequality after making the change of variables
\[
w_i \longmapsto \left\{\begin{array}{cl} w_i & \text{if $i$ is a loop in $\mathrm{M}$,}\\ qw_i& \text{if $i$ is not a loop in $\mathrm{M}$.} \end{array}\right.
\]
%We now have $ \mathrm{Z}_{q,\mathrm{M}}^1(w)=e^1_{[n]}(w)$, and,
Write $L \subseteq [n]$ for the set of loops and  $P_1,\ldots,P_\ell \subseteq [n] \setminus L$ for the parallel classes in $\mathrm{M}$ \cite[Section 1.1]{Oxley}.
The above change of variables gives
\[
\mathrm{Z}_{q,\mathrm{M}}^1(w)=e^1_{[n]}(w) \ \ \text{and} \ \  \mathrm{Z}_{q,\mathrm{M}}^2(w)=e^2_{[n]}(w)-(1-q)\big(e^2_{P_1}(w)+\cdots+e^2_{P_\ell}(w)\big).
 \]
When $q=1$,  the desired inequality directly follows from the case $m=n$ of Lemma \ref{Cauchy-Schwarz}.
Therefore, when proving the desired inequality for an arbitrary $0<q \le 1$, we may assume that
\[
e^2_{P_1}(w)+\cdots+e^2_{P_\ell}(w)<0.
\]
Therefore, exploiting the monotonicity of $ \mathrm{Z}_{q,\mathrm{M}}^2$ in $q$, the desired inequality reduces to
\[
(n-1)e^1_{[n]}(w)^2 - 2n \Big(e^2_{[n]}(w)-e^2_{P_1}(w)-\cdots-e^2_{P_\ell}(w)\Big) \ge 0.
\]
Note that the left-hand side of the above inequality simplifies to
\[
n \Big(e^1_{P_1}(w)^2+\cdots+e^1_{P_\ell}(w)^2+\sum_{i \in L} w_i^2 \Big)-e^1_{[n]}(w)^2.
\]
The conclusion now follows from Lemma \ref{Cauchy-Schwarz}.
\end{proof}

%Several conjectures have been made regarding unimodality and log-concavity of sequences arising in matroid theory.   Only recently have some of these been solved using combinatorial Hodge theory \cite{AHK,HSW}. A conjecture that has resisted the approach of \cite{AHK} is the strongest conjecture of Mason regarding independent sets in a matroid \cite{Mason}. The purpose of this paper is to give a self-contained proof of the strongest conjecture avoiding, but inspired by, Hodge theory. We prove that the Hessian of the homogenous multivariate Tutte polynomial (or the $q$-state Potts model partition function) of a matroid has exactly one positive eigenvalue on the positive orthant when $0<q \le 1$. 
%In a forthcoming paper we will take a more general approach and see that the results proved in this paper fit into a wider context\footnote{In related  forthcoming papers, Anari, Liu, Gharan and Vinzant have independently developed methods  that overlap with our work. In particular, they also prove Mason's conjecture \eqref{MasonThird}.}. 

Mason \cite{Mason} offered the following three conjectures of increasing strength.
Several authors studied correlations in matroid theory partly in pursuit of these conjectures \cite{SW,Wagner08,BBL,KN10,KN11}.

\begin{conjecture}\label{ConjectureMason}
For any matroid $\mathrm{M}$ on $[n]$ and any positive integer $k$,
\begin{enumerate}[(1)]\itemsep 5pt
\item\label{MasonFirst} %The sequence $\{i_k\}_{0 \le k \le d}$ is log-concave:
$
I_k(\mathrm{M})^2 \ge I_{k-1}(\mathrm{M})I_{k+1}(\mathrm{M}),
$
\item\label{MasonSecond} %The sequence $\{k!i_k\}_{0 \le k \le d}$ is log-concave:
$
I_k(\mathrm{M})^2 \ge \frac{k+1}{k} I_{k-1}(\mathrm{M})I_{k+1}(\mathrm{M}),
$
\item\label{MasonThird} %The sequence $\{{n\choose k}(denominator!)i_k\}_{0 \le k \le d}$ is log-concave:
$
I_k(\mathrm{M})^2 \ge \frac{k+1}{k}\frac{n-k+1}{n-k} I_{k-1}(\mathrm{M})I_{k+1}(\mathrm{M}),
$
\end{enumerate}
where $I_k(\mathrm{M})$ is the number of $k$-element independent sets of $\mathrm{M}$.  
\end{conjecture}

Conjecture \ref{ConjectureMason} (\ref{MasonFirst}) 
was proved 
 in \cite{AHK}, and Conjecture \ref{ConjectureMason} (\ref{MasonSecond}) was proved in \cite{HSW}.
Note that Conjecture \ref{ConjectureMason} (\ref{MasonThird}) may be written
 \[
\frac{I_k(\mathrm{M})^2}{{n \choose k}^2} \ge \frac{I_{k+1}(\mathrm{M})}{{n \choose k+1}}\frac{I_{k-1}(\mathrm{M})}{{n \choose k-1}}, 
\]
and the equality holds when all $(k+1)$-subsets of $[n]$ are independent in $\mathrm{M}$.
Conjecture \ref{ConjectureMason} (\ref{MasonThird}) is known to hold when 
$n$ is at most $11$ or $k$ is at most $5$  \cite{KN11}.  
See \cite{Seymour, Dowling, Mahoney, Zhao, HK,HS, Lenz} for other partial results.

\begin{theorem}\label{TheoremMasonThird}
For any matroid $\mathrm{M}$ on $[n]$ and any positive integer $k$,
 \[
\frac{I_k(\mathrm{M})^2}{{n \choose k}^2} \ge \frac{I_{k+1}(\mathrm{M})}{{n \choose k+1}}\frac{I_{k-1}(\mathrm{M})}{{n \choose k-1}},
\]
where $I_k(\mathrm{M})$ is the number of $k$-element independent sets of $\mathrm{M}$.  
\end{theorem}

In  \cite{BH}, direct proofs of  Theorems \ref{mainpotts} and \ref{TheoremMasonThird} were given.\footnote{An independent proof of  \ref{TheoremMasonThird} was given by Anari \emph{et al.} in  \cite{ALGVIII}.} 
Here we deduce Theorem \ref{TheoremMasonThird} from the Lorentzian property of %the polynomial
\[
f_\mathrm{M}(w_0,w_1,\ldots,w_n)%=\sum_{k=0}^n f^k_\mathrm{M}(w) w_0^{n-k}
=\sum_{A \in \mathscr{I}(\mathrm{M})} w^A w_0^{n-|A|}, \quad w=(w_1,\ldots,w_n),
\]
where $\mathscr{I}(\mathrm{M})$ is the collection of independent sets of $\mathrm{M}$.

\begin{proof}[Proof of Theorem \ref{TheoremMasonThird}]
The polynomial $f_\mathrm{M}$ is Lorentzian by Theorem \ref{mainpotts} and the identity
\[
f_\mathrm{M}(w_0,w_1,\ldots,w_n)=\lim_{q \to 0} \mathrm{Z}_{q,\mathrm{M}}(w_0,qw_1,\ldots,qw_n). 
\]
Therefore, by Theorem \ref{flow}, the bivariate polynomial obtained from $f_\mathrm{M}$ by setting $w_1=\cdots=w_n$ is Lorentzian.
The conclusion follows from the fact that a bivariate homogeneous polynomial with nonnegative coefficients is Lorentzian if and only if the sequence of coefficients form an ultra log-concave sequence with no internal zeros.
\end{proof}

The \emph{Tutte polynomial} of a matroid $\mathrm{M}$ on $[n]$ is the bivariate polynomial
\[
\mathrm{T}_\mathrm{M}(x,y)=\sum_{A  \subseteq [n]} (x-1)^{\text{rk}_\mathrm{M}([n])-\text{rk}_\mathrm{M}(A)} (y-1)^{|A|-\text{rk}_\mathrm{M}(A)}.
\]
Theorem \ref{mainpotts} reveals several nontrivial inequalities satisfied by the coefficients of the Tutte polynomial.
For example, if we write
\[
w^{\text{rk}_\mathrm{M}([n])}\mathrm{T}_\mathrm{M}\Big(1+\frac{q}{w},1+w\Big)=\sum_{k=0}^n \Big( \sum_{A \in {n \brack k}} q^{\text{rk}_\mathrm{M}([n])-\text{rk}_\mathrm{M}(A)}\Big) w^k=\sum_{k=0}^n c_q^k(\mathrm{M}) w^k, % q^d \sum_{A  \subseteq E} q^{-r(A)} w^{|A|}.
\]
then the sequence $c_q^k(\mathrm{M})$ is ultra log-concave whenever $0 \le q \le 1$.
This and other results in this subsection are recently extended to flag matroids in \cite{EurHuh}.

\subsection{$\MM$-matrices and Lorentzian polynomials}\label{secMm}

We write $\mathrm{I}_{n}$ for the $n \times n$ identity matrix,
 $\mathrm{J}_{n}$ for  the $n \times n$ matrix all of whose entries are $1$,
 and $1_n$ for the $n \times 1$ matrix all of whose entries are $1$.
Let $A=(a_{ij})$ be an $n \times n$ matrix with real entries.
The following conditions are equivalent if $a_{ij} \le 0$  for all $i \neq j$ \cite[Chapter 6]{BP}:
\begin{enumerate}[--]\itemsep 5pt
%\item The sum of all the $k \times k$ principal minors of $A$ is nonnegative for $k=1,\ldots,n$.
%\item The matrix $A+D$ is nonsingular for every positive diagonal matrix $D$.
\item The real part of each nonzero eigenvalue of $\mathrm{A}$ is positive.
\item The real part of each eigenvalue of $\mathrm{A}$ is nonnegative.
\item All the principal minors of $\mathrm{A}$ are nonnegative.
\item Every real eigenvalue of $\mathrm{A}$ is nonnegative.
\item The matrix $\mathrm{A}+\epsilon \hspace{0.5mm} \mathrm{I}_n$ is nonsingular for every $\epsilon>0$.
\item The univariate polynomial $\det(x\mathrm{I}_n+\mathrm{A})$ has nonnegative coefficients.
\end{enumerate}
%Note that every $\MM$-matrix is a limit of nonsingular $\MM$-matrices.
The matrix $\mathrm{A}$ is an \emph{$\MM$-matrix} if $a_{ij} \le 0$ for all $i \neq j$ and
if it satisfies any one of the above conditions.
One can find $50$ different  characterizations of nonsingular $\MM$-matrices in  \cite[Chapter 6]{BP}.
Among others, we will use  the  $29$-th condition:
\begin{enumerate}[--]\itemsep 5pt
\item There are positive diagonal matrices $D$ and $D'$ such that $D A D'$ has all diagonal entries $1$ and all row sums positive.
\end{enumerate}
For a discussion of $\MM$-matrices in the context of ultrametrics and potentials of finite Markov chains, see \cite{DMS}.

We define the \emph{multivariate characteristic polynomial} of $\mathrm{A}$ by the equation
\[
\mathrm{p}_\mathrm{A}(w_0,w_1,\ldots,w_n)=\det\Big(w_0\mathrm{I}_n+\diag(w_1,\ldots, w_n)\mathrm{A}\Big).
\]
In \cite[Theorem 4]{Holtz}, Holtz proved that the coefficients  of the characteristic polynomial of an $\MM$-matrix form an ultra log-concave sequence. 
We will strengthen this result and prove that the multivariate characteristic polynomial of an $\MM$-matrix is Lorentzian.

\begin{theorem}\label{M-matrix}
If $\mathrm{A}$ is an $\MM$-matrix, then $\mathrm{p}_\mathrm{A}$ is a Lorentzian polynomial.
\end{theorem}

Using Example \ref{ExampleBivariate}, we may recover the theorem of Holtz  by setting $w_1=\cdots=w_n$.

\begin{corollary}
If $\mathrm{A}$ is an $\MM$-matrix, then the support of $\mathrm{p}_\mathrm{A}$  is $\MM$-convex.
\end{corollary}

Since every $\MM$-matrix is a limit of nonsingular $\MM$-matrices,
 it is enough to prove Theorem \ref{M-matrix} for nonsingular $\MM$-matrices.

\begin{lemma}\label{NonsingularSupport}
If $\mathrm{A}$ is a nonsingular $\MM$-matrix,  the support of $\mathrm{p}_\mathrm{A}$ is $\MM$-convex.
\end{lemma}

\begin{proof}
It is enough to prove that the support of
$\mathrm{p}^\natural_\mathrm{A}$ is $\MM^\natural$-convex, where
\[
\mathrm{p}^\natural_\mathrm{A}(w_1,\ldots,w_n)=\mathrm{p}_\mathrm{A}(1,w_1,\ldots,w_n).
\]
If $\mathrm{A}$ is a nonsingular $\MM$-matrix, then all the principal minors of $\mathrm{A}$ are positive, and hence
\[
\text{supp} (\mathrm{p}_\mathrm{A}^\natural)=\{0,1\}^n.
\]
It is straightforward to verify the augmentation property in Lemma \ref{polyind} for $\{0,1\}^n$.
\end{proof}

We prepare the proof of Theorem \ref{M-matrix} with a proposition on doubly sub-stochastic matrices.
Recall that an $n \times n$ matrix $\mathrm{B}=(b_{ij})$ with nonnegative entries is said to be \emph{doubly sub-stochastic} if 
\[
\text{$\sum_{j=1}^n b_{ij} \le 1$ for every $i$} \quad \text{and} \quad  \text{$\sum_{i=1}^n b_{ij} \le 1$ for every $j$}.
\]
A \emph{partial permutation matrix} is a zero-one matrix with at most one nonzero entry in each row and column.
We use Mirsky's analog of the Birkhoff-von Neumann theorem for doubly sub-stochastic matrices \cite[Theorem 3.2.6]{HJ}: The set of $n \times n$ doubly sub-stochastic matrix  is equal to the convex hull of the $n \times n$ partial permutation matrices.

  %and define  $n \times n$ matrices $\mathrm{M}_n$ and $\mathrm{N}_n$ by
%$1_n$ for  the $n \times 1$ matrix all of whose entries are $1$.

\begin{lemma}\label{TypeA}
For $n \ge 2$, define $n \times n$ matrices $\mathrm{M}_n$ and $\mathrm{N}_n$ by
\[
\small
\mathrm{M}_{n}=
\left(\begin{array}{cccccc}
2&1&0& \cdots &0 & 1\\
1&2&1& \cdots & 0 & 0\\
0&1&2 & \cdots &0& 0\\
\vdots &\vdots&\vdots&\ddots&\vdots&\vdots\\
0&0&0&\cdots &2&1\\
1&0&0&\cdots&1&2\\
\end{array}
\right),
\quad
\mathrm{N}_{n}=
\left(\begin{array}{cccccc}
2&1&0& \cdots &0 & 0\\
1&2&1& \cdots & 0 & 0\\
0&1&2 & \cdots &0& 0\\
\vdots &\vdots&\vdots&\ddots&\vdots&\vdots\\
0&0&0&\cdots &2&1\\
0&0&0&\cdots&1&2\\
\end{array}
\right).
\]
Then the matrices $\mathrm{M}_n-\frac{2}{n}\mathrm{J}_n$ and $\mathrm{N}_n-\frac{2}{n}\mathrm{J}_n$ are positive semidefinite.
Equivalently, 
\[
\underline{\mathrm{M}}_{n+1}\coloneq\left(\begin{array}{cccc}
\mathrm{M}_n&1_{n}\\
%\hline
1_n^T&\frac{n}{2}
\end{array}\right),
\quad
\underline{\mathrm{N}}_{n+1}\coloneq\left(\begin{array}{cccc}
\mathrm{N}_n&1_{n}\\
%\hline
1_n^T&\frac{n}{2}
\end{array}\right)
\]
%\[
%\small
%\underline{\mathrm{M}}_{n+1}\coloneq \left(\begin{array}{ccccccc}
%2&1&0& \cdots &0 & 1&1\\
%1&2&1& \cdots & 0 & 0&1\\
%0&1&2 & \cdots &0& 0&1\\
%\vdots &\vdots&\vdots&\ddots&\vdots&\vdots&\vdots\\
%0&0&0&\cdots &2&1&1\\
%1&0&0&\cdots&1&2&1\\
%\hline
%1&1&1& \cdots&1 &1& \frac{n}{2}
%\end{array}
%\right),
%\quad
%\underline{\mathrm{N}}_{n+1}\coloneq \left(\begin{array}{ccccccc}
%2&1&0& \cdots &0 & 0&1\\
%1&2&1& \cdots & 0 & 0&1\\
%0&1&2 & \cdots &0& 0&1\\
%\vdots &\vdots&\vdots&\ddots&\vdots&\vdots&\vdots\\
%0&0&0&\cdots &2&1&1\\
%0&0&0&\cdots&1&2&1\\
%\hline
%1&1&1& \cdots&1 &1& \frac{n}{2}
%\end{array}
%\right)
%\]
are positive semidefinite.
\end{lemma}

\begin{proof}
We define symmetric matrices $\mathrm{L}_{n+1}$ and $\mathrm{K}_{n+1}$ by
\[
\small
\mathrm{L}_{n+1}=
\left(\begin{array}{ccccccc}
1&1&0& \cdots &0 & 0&\frac{1}{2}\\
1&2&1& \cdots & 0 & 0&1\\
0&1&2 & \cdots &0& 0&1\\
\vdots &\vdots&\vdots&\ddots&\vdots&\vdots&\vdots\\
0&0&0&\cdots &2&1&1\\
0&0&0&\cdots&1&1&\frac{1}{2}\\
%\hline
\frac{1}{2}&1&1& \cdots&1 &\frac{1}{2}& \frac{n}{2}
\end{array}
\right),
\quad
\mathrm{K}_{n+1}=
\left(\begin{array}{ccccccc}
1&1&0& \cdots &0 & 0&\frac{1}{2}\\
1&2&1& \cdots & 0 & 0&1\\
0&1&2 & \cdots &0& 0&1\\
\vdots &\vdots&\vdots&\ddots&\vdots&\vdots&\vdots\\
0&0&0&\cdots &2&1&1\\
0&0&0&\cdots&1&2&1\\
%\hline
\frac{1}{2}&1&1& \cdots&1 &1 & \frac{n}{2}
\end{array}
\right).
\]
As before, the subscript indicates the size of the matrix.
We show, by induction on $n$, that the matrices
$\mathrm{L}_{n+1}$ and $\mathrm{K}_{n+1}$ are positive semidefinite.
It is straightforward to check that $\mathrm{L}_{3}$ and $\mathrm{K}_{3}$ are positive semidefinite.
Perform the symmetric row and column elimination of $\mathrm{L}_{n+1}$ and $\mathrm{K}_{n+1}$ based on their $1\times 1$ entries,
and notice that %$\mathrm{L}_{n+1}$  and  $\mathrm{K}_{n+1}$ are, respectively, congruent to the matrices
\[
\small
\mathrm{L}_{n+1}\simeq
\left(\begin{array}{ccccccc}
1&0&0& \cdots &0 & 0&0\\
0&1&1& \cdots & 0 & 0&\frac{1}{2}\\
0&1&2 & \cdots &0& 0&1\\
\vdots &\vdots&\vdots&\ddots&\vdots&\vdots&\vdots\\
0&0&0&\cdots &2&1&1\\
0&0&0&\cdots&1&1&\frac{1}{2}\\
%\hline
0&\frac{1}{2}&1& \cdots&1 &\frac{1}{2}& \frac{n}{2}-\frac{1}{4}
\end{array}
\right),
\quad
\mathrm{K}_{n+1}\simeq
\left(\begin{array}{ccccccc}
1&0&0& \cdots &0 & 0&0\\
0&1&1& \cdots & 0 & 0&\frac{1}{2}\\
0&1&2 & \cdots &0& 0&1\\
\vdots &\vdots&\vdots&\ddots&\vdots&\vdots&\vdots\\
0&0&0&\cdots &2&1&1\\
0&0&0&\cdots&1&1&1\\
%\hline
0&\frac{1}{2}&1& \cdots&1 &1& \frac{n}{2}-\frac{1}{4}
\end{array}
\right),
\]
where the symbol $\simeq$ stands for the congruence relation for symmetric matrices.
Since %$\left(\begin{array}{cc}2&0\\  0 & \mathrm{L}_n\end{array}\right)$ and $\left(\begin{array}{cc}\frac{1}{2}&-\frac{1}{2}\\  -\frac{1}{2} & \frac{1}{2}\end{array}\right)$ are
$\mathrm{L}_n$ is 
 positive semidefinite, 
$\mathrm{L}_{n+1}$ is congruent to the sum of positive semidefinite matrices,
and hence $\mathrm{L}_{n+1}$ is positive semidefinite.
Similarly, since %$\left(\begin{array}{cc}2&0\\  0 & \mathrm{L}_n\end{array}\right)$ and $\left(\begin{array}{cc}\frac{1}{2}&-\frac{1}{2}\\  -\frac{1}{2} & \frac{1}{2}\end{array}\right)$ are
$\mathrm{K}_n$ is 
 positive semidefinite, 
$\mathrm{K}_{n+1}$ is congruent to the sum of positive semidefinite matrices,
and hence $\mathrm{K}_{n+1}$ is positive semidefinite.

We now prove  that the symmetric matrices $\underline{\mathrm{M}}_{n+1}$ and $\underline{\mathrm{N}}_{n+1}$ are positive semidefinite.
Perform the symmetric row and column elimination of $\underline{\mathrm{M}}_{n+1}$ and $\underline{\mathrm{N}}_{n+1}$ based on their $1\times 1$ entries,
and notice that %$\underline{\mathrm{M}}_{n+1}$ and $\underline{\mathrm{N}}_{n+1}$ are congruent to the matrices
\[
\small
\underline{\mathrm{M}}_{n+1}\simeq
\left(\begin{array}{rrrrrrc}
2&0&0& \cdots &0 & 0&0\\ 
0&\frac{3}{2}&1& \cdots & 0 & -\frac{1}{2}&\frac{1}{2}\\
0&1&2 & \cdots &0& 0&1\\
\vdots &\vdots&\vdots&\ddots&\vdots&\vdots&\vdots\\
0&0&0&\cdots &2&1&1\\
0&-\frac{1}{2}&0&\cdots&1&\frac{3}{2}&\frac{1}{2}\\ 
%\hline
0&\frac{1}{2}&1& \cdots&1 &\frac{1}{2}& \frac{n-1}{2}
\end{array}
\right),
\quad
\underline{\mathrm{N}}_{n+1}\simeq
\left(\begin{array}{ccccccc}
2&0&0& \cdots &0 & 0&0\\ 
0&\frac{3}{2}&1& \cdots & 0 & 0&\frac{1}{2}\\
0&1&2 & \cdots &0& 0&1\\
\vdots &\vdots&\vdots&\ddots&\vdots&\vdots&\vdots\\
0&0&0&\cdots &2&1&1\\
0&0&0&\cdots&1&2&1 \\ 
%\hline
0&\frac{1}{2}&1& \cdots&1 &1 & \frac{n-1}{2}
\end{array}
\right).
\]
Since %$\left(\begin{array}{cc}2&0\\  0 & \mathrm{L}_n\end{array}\right)$ and $\left(\begin{array}{cc}\frac{1}{2}&-\frac{1}{2}\\  -\frac{1}{2} & \frac{1}{2}\end{array}\right)$ are
$\mathrm{L}_n$ is 
 positive semidefinite, 
$\underline{\mathrm{M}}_{n+1}$ is congruent to the sum of two positive semidefinite matrices,
and hence $\underline{\mathrm{M}}_{n+1}$ is positive semidefinite.
Similarly, since %$\left(\begin{array}{cc}2&0\\  0 & \mathrm{L}_n\end{array}\right)$ and $\left(\begin{array}{cc}\frac{1}{2}&-\frac{1}{2}\\  -\frac{1}{2} & \frac{1}{2}\end{array}\right)$ are
$\mathrm{K}_n$ is 
 positive semidefinite, 
$\underline{\mathrm{N}}_{n+1}$ is congruent to the sum of two positive semidefinite matrices,
and hence $\underline{\mathrm{N}}_{n+1}$ is positive semidefinite.
\end{proof}

\begin{proposition}\label{DoublySubstochastic}
If $\mathrm{B}$ is an $n \times n$ doubly sub-stochastic matrix, then $2\mathrm{I}_{n}+\mathrm{B}+\mathrm{B}^T-\frac{2}{n} \mathrm{J}_{n}$ is positive semidefinite.
\end{proposition}

\begin{proof}
Let $\mathrm{C}_n$ be the symmetric matrix $2\mathrm{I}_n+\mathrm{B}+\mathrm{B}^T$, and let $\underline{\mathrm{C}}_{n+1}$ be the symmetric matrix
\[
\underline{\mathrm{C}}_{n+1}\coloneq\left(\begin{array}{cccc}
\mathrm{C}_n&1_{n}\\
%\hline
1_n^T&\frac{n}{2}
\end{array}\right).
\]
It is enough to prove that $\underline{\mathrm{C}}_{n+1}$ is positive semidefinite.
Since the convex hull of the  partial permutation matrices is the set of doubly sub-stochastic matrix,
the proof reduces to the case when $\mathrm{B}$ is a partial permutation matrix.
We use the following extension of the cycle decomposition for partial permutations:
For any partial permutation matrix $\mathrm{B}$, there is a permutation matrix $\mathrm{P}$ such that $\mathrm{P}\mathrm{B}\mathrm{P}^T$ is a block diagonal matrix,
where each block diagonal is either zero, identity,
\[
\small
\left(\begin{array}{cccccc}
0&0&0& \cdots &0 & 1\\
1&0&0& \cdots & 0 & 0\\
0&1&0 & \cdots &0& 0\\
\vdots &\vdots&\vdots&\ddots&\vdots&\vdots\\
0&0&0&\cdots &0&0\\
0&0&0&\cdots&1&0\\
\end{array}
\right)
\quad \text{or} \quad
\left(\begin{array}{cccccc}
0&0&0& \cdots &0 & 0\\
1&0&0& \cdots & 0 & 0\\
0&1&0 & \cdots &0& 0\\
\vdots &\vdots&\vdots&\ddots&\vdots&\vdots\\
0&0&0&\cdots &0&0\\
0&0&0&\cdots&1&0\\
\end{array}
\right).
\]
Using the cyclic decomposition for $\mathrm{B}$, 
we can express the  matrix $\underline{\mathrm{C}}_{n+1}$ as a sum, where each summand is positive semidefinite by
 Lemma \ref{TypeA}.
\end{proof}

The remaining part of the proof of Theorem \ref{M-matrix} parallels that of Theorem \ref{mainpotts}.

\begin{proof}[Proof of Theorem \ref{M-matrix}]
Since every $\MM$-matrix is a limit of nonsingular $\MM$-matrices,
we may suppose that $\mathrm{A}$ is a nonsingular $\MM$-matrix.
 For $k=0,1,\ldots,n$, we set
\[
\mathrm{p}^k_\mathrm{A}(w) =\sum_{\alpha \in { n\brack k}} \mathrm{A}_\alpha w^\alpha, \quad w=(w_1,\ldots,w_n),
\]
where $\mathrm{A}_\alpha$ is the principal minor of $\mathrm{A}$ corresponding to $\alpha$,
so that
 \[
 \mathrm{p}_\mathrm{A}(w_0,w_1,\ldots,w_n)=\sum_{k=0}^n  \mathrm{p}^k_\mathrm{A}(w) w_0^{n-k}.
 \]
Lemma \ref{NonsingularSupport} shows that the support of $\mathrm{p}_\mathrm{A}$ is $\MM$-convex.
Therefore, by Theorem \ref{chars}, it is enough to prove that $\partial_i (\mathrm{p}_\mathrm{A})$ is Lorentzian for $i=0,1,\ldots,n$.
We prove this statement by induction on $n$.
The assertion is clear when $n=1$, so suppose $n \ge 2$.

When $i \neq 0$, write $\mathrm{B}$ for the inverse of $\mathrm{A}$ and $\mathrm{B}/i$ for the matrix obtained from $\mathrm{B}$ by deleting the $i$-th row and column. We observe that the $i$-th partial derivative of $p_{\mathrm{A}}$ is given by
\begin{align*}
\partial_i \mathrm{p}_\mathrm{A}(w_0,w_1,\ldots,w_n) 
&= \partial_i \det\Big(w_0\mathrm{I}_n+\diag(w_1,\ldots, w_n)\mathrm{A}\Big)\\
&=\det(\mathrm{A}) \ \partial_i \det\Big(w_0\mathrm{B}+\diag(w_1,\ldots, w_n)\Big)\\
&=\det(\mathrm{A})  \det\Big(w_0 (\mathrm{B}/i)+\diag(w_1,\ldots,\hat w_i, \ldots, w_n)\Big)\\
&=\det(\mathrm{A})    \det  (\mathrm{B}/i)  \det\Big(w_0  \mathrm{I}_{n-1}+\diag(w_1,\ldots,\hat w_i, \ldots, w_n)(\mathrm{B}/i)^{-1} \Big).
%&=\det(\mathrm{A})  \det  (\mathrm{B}/i)  \det\Big( w_0 \mathrm{I}_{n-1}+\diag(w_1,\ldots, w_n)(\mathrm{B}/i)^{-1} \Big)
\end{align*}
By \cite[Theorem 3.1]{Markham}, the matrix $B/i$ has positive determinant and its inverse  is an $M$-matrix, so the induction hypotheis applies to the right-hand side.
%When $i \neq 0$, we have
%\[
%\partial_i (\mathrm{p}_\mathrm{A})=\mathrm{p}_{\mathrm{A} / i},
%\]
%where $\mathrm{A} / i$ is the $\MM$-matrix obtained from $\mathrm{A}$ by deleting the $i$-th row and column.
Thus, to conclude, it  is enough to prove that the following quadratic form is stable:
\[
\frac{n!}{2}w_0^2+(n-1)! \hspace{0.5mm} \mathrm{p}^1_\mathrm{A}(w) w_0 +(n-2)! \hspace{0.5mm}\mathrm{p}^2_\mathrm{A}(w).
\]
%Recall that a homogeneous polynomial $f$ with nonnegative coefficients in $n+1$ variables is stable if and only if the univariate polynomial
% $f(xu-v)$ has only real zeros for all $v \in \mathbb{R}^{n+1}$ for some $u \in \mathbb{R}^{n+1}_{\ge 0}$ satisfying $f(u)>0$.
%Therefore,
As in the proof of Theorem \ref{mainpotts}  it suffices to show that the discriminant of the displayed quadratic form with respect to $w_0$ is nonnegative:
\[
\mathrm{p}^1_\mathrm{A}(w)^2 \ge  \frac{2n}{n-1} \mathrm{p}^2_\mathrm{A}(w) \ \ \text{for all $w \in \mathbb{R}^n$.}
\]
In terms of the entries of $\mathrm{A}$,
the displayed inequality is equivalent to the statement that the matrix $\Big(a_{ij}a_{ji}-\frac{1}{n}a_{ii}a_{jj}\Big)$ is positive semidefinite.
According to the $29$-th characterization of nonsingular $\MM$-matrices in \cite[Chapter 6]{BP},
there are positive diagonal matrices $D$ and $D'$ such that $D A D'$ has all diagonal entries $1$ and all row sums positive.
Therefore, we may suppose that $A$ has all diagonal entries $1$ and all the row sums of $\mathrm{A}$ are positive.
Under this assumption,
\[
\Big(a_{ij}a_{ji}-\frac{1}{n}a_{ii}a_{jj}\Big)=\mathrm{I}_n-\mathrm{B}-\frac{1}{n}\mathrm{J}_n,
\]
where $-\mathrm{B}$ is a symmetric doubly sub-stochastic matrix all of whose diagonal entries are zero.
The conclusion follows from Proposition \ref{DoublySubstochastic}.
\end{proof}

\subsection{Lorentzian  probability measures}\label{secLM}

There are numerous important examples of negatively dependent ``repelling'' random variables in probability theory, combinatorics, stochastic processes, and statistical mechanics.
See, for example, \cite{Pem}. 
A theory of negative dependence for strongly Rayleigh measures was developed in \cite{BBL}, but 
 the theory is too restrictive for several applications. 
Here we introduce a  broader class of discrete probability measures using the Lorentzian property.

A \emph{discrete probability measure} $\mu$  on $\{0,1\}^n$  is a probability measure on $\{0,1\}^n$ such that all subsets of $\{0,1\}^n$ are measurable. 
%We may view  $\mu$ as a function $\{0,1\}^n \to \mathbb{R}_{\ge 0}$ satisfying
%\[
%\sum_{S \in \{0,1\}^n} \mu(S)=1.
%\]
The \emph{partition function}  of $\mu$ is the polynomial
\[
\mathrm{Z}_\mu(w)= \sum_{S \subseteq [n]} \mu\big(\{S\}\big) \prod_{i \in S}w_i.
\]
The following notions capture various aspects of negative dependence: 
\begin{itemize}[--]\itemsep 5pt
\item The measure $\mu$ is  \emph{pairwise negatively correlated} (\text{PNC}) if for all distinct $i$ and $j$ in $[n]$, 
\[
\mu(\mathcal{E}_i \cap \mathcal{E}_j ) \leq \mu(\mathcal{E}_i)\mu(\mathcal{E}_j), %\ \ \text{where $\mathcal{E}_i = \big\{ S \subseteq [n] \mid i \in S\big\}$}.
\]
where $\mathcal{E}_i$ is the collection of all subsets of $[n]$ containing $i$. %= \{ S \subseteq [n] \mid i \in S\}$. 
%\item The measure $\mu$ is \emph{negatively associated} ($\mathrm{NA}$) if for every $A \subseteq [n]$, 
%\[
%\int fg \ d \mu \leq \int f \ d \mu \int g \ d \mu
%\]
%holds for every pair of increasing functions $f: \{0,1\}^A \to \mathbb{R}$ and $g:\{0,1\}^{[n] \setminus A} \to \mathbb{R}$.
%for all increasing functions $f,g : \{0,1\}^n \rightarrow \mathbb{R}$,
%provided that $f$ depends only on a subset $A$ of the $n$ variables and $g$ depends only on a subset disjoint from $A$.
\item The measure $\mu$ is \emph{ultra log-concave} (\text{ULC}) if for every positive integer $k<n$,
\[
\frac {\mu\Big( {n \brack k}\Big)^2}{{\binom n k}^2} \geq  \frac {\mu\Big( {n \brack k-1}\Big)}{\binom n {k-1}} \frac {\mu\Big( {n \brack k+1}\Big)}{\binom n {k+1}}.
\]
%where  $r_k = \mu\Big( {n \brack k}\Big)$. 
\item The measure $\mu$ is \emph{strongly Rayleigh} if for all distinct $i$ and $j$ in $[n]$,%its partition function $\mathrm{Z}_\mu$ satisfies %i.e., if $\mathrm{Z}_\mu$  is non-zero whenever all variables have positive imaginary parts. 
\[
 \mathrm{Z}_\mu(w) \hspace{0.5mm}\partial_i\partial_j   \mathrm{Z}_\mu(w)  \leq \partial_i  \mathrm{Z}_\mu(w)  \hspace{0.5mm}\partial_j   \mathrm{Z}_\mu(w) \ \ \text{for all $w \in \RR^n$}. 
\]
\end{itemize}
Let $\mathrm{P}$ be a property of discrete probability measures.
We say that $\mu$ has property $\underline{\mathrm{P}}$ if, for every $x\in \RR_{>0}^n$, the discrete probability measure on $\{0,1\}^n$ with  the partition function
\[
\mathrm{Z}_\mu(x_1w_1,\ldots, x_nw_n)/\mathrm{Z}_\mu(x_1,\ldots, x_n)
\]
has property $\mathrm{P}$.
The new discrete probability  measure is said to be obtained from $\mu$ by applying the \emph{external field} $x\in \RR_{>0}^n$.  %Then $\mathrm{P}_+$ is said to be closed under \emph{external fields}. Hence 
For example, the property $\underline{\mathrm{PNC}}$ for $\mu$ is equivalent to the $1$-Rayleigh property
\[
 \mathrm{Z}_\mu(w) \hspace{0.5mm}\partial_i\partial_j   \mathrm{Z}_\mu(w)  \leq \partial_i  \mathrm{Z}_\mu(w)  \hspace{0.5mm}\partial_j   \mathrm{Z}_\mu(w) \ \ \text{for all distinct $i$, $j$ in $[n]$ and all $w \in \RR_{>0}^n$}. 
\]
More generally, for a positive real number $c$, we say that $\mu$ is $c$-Rayleigh if 
\[
 \mathrm{Z}_\mu(w) \hspace{0.5mm}\partial_i\partial_j   \mathrm{Z}_\mu(w)  \leq c \hspace{0.5mm} \partial_i  \mathrm{Z}_\mu(w)  \hspace{0.5mm}\partial_j   \mathrm{Z}_\mu(w) \ \ \text{for all distinct $i$, $j$ in $[n]$ and all $w \in \RR_{>0}^n$}. 
\]
%The property $\mathrm{PNC}_+$ is often called the \emph{Rayleigh property} because of its connections to the Rayleigh monotonicity property in electrical networks \cite{Wagner08}.

%In Section~\ref{secLM} we introduce a new negative dependence property defined by Lorentzian polynomials. 
%\begin{definition}
%We say that the measure $\mu$ is \emph{Hodge--Riemann} ($\mathrm{L}$) if the Hessian of the homogenization 
%$$
%t^d\mathrm{Z}_\mu(w/t), \ \ \ \ \ d= \deg \mathrm{Z}_\mu,
%$$
%has exactly one positive eigenvalue for each $(w,t) \in \RR_{>0}^{|E|+1}$. 
%\end{definition}
%We shall see that all strong Rayleigh measures are Lorentzian, and that the Lorentzian property enjoys many virtues of negative dependence. 

\begin{definition}
A discrete probability measure $\mu$ on $\{0,1\}^n$ is  \emph{Lorentzian} if the homogenization of the partition function
$
w_0^n\mathrm{Z}_\mu(w_1/w_0,\ldots, w_n/w_0)
$
 is a Lorentzian polynomial. %in variables $w_0,w_1,\ldots,w_n$.
\end{definition}

For example, if $\mathrm{A}$ is an  $\mathrm{M}$-matrix of size $n$, the probability measure on $\{0,1\}^n$ given by
\[
\mu\big(\{S\}\big) \propto \Big(\text{the principal minor of $\mathrm{A}$ corresponding to $S$}\Big), \quad S \subseteq [n],
\]
is Lorentzian by Theorem \ref{M-matrix}.
Results from the previous sections reveal basic features of Lorentzian measures, some of which may be interpreted as negative dependence properties.

%A number of results from the previous sections can be interpreted as negative dependence properties of Lorentzian measures.
%All strongly Rayleigh measures  are Lorentzian. 

%\begin{proposition}
%If $\mu_1$ and $\mu_2$ are Lorentzian, then $\mu_1 \times \mu_2$ is Lorentzian.
%\end{proposition}

%\begin{proof}
%The product of Lorentzian polynomials is Lorentzian by Corollary \ref{CorollaryProduct}.
%\end{proof}

\begin{proposition}\label{Propositon2-Rayleigh}
If $\mu$ is Lorentzian, then $\mu$ is $2$-Rayleigh. 
\end{proposition}

\begin{proof}
Lemma \ref{DeletionContraction} and Proposition \ref{HRRayleigh} %, and Theorem \ref{chars} 
show that   $\mathrm{Z}_\mu$ is a $2 \Big(1-\frac{1}{n}\Big)$-Rayleigh polynomial.
\end{proof}

\begin{proposition}
If $\mu$ is Lorentzian, then $\mu$ is $\underline{\mathrm{ULC}}$. 
\end{proposition}

\begin{proof}
Since any probability measure obtained from a Lorentzian probability measure by applying an external field is Lorentzian, it suffices to prove that $\mu$ is $\mathrm{ULC}$.
By Theorem \ref{flow},  the bivariate homogeneous polynomial $w_0^n\mathrm{Z}_\mu(w_1/w_0,\ldots, w_1/w_0)$ is Lorentzian.
Therefore, by Example \ref{ExampleBivariate}, its sequence of coefficients must be ultra log-concave.
\end{proof}

%The \emph{symmetric exclusion process} ($\mathrm{SEP}$) is one of the main models considered in interacting particle systems. It is a continuous time Markov chain which models particle that jump (symmetrically) between sites, where each sites may be occupied by at most one particle, see \cite{BBL,Liggett10,Wagner11} and the references therein. A problem that has attracted much attention is to find negative dependence properties that are preserved under $\mathrm{SEP}$. In \cite{BBL} it was proved that strong Rayleigh measures are preserved under $\mathrm{SEP}$. 

\begin{proposition}
The class of Lorentzian measures is preserved under the symmetric exclusion process.
\end{proposition}

\begin{proof}
The statement is Corollary \ref{partsym} for homogenized partition functions of Lorentzian probability measures.
\end{proof}

%In addition, in \cite[Section 2]{Pem} and \cite[Section 1]{BBL}, one finds the following standard methods of obtaining discrete probability measures from a given discrete probability measure:
%\[
%\text{Conditioning, Product, Relabelling, External Field, Symmetrization.}
%\]
%All the methods preserve the class of strongly Rayleigh measures \cite{BBL}, and hence, by Theorem \ref{stab-lor}, preserve the class of Lorentzian measures.
%\begin{remark}
%In \cite{Wagner11}, Wagner proved that strongly Rayleigh measures are preserved under the symmetric exclusion process also when particles are allowed to created and annihilated at sites. This is true also for Lorentzian measures by Theorem \ref{stab-lor}, since the corresponding operators on  partition functions preserve stability. 
%\end{remark}

\begin{proposition}
If $\mu$ is strongly Rayleigh, then $\mu$ is Lorentzian.
\end{proposition}

\begin{proof}
A multi-affine polynomial is stable if and only if it is strongly Rayleigh  \cite[Theorem 5.6]{Branden},
and a polynomial with nonnegative coefficients is stable if and only if its homogenization is stable \cite[Theorem 4.5]{BBL}.
By Proposition \ref{PropositionStableLorentzian}, homogeneous stable polynomials with nonnegative coefficients are Lorentzian.
\end{proof}

For a matroid $\mathrm{M}$ on $[n]$, we define probability measures $\mu_\mathrm{M}$ and $\nu_{\mathrm{M}}$ on $\{0,1\}^n$ by
\begin{align*}
\mu_\mathrm{M}&=\text{the uniform measure on $\{0,1\}^n$ concentrated on the independent sets of $\mathrm{M}$,}\\
\nu_\mathrm{M}&=\text{the uniform measure on $\{0,1\}^n$ concentrated on the bases of $\mathrm{M}$.}
\end{align*}

\begin{proposition}\label{PropositionMatroidMeasures}
For any matroid $\mathrm{M}$ on $[n]$, the measures $\mu_\mathrm{M}$ and $\nu_{\mathrm{M}}$ are Lorentzian.
\end{proposition}

\begin{proof}
Note that the homogenized partition function $f_\mathrm{M}$ of $\mu_\mathrm{M}$ satisfies
\[
f_\mathrm{M}(w_0,w_1,\ldots,w_n)=\lim_{q \to 0} \mathrm{Z}_{q,\mathrm{M}}(w_0,qw_1,\ldots,qw_n). 
\]
Since a limit of Lorentzian polynomials is Lorentzian, $\mu_\mathrm{M}$ is Lorentzian by Theorem \ref{mainpotts}.
%The homogenized partition function $f_\mathrm{M}$ of $\mu_\mathrm{M}$ is Lorentzian by Theorem \ref{mainpotts} and 
%\[
%f_\mathrm{M}(w_0,w_1,\ldots,w_n)=\lim_{q \to 0} \mathrm{Z}_{q,\mathrm{M}}(w_0,qw_1,\ldots,qw_n). 
%\]
The partition function of $\nu_\mathrm{M}$ is Lorentzian by Theorem \ref{charjump}. 
%Since the product of Lorentzian polynomials is Lorentzian, it follows that $\nu_\mathrm{M}$ is Lorentzian.
\end{proof}

Let $G$ be an arbitrary finite graph and let $i$ and $j$ be any distinct edges of $G$.
A conjecture of Kahn \cite{Kahn} and Grimmett--Winkler \cite{GW} states that,
if $F$ is a forest in $G$ chosen uniformly at random, then
\[
\text{Pr}(\text{$F$ contains $i$ and $j$}) \le \text{Pr}(\text{$F$ contains $i$})\hspace{0.5mm} \text{Pr}(\text{$F$ contains  $j$}).
\]
The conjecture is equivalent to the statement that $\mu_\mathrm{M}$ is $1$-Rayleigh for any graphic matroid $\mathrm{M}$.
Propositions \ref{Propositon2-Rayleigh} and \ref{PropositionMatroidMeasures} show that $\mu_\mathrm{M}$ is $2$-Rayleigh for any matroid $\mathrm{M}$.

\end{document}